\documentclass[10pt]{article}
\usepackage[margin=1in]{geometry}

\usepackage{listings}
\usepackage{geometry}

\usepackage{amsmath,amsthm,amsfonts,amssymb,amscd}
\usepackage[colorlinks=true, pdfstartview=FitV, linkcolor=blue,citecolor=blue, urlcolor=blue]{hyperref}
\usepackage[usenames,dvipsnames]{color}

\usepackage{mathtools}
\usepackage{mathabx}
\usepackage{float}
\usepackage{makeidx}
\usepackage{stmaryrd}
\usepackage{mathrsfs}
\usepackage{emptypage}
\usepackage{xfrac}
\usepackage{bbm}
\usepackage{comment} % to put large comments
\usepackage{xcolor} % to use colors in the text
\usepackage{graphicx}
\usepackage{subcaption}
\usepackage{tikz}
\usepackage{shuffle} % Write shuffle product of multipolylogs
\usepackage{bm}

\usepackage[backend=biber,url=false,eprint=false,maxbibnames=99]{biblatex}
\newcommand{\addjoana}[1]{#1}
\newcommand{\addjoel}[1]{#1}

\newcommand{\codenumber}[1]{#1}
\newcommand{\resultnumber}[1]{\textcolor{magenta}{#1}}

\DeclareMathOperator{\real}{Re}
\DeclareMathOperator{\imag}{Im}
\DeclareMathOperator{\midint}{mid}

\newtheorem{theorem}{Theorem}[section]
\newtheorem{lemma}[theorem]{Lemma}
\newtheorem{corollary}[theorem]{Corollary}

\newtheorem{proposition}[theorem]{Proposition}
\newtheorem{definition}[theorem]{Definition}
\newtheorem{remark}[theorem]{Remark}

\graphicspath{{figures/}}

\newcommand{\Li}{\mathrm{Li}}

\newcommand{\inter}[1]{\bm{#1}}

\title{Monotonicity of the first Dirichlet eigenvalue of regular polygons}
\date{\today}
 \author{Joel Dahne, Javier G\'omez-Serrano, Joana Pech-Alberich}

\begin{document}

\maketitle
\begin{abstract}
In this paper we prove that the first Dirichlet eigenvalue $\lambda_1^N$ of an $N$-sided regular polygon of fixed area is a monotonically decreasing function of $N$ for all $N \geq 3$, as well as the monotonicity of the quotients $\displaystyle \frac{\lambda_1^{N}}{\lambda_1^{N+1}}$. This settles a conjecture of Antunes--Freitas from 2006~\cite[Conjecture 3.2]{Antunes-Freitas:new-bounds-dirichlet-eigenvalue}.
\end{abstract}

\section{Introduction}

\subsection{Historical background}

Consider a bounded planar domain \(\Omega\) and let
\(\lambda_{k} = \lambda_{k}(\Omega)\) be the eigenvalues of the Dirichlet
Laplacian, namely the numbers
\(0 < \lambda_{1} < \lambda_{2} \leq \lambda_{3} \leq \dots\) that satisfy
\begin{equation*}
  \begin{split}
    -\Delta u_{k} & = \lambda_{k} u_{k}  \text{ in } \Omega,\\
    u_{k} & = 0  \text{ on } \partial \Omega.
  \end{split}
\end{equation*}
In this paper we are interested in the behavior of
\(\lambda_{1}(\mathcal{P}_{N})\) as a function of \(N\), where
\(\mathcal{P}_{N}\) denotes the regular \(N\)-sided polygon of area
\(\pi\). More precisely, we establish the following monotonicity
result.

\begin{theorem}\label{main_thm}
  Let \(\mathcal{P}_{N}\) be the regular \(N\)-sided polygon of area
  \(\pi\), with \(N \geq 3\). Then
  \begin{equation*}
    \lambda_{1}(\mathcal{P}_{3}) > \lambda_{1}(\mathcal{P}_{4}) > \dots > \lambda_{1}(\mathcal{P}_{N}) > \lambda_{1}(\mathcal{P}_{N + 1}) > \dots > \lambda_{1}(\mathbb{D}),
  \end{equation*}
  where \(\mathbb{D}\) is the unit disk. That is, the first Dirichlet
  eigenvalues are monotonically decreasing in \(N\). Furthermore, let
  \(q_{N} =
  \frac{\lambda_{1}(\mathcal{P}_{N})}{\lambda_{1}(\mathcal{P}_{N + 1})}\). Then
  it also holds that
  \begin{equation*}
    q_{3} > q_{4} > \dots > q_{N} > q_{N + 1} > \dots > 1.
  \end{equation*}
\end{theorem}

Theorem~\ref{main_thm} settles a conjecture of Antunes--Freitas from
2006~\cite[Conjecture
3.2]{Antunes-Freitas:new-bounds-dirichlet-eigenvalue}, where the
authors gave numerical evidence supporting it. The conjecture has
gained considerable attention recently, appearing in several surveys
such as the one by
Henrot~\cite[Conjecture 6.10]{Henrot:book-shape-optimization} and in
lists of open problems at AIM~\cite[Question
2.3.3]{Aim:open-problems-shape-optimization}.

Fixing the area of \(\mathcal{P}_{N}\) means that the monotonicity is
scale-invariant since
\(\lambda_{1}(t\Omega) = t^{-2} \lambda_{1}(\Omega)\). This is the same
normalization as in the Faber--Krahn inequality, which states that the
first Dirichlet eigenvalue of a bounded domain is greater than or
equal to the first Dirichlet eigenvalue of the disk with the same
area. Applied to \(\mathcal{P}_{N}\), this gives
\(\lambda_{1}(\mathcal{P}_{N}) \geq \lambda_{1}(\mathbb{D})\), and since
\(\mathcal{P}_{N}\) tends to the unit disk \(\mathbb{D}\) as
\(N \to \infty\), we get
\(\lambda_{1}(\mathcal{P}_{N}) \to \lambda_{1}(\mathbb{D})\) from above.
Theorem~\ref{main_thm} shows that this convergence is monotone.
Moreover, the monotonicity of the quotients \(q_{N}\) is equivalent to
the sequence \(\lambda_{1}(\mathcal{P}_{N})\) being log-convex, which
in particular implies convexity.

Under a different normalization, Nitsch~\cite[Theorem
2]{Nitsch:first-dirichlet-eigenvalue-regular-polygons} proved the
weaker bound
\begin{equation*}
  \lambda_{1}(\mathcal{P}_{N + 1}) <
  \lambda_{1}(\mathcal{P}_{N})
  \frac{\cos\left(\frac{\pi}{N}\right)}{\cos\left(\frac{\pi}{N + 1}\right)},
\end{equation*}
with the polygons normalized to have the same circumradius.
Renormalizing by area to make the inequality scale invariant yields
\begin{equation*}
  \lambda_{1}(\mathcal{P}_{N + 1})
  < \lambda_{1}(\mathcal{P}_{N})\left(
  \frac{\cos\left(\frac{\pi}{N}\right)}{\cos\left(\frac{\pi}{N + 1}\right)}
  \frac{\sin\left(\frac{2\pi}{N + 1}\right)}{\sin\left(\frac{2\pi}{N}\right)}
  \frac{N + 1}{N}
  \right)
  = \lambda_{1}(\mathcal{P}_{N})\left(
  \frac{\sin\left(\frac{\pi}{N + 1}\right)}{\sin\left(\frac{\pi}{N}\right)}
  \frac{N + 1}{N}
  \right),
\end{equation*}
where the factor in parentheses is greater than \(1\), so this bound
falls short of the first part of Theorem~\ref{main_thm}.

A related conjecture is due to P\'olya and
Szeg\"o~\cite{Polya-Szego:isoperimetric-inequalities-book}: among
\(N\)-sided polygons of the same area, the regular one minimizes the
first eigenvalue. In particular, the P\'olya--Szeg\"o conjecture
implies the first part of Theorem~\ref{main_thm}. P\'olya and Szeg\"o
themselves proved their conjecture for \(N = 3, 4\) using Steiner
symmetrization methods, which are doomed to fail for higher \(N\). In
recent years, some progress has been made. Bogosel--Bucur reduced the
problem to a finite (though very large)
calculation~\cite{Bogosel-Bucur:polygonal-faber-krahn} and then proved
the local minimality in the cases of \(N = 5, 6\)
in~\cite{Bogosel-Bucur:polygonal-faber-krahn-validated}.
Indrei~\cite{Indrei:first-eigenvalue-polygons} found, in the case of
\(N\) sufficiently large, a manifold on which the regular polygon
minimizes \(\lambda_{1}\).

Attempts to prove a weaker version of the Antunes--Freitas conjecture
have been centered on establishing it \emph{for sufficiently
  large \(N\)}, focusing on obtaining precise asymptotics of
\(\lambda_{1}\) (and more generally of \(\lambda_{k}\)) as
\(N \to \infty\). In particular, efforts have been made to produce a
series expansion of the form
\begin{equation*}
  \frac{\lambda_{k}(\mathcal{P}_{N})}{\lambda_{k}(\mathbb{D})} =
  1 + \frac{C_{k,1}}{N} + \frac{C_{k,2}}{N^{2}} + \frac{C_{k,3}}{N^{3}} + \cdots.
\end{equation*}
The first computation of any expansion dates to
Molinari~\cite{Molinari:ground-state-regular-polygonal-billiards},
where he numerically computed an expansion with respect to a hidden
parameter in the conformal mapping from the disk to the regular
polygon. Later,
Grinfeld--Strang~\cite{Grinfeld-Strang:laplacian-eigenvalues-polygon,Grinfeld-Strang:laplace-eigenvalues-regular-polygons-series}
computed the first four terms of such an expansion, \(C_{1,i}\) for
\(i = 1, \ldots, 4\), using the calculus of moving surfaces (CMS), in
which one continuously deforms the polygon into the disk and performs
shape-derivative computations. Boady, in his
thesis~\cite{Boady:calculus-moving-surfaces-thesis}, computed two more
terms using CMS: \(C_{1,5}\) and \(C_{1,6}\).
Jones~\cite{Jones:fundamental-laplacian-eigenvalue-regular-polygon-dirichlet}
proposed an expression for the next two coefficients \(C_{1,7}\) and
\(C_{1,8}\) using high-precision numerics and the LLL algorithm to
detect linear (integer) dependence relations between different
quantities. For other numerical calculations, as well as other
expansions (star-shaped domains, Neumann boundary conditions, etc.),
see~\cite{Berghaus-Jones-Monien-Radchenko:computation-eigenvalues-2d-shapes}.

A major breakthrough came with the work of
Berghaus--Georgiev--Monien--Radchenko~\cite{Berghaus-Georgiev-Monien-Radchenko:dirichlet-eigenvalues-regular-polygons},
where they provided explicit expressions for \(C_{1,k}\) with
\(k = 1, \ldots, 14\) and an algorithm to compute any \(C_{n,k}\), as
well as expansions of the eigenfunctions. These expansions involve
multiple zeta values. Moreover, they proved
Theorem~\ref{main_thm} for all sufficiently large \(N\). A limitation
of their method is that such an \(N\) is not quantified and could be
very large; in particular, it does not yield a statement valid for
every \(N\). One of the key contributions of this paper is to
establish effective bounds for the remainder term in the asymptotic
expansion, which allows us to give an explicit \(N_{0}\) such that the
expansion implies the monotonicity of \(\lambda_{1}(\mathcal{P}_{N})\)
and of \(q_{N}\) for \(N \geq N_{0}\).

\subsection{Proof of Theorem~\ref{main_thm}}
\label{sec:proof-theorem-main}

The proof of Theorem~\ref{main_thm} is split into two parts, one
treating small values of \(N\) and one treating large values. More
precisely, we let \(N_{0} = \codenumber{64}\) and prove the result for
\(N \geq N_{0}\) in Proposition~\ref{prop:N_large} and for
\(3 \leq N \leq N_{0}\) in Proposition~\ref{prop:N_small}. Together,
the two results immediately imply Theorem~\ref{main_thm}.

Both results rely on a posteriori stability estimates for eigenpairs of
the Dirichlet Laplacian. Obtaining such estimates is a fundamental
problem in spectral analysis and has been extensively
studied~\cite{Barnett-Hassell:quasi-ortogonality-dirichlet-eigenvalues,Behnke-Goerisch:inclusions-eigenvalues,Bramble-Payne:upper-lower-bounds-elliptic-equations,Fox-Henrici-Moler:approximations-bounds-eigenvalues,Kato:upper-lower-bounds-eigenvalues,Lehmann:optimale-eigenwerte,Moler-Payne:bounds-eigenvalues,Still:computable-bounds-eigenvalues,Weinberger:error-bounds-rayleigh-ritz}.
The overarching theme of these results is that if one can find an
approximate eigenpair \((u_{app}, \lambda_{app})\) satisfying
the equation up to some error, then one can show that there is a true
eigenpair \((u, \lambda)\) near this approximation. The precise
formulation that we will make use of here is the following result.

\begin{lemma}[\cite{Fox-Henrici-Moler:approximations-bounds-eigenvalues,
    Moler-Payne:bounds-eigenvalues}]\label{lem:FoxHenriciMoler}
  Let \(\Omega\subset\mathbb{R}^{2}\) be a bounded open set. Let
  \(\lambda_{app} \in \mathbb{R}\) and
  \(u_{app} \in \mathcal{C}^{2}(\Omega) \cap
  \mathcal{C}(\overline{\Omega})\), with \(u_{app} \not\equiv 0\), be
  an approximate eigenvalue and eigenfunction --- that is, they
  satisfy \(\Delta u_{app} + \lambda_{app} u_{app} = 0\) in \(\Omega\)
  but not necessarily \(u_{app} = 0\) on \(\partial\Omega\). Define
  \begin{equation}\label{mu:bound}
    \mu = \frac{\sqrt{|\Omega|}\|u_{app}\|_{L^{\infty}(\partial\Omega)}}{\|u_{app}\|_{L^{2}(\Omega)}},
  \end{equation}
  where \(|\Omega|\) is the area of the domain. Then there exists an
  eigenvalue \(\lambda\) of the Dirichlet Laplacian on \(\Omega\) such
  that
  \begin{equation}
    \label{eq:eigenvalue-bound}
    \frac{|\lambda_{app} - \lambda|}{\lambda} \leq \mu.
  \end{equation}
\end{lemma}

Given an approximate eigenpair, getting an upper bound for \(\mu\) in
the lemma requires computing an upper bound of
\(\|u_{app}\|_{L^{\infty}(\partial\Omega)}\) and a lower bound of
\(\|u_{app}\|_{L^{2}(\Omega)}\). The value of
\(\|u_{app}\|_{L^{2}(\Omega)}\) is in general relatively stable, and
the quality of the final enclosure is hence primarily determined by
how small the bound for \(\|u_{app}\|_{L^{\infty}(\partial\Omega)}\)
is. To be able to conclude the result of Theorem~\ref{main_thm}, the
resulting enclosures need to be tight enough to isolate the
value of \(\lambda_{1}(\mathcal{P}_{N})\) from
\(\lambda_{1}(\mathcal{P}_{N + 1})\) as well as \(q_{N}\) from
\(q_{N + 1}\).

In the large \(N\) case the approximate eigenpair is based on a finite
truncation of the asymptotic expansions established by
Berghaus--Georgiev--Monien--Radchenko~\cite{Berghaus-Georgiev-Monien-Radchenko:dirichlet-eigenvalues-regular-polygons}.
The number of terms kept in the expansion determines the asymptotic
behavior of the resulting enclosure. Asymptotically we have
\begin{equation*}
  \lambda_{1}(\mathcal{P}_{N + 1}) - \lambda_{1}(\mathcal{P}_{N}) = O(N^{-4})
\end{equation*}
and
\begin{equation*}
  q_{N + 1} - q_{N} = O(N^{-5}),
\end{equation*}
which means that we need to keep sufficiently many terms for the
defect \(\|u_{app}\|_{L^{\infty}(\partial\mathcal{P}_{N})}\) to be
smaller than \(N^{-5}\). As we will see in
Section~\ref{sec:the-large-N-case}, keeping terms up to order
\(N^{-5}\)
gives a defect of order \(O(N^{-6})\). See
Definition~\ref{def:u_app-lambda_app} for the precise construction of
the truncated approximation.

For small \(N\) the approximate eigenpair is instead based on the
Method of Particular Solutions (MPS), described in more detail in
Sections~\ref{sec:computational-aspects}
and~\ref{sec:the-small-N-case}. In this case separate approximations are
computed for each \(N\), which gives considerable flexibility and
allows the method to perform well for a large range of domains.

The enclosure from Lemma~\ref{lem:FoxHenriciMoler} does not directly
give any information about the index of the eigenfunction near
\(\lambda_{app}\). Verifying that the computed enclosures indeed are
for \(\lambda_{1}(\mathcal{P}_{N})\) and not for, say,
\(\lambda_{2}(\mathcal{P}_{N})\), therefore requires an additional
argument. In both the large \(N\) and small \(N\) cases, this is done
by verifying that the computed enclosure of the eigenvalue is strictly
smaller than \(\lambda_{2}(\mathbb{D}_{R})\) for some disk
\(\mathbb{D}_{R}\) of radius \(R\) containing \(\mathcal{P}_{N}\),
which by domain monotonicity implies that it is strictly smaller than
\(\lambda_{2}(\mathcal{P}_{N})\) and hence must be the first
eigenvalue.

\subsection{Computational aspects of the proof}
\label{sec:computational-aspects}

Parts of the proof of Theorem~\ref{main_thm} rely on
computer-assisted arguments based on interval arithmetic. Lately,
these types of techniques have become increasingly popular and
successful. The main paradigm is to replace floating-point numbers on
a computer by rigorous bounds, which are then propagated through every
operation that the computer performs, taking into account any error
made throughout the process. We refer to the
book~\cite{Tucker:validated-numerics-book} for an introduction to
rigorous numerics, and to the
survey~\cite{GomezSerrano:survey-cap-in-pde} and the recent
book~\cite{Nakao-Plum-Watanabe:cap-for-pde-book} for a more specific
treatment of computer-assisted proofs in PDE.

In the context of spectral problems, computer-assisted techniques have
been applied in several different settings. Examples include the works
by Nigam--Siudeja--Young~\cite{Nigam-Siudeja-Young:fem-schiffer-pentagon},
Filonov--Levitin--Polterovich--Sher~\cite{Filonov-Levitin-Polterovich-Sher:inequalities-polya-aharonov-bohm,Filonov-Levitin-Polterovich-Sher:polya-conjecture-balls},
Cao-Labora--Fern\'andez~\cite{CaoLabora-Fernandez:contractible-schiffer-counterexample},
Bogosel--Bucur~\cite{Bogosel-Bucur:polygonal-faber-krahn-validated}
and Endo--Liu~\cite{Endo2025,Endo2026}. The approach we take in the
small \(N\) case is closely related to the works by
G\'omez-Serrano--Orriols~\cite{GomezSerrano-Orriols:negative-hearing-shape-triangle},
Dahne--Salvy~\cite{Dahne-Salvy:enclosures-eigenvalues} and
Dahne--G\'omez-Serrano--Hou~\cite{Dahne-GomezSerrano-Hou:counterexample-payne}.

For large \(N\), computer-assisted bounds play a key role in
establishing several bounds related to the truncated asymptotic
expansion. The computation of these bounds is split into several smaller
lemmas, two of the main ones being
Lemmas~\ref{lemma:bounds_and_splitting:first_term_J0:upper-bounds-for-terms-from-g-function}
and~\ref{lemma:bounds-I_k_l}.

In the small \(N\) case the computer plays a role in both finding the
approximate eigenpair and establishing the necessary bounds to apply
Lemma~\ref{lem:FoxHenriciMoler}. For finding the approximation,
conventional numerical methods can be used; rigorous numerical methods
based on interval arithmetic are only needed in the a posteriori
computation of error bounds. Locating accurate approximate eigenpairs
is done through the MPS. This method was introduced by Fox, Henrici and
Moler~\cite{Fox-Henrici-Moler:approximations-bounds-eigenvalues} and has
since been adapted by many authors
(see~\cite{Antunes-Valtchev:mfs-corners-cracks,Read-Sneddon-Bode:series-method-mps,Golberg-Chen:mfs-survey,Fairweather-Karageorghis:mfs-survey,Betcke:generalized-svd-mps}
as a sample, and the thorough
review~\cite{Betcke-Trefethen:method-particular-solutions}). The main
idea is to consider a set of functions that solve the eigenvalue
problem without boundary conditions as a basis, and to write the
solution of the problem with boundary conditions as a linear
combination of them, solving for the coefficients that minimize the
error on the boundary. The choice of functions plays an important role
in the convergence rate of the method. Typically, the choices have
been rational
functions~\cite{Hochman-Leviatan-White:rational-function-laplacian} or
products of Bessel functions and trigonometric polynomials centered at
certain points; more recent examples include the use of so-called
\emph{lightning charges} introduced by Gopal and
Trefethen~\cite{Gopal-Trefethen:new-laplace-solver-pnas,Gopal-Trefethen:laplace-solver-detailed}.
For regular polygons it is natural to let one of the basis functions
be the eigenfunction of a disk. Even for modestly large values of
\(N\), this basis function alone gives a very good approximation of the
polygon's eigenfunction. As a consequence, relatively few basis
functions are required to get sufficiently good enclosures even for
larger values of \(N\). We stress that these methods produce accurate
approximations, but there is no explicit control of the error with
respect to the true solution. This is handled by the a posteriori
error bounds from Lemma~\ref{lem:FoxHenriciMoler}.

An alternative approach to the MPS would be to make use of finite
element methods, such as those developed by
Liu--Oishi~\cite{Liu-Oishi:verified-eigenvalues-laplacian-polygons,Liu:framework-verified-eigenvalues}.
However, the relatively tight tolerance required near \(N = N_{0}\)
would likely give a high computational cost for these methods compared
to the MPS. Since the MPS worked well, we did not investigate this
further.

Three aspects of computer-assisted proofs play a particularly important
role in this work:
\begin{description}
\item[Special functions] Several types of special functions are used,
  in particular for the truncated asymptotic expansion in the large
  \(N\) case. These include Bessel functions, multiple polylogarithms,
  gamma functions and hypergeometric functions. For most of these
  functions there are existing implementations that we make use of,
  see e.g.~\cite{Johansson2019}; for others we provide our own
  implementations, see Appendix~\ref{sec:polylogs}.
\item[Enclosing extrema] A key requirement of
  Lemma~\ref{lem:FoxHenriciMoler} is to bound the approximate
  eigenfunction on the boundary of the domain. For that reason we need
  to compute several bounds of the form \(\max_{x \in I} f(x)\) or
  \(\min_{x \in I} f(x)\) for various kinds of functions \(f\).
  Enclosing such extrema is a standard problem for computer-assisted
  proofs and we discuss the details of our approach in
  Appendix~\ref{sec:enclosing-extrema}.
\item[Integration] Several of the remainder terms for the asymptotic
  expansion involve integration. Rigorous computation of integrals is
  a well-studied problem, and we make use of the approach
  from~\cite{Johansson2018numerical}, with more details given in
  Appendix~\ref{sec:rigorous-integration}.
\end{description}

\subsection{Structure of the paper}

Section~\ref{sec:the-large-N-case} will be devoted to proving
Proposition~\ref{prop:N_large} (the large \(N\) case), whereas
Section~\ref{sec:the-small-N-case} will prove
Proposition~\ref{prop:N_small} (the small \(N\) case).
Appendix~\ref{app:aux_lemmas} contains auxiliary numerical lemmas,
Appendix~\ref{app:implementation} provides technical details of the
computer-assisted parts of the proof and
Appendix~\ref{app:bounds-for-K4} includes computer-assisted bounds and
expansions of the residue function \(K_{4}(N,z,t)\) introduced in
Section~\ref{sec:the-large-N-case}. The full code for the
computer-assisted computations is available
at~\cite{SpectralRegularPolygon.jl}. Throughout the paper we use the
color \resultnumber{magenta} to indicate numbers that come from
computational results.

\subsection{Acknowledgements}
JGS has been partially supported by the MICINN (Spain) research grant number PID2021–125021NA–I00,  by NSF under Grants
DMS-2245017, DMS-2247537, DMS-2434314 and DMS-2554957, by the AGAUR project 2021-SGR-0087 (Catalunya) and by a Simons Fellowship. JPA has been partially supported by NSF under Grant DMS-2247537, by a MOBINT grant from Generalitat de Catalunya, and by CFIS Excellence Grants (for studies at UPC and the Mobility Program), particularly by Fundaci\'o Privada Mir-Puig, CFIS partners, and donors of the crowdfunding program. JPA would also like to wholeheartedly thank Brown University for hosting her during part of the development of this project.
This material is based upon work supported by a grant from the Institute for Advanced Study School of
Mathematics. We thank Bogdan Georgiev and Danylo Radchenko for many useful conversations.

\section{The large $N$ case}
\label{sec:the-large-N-case}

This section is devoted to prove the following proposition:

\begin{proposition}
  \label{prop:N_large}
  For \(N_{0} = \codenumber{64}\) we have
  \begin{equation*}
    \lambda_1(\mathcal{P}_{N_0}) > \lambda_1(\mathcal{P}_{N_0+1}) >
    \dots > \lambda_1(\mathcal{P}_N) > \lambda_1(\mathcal{P}_{N+1}) >
    \dots > \lambda_1(\mathbb{D}) \, .
  \end{equation*}
  Furthermore, for \(q_N\) as in Theorem~\ref{main_thm}, it also holds
  that
  \begin{equation*}
    q_{N_0} > q_{N_0+1} > \dots > q_N > q_{N+1} > \dots > 1 \, .
  \end{equation*}
\end{proposition}

Our main goal is now to construct the approximate eigenpair $(u_{app},\lambda_{app})$ in a way that Lemma~\ref{lem:FoxHenriciMoler} yields accurate enough bounds for $N \geq N_0$. We will first describe the construction, then (upper) bound $\|u_{app}\|_{L^\infty(\partial \mathcal{P}_n)}$ and (lower) bound $\|u_{app}\|_{L^{2}(\mathcal{P}_{N})}$. For simplicity, we will sometimes drop the dependence on $N$ from the notation whenever it is clear. We will also denote $\lambda^{(N)} = \lambda_1(\mathcal{P}_N)$ and $\lambda =  \lambda_1(\mathbb{D}) = j_{0,1}^2$, where $j_ {0,1}$ is the first zero of the Bessel function $J_0(r)$.

First, let us introduce the  Schwarz-Christoffel map and its properties, since they will be used throughout the paper. The Schwarz-Christoffel map $f_N: \mathbb{D} \to \mathcal{P}_N$ maps the unit disk $\mathbb{D}$ conformally onto $\mathcal{P}_N$, as well as their respective boundaries. It is defined by:
\begin{equation}
\label{eq:Schwarz-Christoffel_map}
    f_N(z) =  c_N z \; _2F_1 \left(\frac{2}{N}, \frac{1}{N}, 1 + \frac{1}{N}; z^N\right) = c_N \int_0^z \frac{d\zeta}{(1-\zeta^N)^{2/N}} \, , \quad z \in \mathbb{D},
\end{equation}
where
\begin{equation}\label{eq:c_N}
    c_N = \sqrt{\dfrac{\Gamma(1-\frac{1}{N})^2 \Gamma(1+\frac{2}{N})}{ \Gamma(1+\frac{1}{N})^2 \Gamma(1-\frac{2}{N}) }}
\end{equation}
in order to ensure that $\text{Area}(\mathcal{P}_N) = \pi$ and $_2F_1$ is the hypergeometric function given by
\begin{equation*}
    \,_2F_1(a,b,c;z) = \sum_{n \geq 0} \frac{(a)_n (b)_n}{(c)_n} \frac{z^n}{n!} \, , \quad |z| < 1 \, ,
\end{equation*}
where $(a)_n = a(a+1)\cdots (a+n-1)$ is the rising Pochhammer symbol. To simplify notation we will write
\begin{equation}\label{eq:FNz_integral}
    F_N(z) = \; _2F_1 \left(\frac{2}{N}, \frac{1}{N}, 1 + \frac{1}{N}; z\right) = 1 + \frac{1}{N} \int_0^1 t^{1/N} \left( (1-tz)^{-2/N} - 1 \right) \frac{dt}{t} \, ,
\end{equation}
the last equality following from integral representation of the hypergeometric function. Moreover, by expanding in powers of $1/N$, we have
\begin{equation}
\label{eq:FNz_expansion_powers_1/N}
    F_N(z) = 1 + \sum_{n=2}^{\infty} S_n(z) \frac{1}{N^n} \, ,
\end{equation}
where $S_n(z) = \sum_{j=1}^{n-1} (-1)^{j-1} 2^{n-j} S_{j,n-j}(z) $ with $S_{n,p}(z)$ being the Nielsen generalized polylogarithm function~\cite{Kolbig1986}.

We also present here the multiple polylogarithm functions, which will take a central part in some of the expressions in the paper. For $m_1,\dots,m_k$ positive integers, the multiple polylogarithm $\Li_{m_1,\dots,m_k}(z)$ is defined for $|z|<1$ by the series expansion
\begin{equation}
    \label{eq:multiple-polylogs-series}
    \Li_{m_1,\dots,m_k}(z) = \sum_{0<n_1<n_2<\dots<n_k} \frac{z^{n_k}}{n_1^{m_1} n_2^{m_2} \cdots n_r^{m_k}} .
\end{equation}
Moreover, they can be expressed by the set of recurrent differential equations
\begin{equation}
  \label{eq:multiple-polylogs-differential-equations}
  \begin{split}
    z \frac{d}{dz}\Li_{m_1,\dots,m_k}(z) &= \Li_{m_1,\dots, m_{k-1}, m_k-1}(z) \quad \text{if}\quad m_k \geq 2 \, , \\
    (1-z) \frac{d}{dz}\Li_{m_1,\dots,m_{k-1},1}(z) &= \Li_{m_1,\dots, m_{k-1}}(z) \quad \text{if}\quad k \geq 2 \, ,
  \end{split}
\end{equation}
together with $\Li_1(z) = -\log(1-z)$ and initial conditions $\Li_{m_1,\dots,m_k}(0) = 0$ for all $m_i\in\mathbb{Z}^{+}$, $i=1,\dots,k$, and $k\in\mathbb{Z}^+$.

We borrow the expansion of the approximate solution\footnote{There is a typo in the definition of $V_4(z)$ in~\cite{Berghaus-Georgiev-Monien-Radchenko:dirichlet-eigenvalues-regular-polygons}. We thank the authors for pointing this out to us.} from~\cite{Berghaus-Georgiev-Monien-Radchenko:dirichlet-eigenvalues-regular-polygons}.
\addjoana{The adopted expressions are treated as an ansatz, so later on, in~\eqref{eq:upper_bound_error}, we independently verify that they produce the desired cancellation in the $1/N$ expansion.}
We construct an approximate solution in the following way:

\begin{definition}
\label{def:u_app-lambda_app}
Let $\lambda_{app}$ be
\begin{equation}
\label{eq:expansion_lambda_N}
    \lambda_{app} = \lambda \left (  1 + \frac{4 \zeta(3)}{N^3} + \frac{(12 - 2\lambda) \zeta(5)}{N^5}\right ),
\end{equation}
where $\lambda = \lambda_1(\mathbb{D})$, and let $u_{app}:\mathcal{P}_N \to \mathbb{R}$ in $\mathcal{C}^2(\mathcal{P}_N) \cap \mathcal{C}(\overline{\mathcal{P}_N})$ be a dyhedrally symmetric function satisfying $\Delta u_{app} + \lambda_{app} u_{app} = 0$ in $\mathcal{P}_N$. Following~\cite[p. 58]{Vekua:methods-elliptic-equations}, we can write $u_{app}$ as
    \begin{equation}
    \label{eq:definition-of-u_app}
        u_{app}(x) = a_0 J_0(\sqrt{\lambda_{app}} \, |x|) + \text{Re} \int_0^x U(t) J_0 \left (\sqrt{\lambda_{app} \bar{x} (x-t)} \right ) dt \, ,
    \end{equation}
    with
    \begin{equation*}
        a_0 = \frac{c_N}{\lambda^{1/2} J_1(\lambda^{1/2})} \, \text{ and } \,
        U(f_N(z^{1/N}))  = \frac{V(z) (1-z)^{2/N}}{z^{1/N}} = \frac{V(z) c_N}{z^{1/N} f_N'(z^{1/N}) }  \, , \quad z \in \mathbb{D} \, ,
    \end{equation*}
    where $f_N$ is the Schwarz-Christoffel map and
    \begin{equation*}
        V(z) = \frac{V_1(z)}{N} + \frac{V_2(z)}{N^2} + \frac{V_3(z)}{N^3} + \frac{V_4(z)}{N^4}.
    \end{equation*}
    The functions $V_j : \mathbb{D} \to \mathbb{C}$ are holomorphic
    and $V_j(0) = 0$ for $j = 1,2,3,4.$ Their expressions can be seen
    below:
    \begin{equation}\label{eq:Vs}
      \begin{split}
        V_1(z) &= 2\Li_{1}(z) \, , \\
        V_2(z) &= \left( \frac{\lambda}{2} - 2 \right) \Li_2(z) + 4\Li_{1,1}(z) \, , \\
        V_3(z) &= \left( \frac{\lambda^2}{16} - \lambda + 2 \right) \Li_3(z) + (3\lambda - 12) \Li_{1,2}(z) + (\lambda - 4) \Li_{2,1}(z) + 8\Li_{1,1,1}(z) \\
        V_4(z) &= \left( \frac{\lambda^3}{192} - \frac{\lambda^2}{8} - \frac{\lambda}{2} - 2 \right) \Li_4(z) + \left( \frac{\lambda^2}{8} - 2\lambda + 4 \right)\Li_{3,1}(z) + \left( \frac{\lambda^2}{4} -4\lambda + 12 \right) \Li_{2,2}(z) \\
               & + \left( \frac{5\lambda^2}{8} - 8\lambda + 28 \right) \Li_{1,3}(z) + (2\lambda - 8) \Li_{2,1,1}(z) + (6\lambda - 24)\Li_{1,2,1}(z) + (14\lambda - 56)\Li_{1,1,2}(z) \\
               & + 16\Li_{1,1,1,1}(z) + 2 \lambda\zeta(3)\Li_1(z)
      \end{split}
    \end{equation}
    where $\Li_s(z) = \sum_{n>0} \dfrac{z^n}{n^s}$ is the polylog
    function, and the multiple polylogarithms \(\Li_{1,1}\),
    \(\Li_{1,2}\), \(\Li_{1,3}\), \(\Li_{2,1}\), \(\Li_{2,2}\),
    \(\Li_{3,1}\), \(\Li_{1,1,1}\), \(\Li_{1,2,1}\), \(\Li_{2,1,1}\)
    and \(\Li_{1,1,1,1}\) are given
    by~\eqref{eq:multiple-polylogs-series} and described in more
    detail in Appendix~\ref{sec:polylogs}.
\end{definition}

\begin{remark}
Even though the asymptotic expansion in~\cite{Berghaus-Georgiev-Monien-Radchenko:dirichlet-eigenvalues-regular-polygons} contains more terms, heuristically, it is enough to take $(\lambda_{app},u_{app})$ correct up to one order in $1/N$ less than the desired error for $|\lambda^{(N)} - \lambda_{app}|$. For instance, taking $(\lambda_{app},u_{app})$ correct up to $O(N^{-4})$ will lead to $O(N^{-5})$ errors, which are much smaller than the $O(N^{-4})$ distance between $\lambda^{(N)}$ and $\lambda^{(N-1)}$ for sufficiently large $N$.
Making the approximation as simple as possible will lead to easier estimates of the different errors arising throughout the proof.
However, when analyzing the $q_N$ ratio, we need to expand further since $|q_N - q_{N-1}| = O(N^{-5})$, needing an approximation which is correct up to $O(N^{-5})$.
This is, for example, why we are taking terms up to order $N^{-5}$ in $\lambda_{app}$.
\end{remark}

In order to be able to evaluate more efficiently and extract cancellations, we perform the change of variables $v_{app}(z) = u_{app}(f_N(z^{1/N}))$ using the Schwarz-Christoffel map. We obtain that
\begin{equation}
\label{eq:candidate-on-disk-v_app}
    v_{app}(z) = a_0 J_0(\rho^{1/2} |z|^{1/N} |F_N(z)|)  +  \frac{c_N}{N} \text{Re} \int_0^z V(t) K(z,t) \frac{dt}{t} \, ,
\end{equation}
where $\rho = c_N^2 \lambda_{app}$ and
\begin{equation}
  \label{eq:definition-of-K(z,t)}
  K(z,t) = J_0 \left(\rho^{1/2} |z|^{1/N} \sqrt{F_N(\bar{z}) \left( F_N(z) - (t/z)^{1/N} F_N(t) \right)}\right) \, .
\end{equation}
\addjoana{Note that both integrals appearing in~\eqref{eq:definition-of-u_app} and~\eqref{eq:candidate-on-disk-v_app} are thought of as integrating along the segment $[0,x]$ and $[0,z]$, respectively. Thus, both Bessel functions from~\eqref{eq:definition-of-u_app} and $K(z,t)$ have real arguments, so they are real-valued.}

In order to apply Lemma \ref{lem:FoxHenriciMoler}, we want to bound the quantities $\| u_{app} \|_{L^{\infty}(\partial \mathcal{P}_N)} = \|v_{app}\|_{L^{\infty}(\partial \mathbb{D})}$ and $\|u_{app}\|_{L^2(\mathcal{P}_N)}$.
Section~\ref{subsec:upper_bound_boundary} is devoted to computing an upper bound $\varepsilon(N)$ of the former, and  Section~\ref{subsec:lower_bounds_on_L2_norm} to deriving a lower bound $\eta(N)$ for the latter.

\subsection{Upper bounds on the defect $\| u_{app}\|_{L^{\infty}(\partial \mathcal{P}_N)}$}
\label{subsec:upper_bound_boundary}

In this subsection we want to compute an upper bound $\varepsilon(N)$ for $\| u_{app}\|_{L^{\infty}(\partial \mathcal{P}_N)}$. Changing variables into $v_{app}$ we need to bound:
\begin{align}
    \sup_{z \in \partial \mathbb{D}} |v_{app}(z)|
     & = \sup_{|z| = 1} \left | a_0 J_0 \left ( \rho^{1/2}  |F_N(z)| \right )  + \frac{c_N}{N} \text{Re} \int_0^{z} V(s) K(z,s) \frac{ds}{s} \right |.
\label{eq:supreme_for_which_we_want_an_upper_bound}
\end{align}
We will bound this supremum after extracting the relevant cancellations. We expect the supremum to be $O(N^{-6})$.

Let us focus on the first part: $a_0 J_0 \left ( \rho^{1/2} |F_N(z)| \right )$ with $\rho = c_N^2 \lambda_{app}$.
Let $g(w) = J_0(w\sqrt{\lambda})$. Thus, $J_0\left(c_N \sqrt{\lambda_{app}} |F_N(z)| \right) = g\left(\dfrac{\sqrt{\lambda_{app}}}{\sqrt{\lambda}} c_N |F_N(z)| \right)$. Expanding \(g\) up to second order around $\dfrac{\sqrt{\lambda_{app}}}{\sqrt{\lambda}} c_N|F_N(z)| = 1$ gives us
\begin{equation}
    \label{expansion_g}g\left(\dfrac{\sqrt{\lambda_{app}}}{\sqrt{\lambda}} c_N |F_N(z)| \right) = g(1) + g'(1)\left(\dfrac{\sqrt{\lambda_{app}}}{\sqrt{\lambda}} c_N |F_N(z)|  -  1\right) + g''(1)\frac{1}{2} \left(\dfrac{\sqrt{\lambda_{app}}}{\sqrt{\lambda}} c_N |F_N(z)|  - 1\right)^2 + E_3(N,z) \, ,
\end{equation}
where $g(1) = J_0(\sqrt{\lambda}) = 0$ given that $\lambda = j_{0,1}^2$,
and for the error term we have the bound
\begin{equation}\label{eq:E_3-bound}
  |E_3(N,z)| \leq \frac{1}{6}||g'''||_{L^{\infty}([0, \infty))}
  \left|\dfrac{\sqrt{\lambda_{app}}}{\sqrt{\lambda}} c_N |F_N(z)| - 1\right|^3.
\end{equation}

\begin{lemma}
\label{lemma:bounds_and_splitting:first-term-J0:expansion-of-argument-inside-g-function}
    We have the following expansion up to order $O(N^{-6})$:
    \begin{equation*}
        \frac{\sqrt{\lambda_{app}}}{\sqrt{\lambda}} c_N |F_N(z)| - 1 = \frac{1}{N^2} b_2(z) + \frac{1}{N^3} b_3(z) + \frac{1}{N^4} b_4(z) + \frac{1}{N^5} b_5(z) + \frac{1}{N^6}T_6(N,z) \, .
    \end{equation*}
    Furthermore, we may also write
    \begin{equation*}
        \begin{split}
        \frac{\sqrt{\lambda_{app}}}{\sqrt{\lambda}} c_N |F_N(z)| - 1 & = \frac{1}{N^2} T_2(N,z) \\
        & = \frac{1}{N^2} b_2(z) + \frac{1}{N^3} b_3(z) + \frac{1}{N^4}T_4(N,z) \, .
        \end{split}
    \end{equation*}
\end{lemma}
\begin{proof}
    We use the expansion of $F_N(z)$ in powers of $1/N$ given by \eqref{eq:FNz_expansion_powers_1/N}, namely, $F_N(z) =  1 + \sum_{n=2}^{\infty} S_n(z) \frac{1}{N^n}$.
    By expanding $\frac{\sqrt{\lambda_{app}}}{\sqrt{\lambda}} c_N |F_N(z)| - 1$ up to order $O(N^{-5})$ we obtain that
    \begin{equation}\label{eq:bs}
      \begin{split}
        b_2(z) &= \real S_2(z),\\
        b_3(z) &= \real S_3(z),\\
        b_4(z) &= \frac{(\imag S_2(z))^2}{2} + \real S_4(z),\\
        b_5(z) &= \imag S_2(z) \cdot \imag S_3(z) + \real S_5(z) - \lambda \zeta(5).
      \end{split}
    \end{equation}
    The rest of the terms $T_2(N,z), T_4(N,z)$ and $T_6(N,z)$ are the error terms left from subtracting the corresponding $b_i(z)$ from the original expression, that is,
    \begin{equation}\label{eq:T_246}
      \begin{split}
        T_2(N,z) &= \left ( \frac{\sqrt{\lambda_{app}}}{\sqrt{\lambda}} c_N |F_N(z)| - 1 \right ) N^2 \, , \\
        T_4(N,z) &=  \left ( \frac{\sqrt{\lambda_{app}}}{\sqrt{\lambda}} c_N |F_N(z)| - 1  - \frac{b_2(z)}{N^2} - \frac{b_3(z)}{N^3} \right ) N^4 \\
        T_6(N,z) &=  \left ( \frac{\sqrt{\lambda_{app}}}{\sqrt{\lambda}} c_N |F_N(z)| - 1  - \frac{b_2(z)}{N^2} - \frac{b_3(z)}{N^3} - \frac{b_4(z)}{N^4} - \frac{b_5(z)}{N^5} \right ) N^6 \, .
      \end{split}
    \end{equation}
\end{proof}

Depending on which order of $\frac{1}{N}$ is needed, we use these expressions interchangeably. Substituting into \eqref{expansion_g}:
\begin{align*}
    g\left(\dfrac{\sqrt{\lambda_{app}}}{\sqrt{\lambda}} c_N |F_N(z)|\right)  & = g'(1)\left(\dfrac{\sqrt{\lambda_{app}}}{\sqrt{\lambda}} c_N |F_N(z)| - 1\right) + g''(1)\frac{1}{2} \left(\dfrac{\sqrt{\lambda_{app}}}{\sqrt{\lambda}} c_N |F_N(z)| - 1\right)^2 + E_3(N,z) \\
    & = g'(1) \left(\frac{1}{N^2}b_2(z) + \frac{1}{N^3}b_3(z) + \frac{1}{N^4}b_4(z) + \frac{1}{N^5}b_5(z) + \frac{1}{N^6}T_6(N,z) \right) \\
    & \quad + \frac{g''(1)}{2} \left(\frac{1}{N^2} b_2(z) + \frac{1}{N^3} b_3(z) + \frac{1}{N^4}T_4(N,z)\right)^2 + E_3(N,z) \\
    & =  g'(1) \left(\frac{1}{N^2}b_2(z) + \frac{1}{N^3}b_3(z) + \frac{1}{N^4}b_4(z) + \frac{1}{N^5}b_5(z) + \frac{1}{N^6}T_6(N,z) \right) + \frac{g''(1)}{2} \left( \frac{1}{N^4}b^2_2(z) \right. \\
    & \quad \left. + \frac{1}{N^6}b_3^2(z) + \frac{1}{N^8}T_4^2(N,z) + \frac{2}{N^5}b_2(z)b_3(z) + \frac{2}{N^6}b_2(z)T_4(N,z) + \frac{2}{N^7}b_3(z)T_4(N,z) \right)  + E_3(N,z) \\
    & =  \frac{1}{N^2} \left( g'(1)b_2(z) \right)
      +  \frac{1}{N^3} \left( g'(1)b_3(z) \right)
      +  \frac{1}{N^4} \left( g'(1)b_4(z) + \frac{g''(1)}{2} b_2^2(z) \right) \\
     & \quad  +  \frac{1}{N^5} \left( g'(1)b_5(z) + g''(1) b_2(z)b_3(z) \right) \\
    & \quad + \frac{1}{N^6}\left( g'(1) T_6(N,z) + \frac{g''(1)}{2}b_3^2(z) + g''(1) b_2(z)T_4(N,z) \right)
    + E_3(N,z) \\
    & \quad + \frac{g''(1)}{N^7}b_3(z)T_4(N,z) + \frac{g''(1)}{2 N^8}T_4^2(N,z) \, .
\end{align*}
Here, the first two lines in the last expression are the main terms, which will cancel with the corresponding terms from the second summand of
\eqref{eq:supreme_for_which_we_want_an_upper_bound}, and the last two lines are the error terms. Let us denote
\begin{align*}
    E_{J_0}(N, z) = \frac{1}{N^6}\left( g'(1) T_6(N,z) + \frac{g''(1)}{2}b_3^2(z) + g''(1) b_2(z)T_4(N,z) \right) + E_3(N,z) + \frac{g''(1)}{N^7}b_3(z)T_4(N,z) + \frac{g''(1)}{2 N^8}T_4^2(N,z) ,
\end{align*}
where $|E_3(N,z)| \leq \frac{1}{6} \|g'''\|_\infty \frac{\|T_2(N,\cdot)\|_\infty}{N^6}$.

We now provide bounds on $E_{J_0}$.

\begin{lemma}
  \label{lemma:bounds_and_splitting:first_term_J0:upper-bounds-for-terms-from-g-function}
  Let \(|z| = 1\), we have the bounds
  \begin{equation*}
    |b_{2}(z)| \leq C_{b,2},\
    |b_{3}(z)| \leq C_{b,3},\
    |b_{4}(z)| \leq C_{b,4},\
    |b_{5}(z)| \leq C_{b,5}
  \end{equation*}
  with
  \begin{equation*}
    C_{b,2} = \codenumber{3.5},\
    C_{b,3} = \codenumber{2.5},\
    C_{b,4} = \codenumber{10},\
    C_{b,5} = \codenumber{11}.
  \end{equation*}
  Moreover, for \(N \geq N_0 = \codenumber{64}\) we also have
  \begin{equation*}
    |T_{2}(N,z)| \leq C_{T,2},\
    |T_{4}(N,z)| \leq C_{T,4},\
    |T_{6}(N,z)| \leq C_{T,6}
  \end{equation*}
  with
  \begin{equation*}
    C_{T,2} = \codenumber{4},\
    C_{T,4} = \codenumber{11},\
    C_{T,6} = \codenumber{50}.
  \end{equation*}
\end{lemma}
\begin{proof}
  Let us start by noting that for \(T_{2}\) and \(T_{4}\) we have,
  assuming the above bounds for the \(b\)'s and \(T_{6}\),
  \begin{equation*}
    |T_2(N,z)| \leq
    |b_{2}(z)|
    + \frac{|b_{3}(z)|}{N_{0}}
    + \frac{|b_{4}(z)|}{N_{0}^{2}}
    + \frac{|b_{5}(z)|}{N_{0}^{3}}
    + \frac{|T_{6}(N,z)|}{N_{0}^{4}}
    \leq
    C_{b,2}
    + \frac{C_{b,3}}{N_{0}}
    + \frac{C_{b,4}}{N_{0}^{2}}
    + \frac{C_{b,5}}{N_{0}^{3}}
    + \frac{C_{T,6}}{N_{0}^{4}}
    \leq C_{T,2}
  \end{equation*}
  and
  \begin{equation*}
    |T_4(N,z)| \leq
    |b_{4}(z)| +
    \frac{|b_{5}(z)|}{N_{0}} +
    \frac{|T_{6}(N,z)|}{N_{0}^{2}}
    \leq
    C_{b,4} +
    \frac{C_{b,5}}{N_{0}} +
    \frac{C_{T,6}}{N_{0}^{2}}
    \leq C_{T,4}.
  \end{equation*}
  It therefore suffices to verify the bounds for the \(b\)'s and
  \(T_{6}\).

  Since the functions are all conjugate symmetric it suffices to bound
  them for \(z = e^{i\theta}\) with \(\theta \in [0, \pi]\). We make
  use of the algorithm discussed in
  Appendix~\ref{sec:enclosing-extrema} for enclosing the maximum on
  this interval. For the \(b\)'s this is straightforward using the
  expressions in terms of \(S_{n}\)'s given by \eqref{eq:bs}. The
  procedure to compute enclosures of the \(S_{n}\)'s is discussed in
  Appendix~\ref{sec:polylogs}.

  For the \(b\)'s this approach gives us the following enclosures for
  the maximum values:
  \begin{align*}
    \max_{\theta \in [0, \pi]} |b_{2}(e^{i\theta})| \in \resultnumber{[3.29 \pm 1.28 \cdot 10^{-3}]}, &\quad
    \max_{\theta \in [0, \pi]} |b_{3}(e^{i\theta})| \in \resultnumber{[2.41 \pm 6.82 \cdot 10^{-3}]},\\
    \max_{\theta \in [0, \pi]} |b_{4}(e^{i\theta})| \in \resultnumber{[9.8 \pm 0.0707]}, & \quad
    \max_{\theta \in [0, \pi]} |b_{5}(e^{i\theta})| \in \resultnumber{[10.80 \pm 8.35 \cdot 10^{-3}]},
  \end{align*}
  which immediately implies the required bounds.

  To compute enclosures for \(T_{6}\) we make use of Taylor models. We
  refer to Appendix~\ref{sec:taylor-models} for more details about
  Taylor models and how to compute \(T_{6}\) using them. With this we
  get the following enclosure
  \begin{equation*}
    \max_{\theta \in [0, \pi]} |T_{6}(N,e^{i\theta})| \in \resultnumber{[25.49, 45.91]},
  \end{equation*}
  valid for $N \geq N_0$. This again implies the required bound.
\end{proof}

\begin{lemma}
\label{lemma:g_function_derivatives}
    For $g(w) = J_0(w\sqrt{\lambda})$, we have that
    \begin{align*}
        g'(1) &= - \sqrt{\lambda} J_1(\sqrt{\lambda}) \, , \\
        g''(1) &= \frac{1}{2} J_2(\sqrt{\lambda}) \lambda \, ,
    \end{align*}
    and in particular $g'(1) = -g''(1)$. Moreover, we have the bound
    \begin{equation*}
      \|g'''\|_{L^{\infty}([0, \infty))} \leq C_{g'''},
    \end{equation*}
    with \(C_{g'''} = \codenumber{6.1}\).
\end{lemma}
\begin{proof}
  Given that $\lambda = j_{0,1}^2$, $J_0'(x) = -J_1(x)$ and
  $J_k'(x) = \frac{1}{2}(J_{k-1}(x) - J_{k+1}(x))$ for $k \geq 1$, we
  get
  \begin{align*}
    g'(w) &= - \sqrt{\lambda} J_1(w \sqrt{\lambda}) \, , \quad
    & g'(1) &= - \sqrt{\lambda} J_1(\sqrt{\lambda}) \, , \\
    g''(w) &= - \frac{1}{2} \lambda (J_0(w \sqrt{\lambda}) - J_2(w \sqrt{\lambda})) \, , \quad
    & g''(1) &= \frac{1}{2} J_2(\sqrt{\lambda}) \lambda \, , \\
    g'''(w) &= - \frac{1}{4} \lambda^{3/2} \left( 3 J_1(w \sqrt{\lambda}) - J_3(w \sqrt{\lambda})\right) \,  . &
  \end{align*}
  Applying the recurrence identity,
  $J_{\nu}(z) = \frac{2 (\nu+1)}{z} J_{\nu+1}(z) - J_{\nu+2}(z)$, of
  the Bessel functions at $z=\sqrt{\lambda}$ gives us that
  $g'(1) = -g''(1)$.

  The bound for \(\|g'''\|_{L^{\infty}([0, \infty))}\) is
  computer-assisted. We split the interval into two parts, \([0, 5]\)
  and \((5, \infty)\). For $w \in [0, 5]$ we make use of the algorithm
  discussed in Appendix~\ref{sec:enclosing-extrema} for enclosing the
  maximum, giving us:
  \begin{equation*}
    \max_{w \in [0, 5]} |g'''(w)| \in \resultnumber{[6.06924 \pm 3.39 \cdot 10^{-6}]} < C_{g'''}.
  \end{equation*}
  For $w \in (5, \infty)$ we use the bound
  \begin{equation*}
    |J_{\nu}(x)| \leq 0.7858 x^{-1 / 3},
  \end{equation*}
  from~\cite{fungrim-7f3485}, giving
  \begin{equation*}
    |g'''(w)| \leq 0.7858\lambda^{4 / 3}w^{-1 / 3}
    \leq 0.7858\lambda^{4 / 3}5^{-1 / 3}
    \in \resultnumber{[4.7703 \pm 8.71 \cdot 10^{-6}]} < C_{g'''}.
  \end{equation*}
\end{proof}

Combining the bounds from
Lemmas~\ref{lemma:bounds_and_splitting:first_term_J0:upper-bounds-for-terms-from-g-function}
and~\ref{lemma:g_function_derivatives} and Equation~\eqref{eq:E_3-bound}, we can get a bound on
$E_{J_0}$:

\begin{corollary}
  \label{cor:bounds_and_splitting:first_term_J0:error-bound-for-g-function}
  Let \(N \geq N_{0} = \codenumber{64}\), then for \(|z| = 1\) we have
  the bound
  \begin{align*}
    |E_{J_0}(z)|
    & \leq \frac{1}{N^6}\left( |g'(1)|\|T_6(N,\cdot)\|_{\infty}
      + \frac{|g''(1)|}{2}\|b_3(z)\|_{\infty}^2
      + |g''(1)|\|b_2(z)\|_{\infty}\|T_4(N,\cdot)\|_{\infty}
      + \frac{1}{6}\| T_2(N,\cdot)\|^3_{\infty}\| g''' \|_{\infty}\right.\\
    &\quad \left. + \frac{|g''(1)|}{N}\|b_3(z)\|_{\infty} \|T_4(N,\cdot)\|_{\infty}
      + \frac{|g''(1)|}{2 N^2}\|T_4(N,\cdot)\|_{\infty}^2 \right)\\
    &
      \leq  \frac{1}{N^6} \left(
      |g'(1)|C_{T,6}
      + \frac{|g''(1)|}{2}C_{b,3}^{2}
      + |g''(1)|C_{b,2}C_{T,4}
      + \frac{1}{6}C_{T,2}^{3}C_{g'''}
      + \frac{|g''(1)|}{N}C_{b,3}C_{T,4}
      + \frac{|g''(1)|}{2 N^2}C_{T,4}^{2}
      \right).
  \end{align*}
\end{corollary}

We now proceed to extract the main terms and the error terms from the second summand of \eqref{eq:supreme_for_which_we_want_an_upper_bound}.
As mentioned earlier, it is enough to expand terms until the error is of order $O(N^{-6})$. Note that there is an additional factor $1/N$ in front of the integral and the function $V(z)$ itself also contains a $1/N$ factor. Therefore, we can begin by expanding the \addjoana{Bessel} kernel $K(z,t)$ up to order $O(N^{-4  })$. \addjoana{Recall that $t$ is taken along the segment $[0,z]$, so $K(z,t)$ and its Taylor coefficients are real-valued.}

\begin{lemma}
\label{lemma:exact-expression-for-K-up-to-order_-3}
    For $|z|=1$ \addjoana{and $t=rz$, $0<r<1$}, we write $K(z,t) = d_0(z,t) + \frac{1}{N}d_1(z,t) + \frac{1}{N^2}d_2(z,t) + \frac{1}{N^3}d_3(z,t) + \frac{1}{N^4}K_4(N,z,t)$. Then, the exact expressions for the terms up to order $O(N^{-3})$ are given by
    \begin{equation}
      \begin{split}
        d_0(z,t) &= 1 \, , \\
        d_1(z,t) &= \frac{\lambda  \log(\frac{t}{z}) }{4} \, , \\
        d_2(z,t) &= \frac{1}{64}\lambda^2 \log^2\left(\frac{t}{z}\right) + \frac{1}{8} \lambda \left( \log^2\left(\frac{t}{z}\right) +2 S_2(t) - 2S_2(z) \right) \\
        d_3(z,t) &= \frac{\lambda^3 \log^3\left(\frac{t}{z}\right)}{2304} + \frac{1}{64} \lambda^2 \log\left(\frac{t}{z}\right) \left( \log^2\left(\frac{t}{z}\right) + 2 S_2(t) - 2 S_2(z) \right) \\
                 & + \frac{1}{24} \lambda \left( \log^3\left(\frac{t}{z}\right) + 6 \log\left(\frac{t}{z}\right) S_2(t) + 6 \log\left(\frac{t}{z}\right)S_2(\bar{z}) + 6S_3(t) - 6 S_3(z) \right) \, .
      \end{split}
    \end{equation}
\end{lemma}
\begin{proof}
Given the definition of $K(z,t)$ from \eqref{eq:definition-of-K(z,t)} and using the condition $|z|=1$, \addjoana{$t=rz$ with $0<r<1$, we first note that the argument of the Bessel function is real. Thus, we can} expand the Taylor series of the Bessel function at the origin:
\begin{multline*}
    K(z,t) = J_0 \left(\rho^{1/2} \sqrt{F_N(\bar{z}) \left( F_N(z) - (t/z)^{1/N} F_N(t) \right)}\right) \\
    = \sum_{n=0}^{3} \left(-\frac{\rho}{4}\right)^n \frac{F_N(\bar{z})^n (F_N(z) -(t/z)^{1/N} F_N(t) )^n}{(n!)^2} + \frac{J_0^{(8)}(\xi)}{8!} \left(\rho  F_N(\bar{z}) \left(F_N(z) -\left(\frac{t}{z}\right)^{1/N} F_N(t) \right) \right)^{4} \, ,
\end{multline*}
with $\xi \in \left(0, \rho^{1/2} \sqrt{F_N(\bar{z}) \left( F_N(z) - (t/z)^{1/N} F_N(t) \right)} \right)$.
The dependency on $N$ of the functions that appear in the $K(z,t)$ expansion is the following:
$$\rho = c_N^2 \lambda_{app} =\lambda \left( 1 + \frac{-2\lambda^2\zeta(5)}{N^5} + O(N^{-6}) \right) \, ,$$
$$F_N(z) = 1 + \sum_{n=2}^{\infty} \frac{S_n(z)}{N^n} = 1 + \frac{S_2(z)}{N^2} + \frac{S_3(z)}{N^3} + O(N^{-4})  \, ,$$
$$\left(\frac{t}{z}\right)^{1/N} = \sum_{n\geq0} \frac{\log^n(t/z)}{n!} \frac{1}{N^n} = 1 + \frac{\log(t/z)}{N} + \frac{\log^2(t/z)}{2N^2} + \frac{\log^3(t/z)}{6N^3} + O(N^{-4})  \, .$$
Note that the residue of $K(z,t)$ in the Taylor expression is of order $O(N^{-4})$, so expanding the functions in the finite sum up to order $O(N^{-3})$ is enough to compute the terms $d_0(z,t), d_1(z,t), d_2(z,t)$ and $d_3(z,t)$, which result in the ones from the statement.
Note that the construction of $c_N$ and $\lambda_{app}$ causes the $N^{-3}$ term in the expansion of $\rho$ to be canceled, and some $N^{-5}$ terms to be partially canceled.
\end{proof}

Let us expand up to order $O(N^{-5})$ the multiplication of $V(t)$ and the kernel $K(z,t)$ from the integrand.

\begin{align*}
    V(t) K(z,t)
    & = \left( \frac{V_1(t)}{N} + \frac{V_2(t)}{N^2} + \frac{V_3(t)}{N^3} + \frac{V_4(t)}{N^4} \right) \left( d_0(z,t) + \frac{1}{N}d_1(z,t) + \frac{1}{N^2}d_2(z,t) + \frac{1}{N^3}d_3(z,t) + \frac{1}{N^4}K_4(N,z,t) \right) \\
    & = \frac{1}{N} d_0(z,t) V_1(t)
    + \frac{1}{N^2} \left(  d_0(z,t) V_2(t) + d_1(z,t) V_1(t) \right) \\
    & \quad + \frac{1}{N^3} \left( d_0(z,t) V_3(t) + d_1(z,t) V_2(t) + d_2(z,t) V_1(t) \right) \\
    & \quad + \frac{1}{N^4} \left( d_0(z,t) V_4(t) + d_1(z,t) V_3(t) + d_2(z,t) V_2(t) + d_3(z,t) V_1(t) \right) \\
    & \quad + \frac{1}{N^5} \Big(  d_1(z,t) V_4(t) + d_2(z,t) V_3(t) + d_3(z,t) V_2(t) + K_4(N,z,t) V_1(t)  \\
    &  \qquad + \frac{1}{N} \left(  d_2(z,t) V_4(t) + d_3(z,t) V_3(t) + K_4(N,z,t) V_2(t) \right) \\
    & \qquad + \frac{1}{N^2} \left(  d_3(z,t) V_4(t) + K_4(N,z,t) V_3(t) \right) + \frac{1}{N^3} \left(  K_4(N,z,t) V_4(t) \right) \Big) \, .
\end{align*}
Integrating term by term, let us define
\begin{align*}
    G_2(z) & = \text{Re} \int_0^z \left( \frac{1}{t} d_0(z,t) V_1(t) \right) dt \, , \\
    G_3(z) & = \text{Re} \int_0^z \left( \frac{1}{t} (d_0(z,t) V_2(t) + d_1(z,t) V_1(t)) \right) dt \, , \\
    G_4(z) & = \text{Re} \int_0^z \left( \frac{1}{t} ( d_0(z,t) V_3(t) + d_1(z,t) V_2(t) + d_2(z,t) V_1(t) ) \right) dt \, , \\
    G_5(z) & = \text{Re} \int_0^z \left( \frac{1}{t} ( d_0(z,t) V_4(t) + d_1(z,t) V_3(t) + d_2(z,t) V_2(t) + d_3(z,t) V_1(t) ) \right) dt \, ,
\end{align*}
and the error term being
\begin{align*}
    E_K(N, z) & = \frac{1}{N^6} \text{Re} \int_0^z \frac{1}{t} \Big(  d_1(z,t) V_4(t) + d_2(z,t) V_3(t) + d_3(z,t) V_2(t) + K_4(N,z,t) V_1(t) \\
    & \quad\quad\quad\quad + \frac{1}{N} \left(  d_2(z,t) V_4(t) + d_3(z,t) V_3(t) + K_4(N,z,t) V_2(t) \right) \\
    & \quad\quad\quad\quad + \frac{1}{N^2} \left(  d_3(z,t) V_4(t) + K_4(N,z,t) V_3(t) \right) \\
    & \quad\quad\quad\quad + \frac{1}{N^3} \left(  K_4(N,z,t) V_4(t) \right) \Big) dt \, .
\end{align*}

We have the following bounds:

\begin{lemma}
  \label{lemma:bounds-I_k_l}
  Let $|z| =1$, then
  \begin{equation*}
    \left|\real \int_0^z  \frac{1}{t} d_k(z,t) V_l(t)  dt \right| \leq C_{I,k,l}
  \end{equation*}
  for \(C_{I,k,l}\) given by
  \begin{equation*}
    C_{I,1,4} = \codenumber{15},\quad
    C_{I,2,3} = \codenumber{10},\quad
    C_{I,2,4} = \codenumber{1000},\quad
    C_{I,3,2} = \codenumber{35},\quad
    C_{I,3,3} = \codenumber{500} \text{ and }
    C_{I,3,4} = \codenumber{1000}.
  \end{equation*}
  Moreover, for \(N \geq N_{0} = \codenumber{64}\) we also have
  \begin{equation*}
    \left|\real \int_0^z  \frac{1}{t} K_4(z,t) V_l(t) dt \right| \leq C_{I,K,l}
  \end{equation*}
  for
  \begin{equation*}
    C_{I,K,1} = \codenumber{40},\quad
    C_{I,K,2} = \codenumber{500},\quad
    C_{I,K,3} = \codenumber{1000} \text{ and }
    C_{I,K,4} = \codenumber{5000}.
  \end{equation*}
\end{lemma}
\begin{proof}
  The proof is similar to
  Lemma~\ref{lemma:bounds_and_splitting:first_term_J0:upper-bounds-for-terms-from-g-function}.
  The functions are conjugate symmetric, so it suffices to bound them
  for \(z = e^{i\theta}\) with \(\theta \in [0, \pi]\). The maximum is
  then enclosed using the algorithm discussed in
  Appendix~\ref{sec:enclosing-extrema}. We now discuss the specific
  technicalities of the computation of enclosures of the integrals.

  In the error term \(E_{K}\), some of the integrals appear with a
  factor \(\frac{1}{N}\) to some power in front of them, for these
  integrals we don't need as precise bounds since the factor
  \(\frac{1}{N}\) helps making them small in either case. For
  \(C_{I,k,l}\) this is the case when \(k + l > 5\) and for
  \(C_{I,K,l}\) when \(l > 1\). We comment further on the adjustments
  we make in these cases below.

  Let us start with how to compute an enclosure of
  \begin{equation*}
    I_{k,l}(z) := \int_{0}^{z} \frac{1}{t} d_{k}(z, t) V_{l}(t)\ dt.
  \end{equation*}
  The integrand has a (integrable) singularity at \(t = 0\), to handle
  this we split the integral as
  \begin{equation*}
    I_{k,l}(z) = \int_{0}^{az} \frac{1}{t} d_{k}(z, t) V_{l}(t)\ dt
    + \int_{az}^{z} \frac{1}{t} d_{k}(z, t) V_{l}(t)\ dt
    =: I_{k,l,1}(z) + I_{k,l,2}(z),
  \end{equation*}
  for some \(0 < a < 1\). When \(k + l = 5\) we take
  \(a = \codenumber{10^{-4}}\) in the computations\footnote{In reality
    the code uses an exact floating point number close to this value.
    This is true for other similar cases in the paper where we split
    an integral like this.} and we expect \(I_{k,l,2}(z)\) to be the
  main contribution to the integral and \(I_{k,l,1}(z)\) to be an
  error term. For \(k + l > 5\) we take \(a = 0.5\), this leads to
  significant overestimations but is good enough for the rough bounds
  that we need in this case.

  For \(I_{k,l,1}\) we note that \(V_{l}(t) / t\) is bounded near
  \(t = 0\). Using Lemma~\ref{lemma:V_log_bounds} we obtain that
  \begin{equation*}
    \left|V_{l}(t) / t\right| \leq D_{l,a},
  \end{equation*}
  hence
  \begin{equation*}
    |I_{k,l,1}(z)| \leq D_{l,a} |z| \int_{0}^{a} |d_{k}(z, s z)|\ ds.
  \end{equation*}
  The remaining integral can now be explicitly computed, we refer to
  Lemma~\ref{lemma:integral_d} for the details.

  For \(I_{k,l,2}\) the integrand is in general finite on the interval
  of integration, except for when \(z = 1\) in which case the
  integrand has a (integrable) singularity at \(t = 1\). When our
  interval enclosure of \(z\) does not contain \(1\) we compute the
  entire integral using the techniques for rigorous numerical
  integration discussed in Appendix~\ref{sec:rigorous-integration}.
  For computing the polylogarithms occurring in \(d_{k}\) and
  \(V_{l}\) we proceed as discussed in Appendix~\ref{sec:polylogs}.
  Otherwise, when the interval of \(z\) contains \(1\), we further
  split the integral as
  \begin{equation*}
    I_{k,l,2}(z) = \int_{az}^{bz} \frac{1}{t} d_{k}(z, t) V_{l}(t)\ dt
    + \int_{bz}^{z} \frac{1}{t} d_{k}(z, t) V_{l}(t)\ dt
    =: I_{k,l,2,1}(z) + I_{k,l,2,2}(z),
  \end{equation*}
  for some \(a < b < 1\) close to \(1\) (in the computations we take
  \(b = \codenumber{0.999}\) when \(k + l = 5\) and
  \(b = \codenumber{0.99}\) when \(k + l > 5\)). In this case we
  expect \(I_{k,l,2,1}(z)\) to be the main contribution to the
  integral and \(I_{k,l,2,2}(z)\) to be an error term. For
  \(I_{k,l,2,1}\) the integrand is now bounded on the interval of
  integration, and we compute it using the same techniques for
  rigorous integration as before.

  For \(I_{k,l,2,2}\) the approach is similar to for \(I_{k,l,1}\),
  except that in this case we use that \(d_{k}(z, t) / t\) is bounded
  near \(z = t = 1\). We then get
  \begin{equation*}
    |I_{k,l,2,2}(z)| \leq \left(\max_{s \in [b, 1]} \left|\frac{d_{k}(z, s z)}{s z}\right|\right) |z|
    \int_{b}^{1} |V_{l}(s z)|\ ds,
  \end{equation*}
  where the maximum can be directly computed with interval arithmetic.
  To compute the remaining integral we make use of
  Lemma~\ref{lemma:V_log_bounds}, from which we have the bound
  \begin{equation*}
    |V_{l}(t)| \leq \sum_{j = 1}^{l} C_{V,l,j}\frac{|\log(1 - |t|)|^{j}}{j!}.
  \end{equation*}
  This gives us
  \begin{equation*}
    \int_{b}^{1} |V_{l}(sz)|\ ds \leq \sum_{j = 1}^{l}\frac{C_{V,l,j}}{j!}
    \int_{b}^{1}|\log(1 - s|z|)|^{j}\ ds.
  \end{equation*}
  Where for the last integral we note that since \(\log(1 - s|z|)\) has
  constant sign and \(|z| = 1\) we have
  \begin{equation*}
    \int_{b}^{1}|\log(1 - s|z|)|^{j}\ ds
    = \left|\int_{b}^{1}\log(1 - s|z|)^{j}\ ds\right|
    = \left|\int_{b}^{1}\log(1 - s)^{j}\ ds\right|.
  \end{equation*}
  Here the last integral can be computed explicitly, see
  Lemma~\ref{lemma:basic-integrals}.

  Following the above procedure we compute the enclosures
  \begin{equation*}
    \max_{\theta \in [0, \pi]} |I_{1,4}(e^{i\theta})| \in \resultnumber{[14 \pm 0.923]},\quad
    \max_{\theta \in [0, \pi]} |I_{2,3}(e^{i\theta})| \in \resultnumber{[7.848, 9.756]} \text{ and }
    \max_{\theta \in [0, \pi]} |I_{3,2}(e^{i\theta})| \in \resultnumber{[30 \pm 4.12]}
  \end{equation*}
  for the cases when \(k + l = 5\). For the other integrals we get
  the, very rough, enclosures
  \begin{equation*}
    \max_{\theta \in [0, \pi]} |I_{2,4}(e^{i\theta})| \in \resultnumber{[0, 964]},\quad
    \max_{\theta \in [0, \pi]} |I_{3,3}(e^{i\theta})| \in \resultnumber{[0, 340]} \text{ and }
    \max_{\theta \in [0, \pi]} |I_{3,4}(e^{i\theta})| \in \resultnumber{[0, 970]}.
  \end{equation*}
  This implies the required bounds.

  Computing
  \begin{equation*}
    I_{K,l}(z) = \int_{0}^{z}\frac{1}{t}K_{4}(N, z, t)V_{l}(t)\ dt
  \end{equation*}
  follows a similar procedure, though more work is required for
  computing \(K_{4}\). We split the integral as
  \begin{equation*}
    I_{K,l}(z) = \int_{0}^{az}\frac{1}{t}K_{4}(N, z, t)V_{l}(t)\ dt
    + \int_{az}^{z}\frac{1}{t}K_{4}(N, z, t)V_{l}(t)\ dt
    = I_{K,l,1}(z) + I_{K,l,2}(z),
  \end{equation*}
  with \(a = \codenumber{10^{-4}}\) when \(l = 1\) and \(a = 0.25\)
  when \(l > 1\).

  For \(I_{K,l,1}\) we have as before that \(V_{l}(t) / t\) is bounded
  near \(t = 0\), and we have
  \begin{equation*}
    |I_{K,l,1}(z)| \leq D_{l,a} |z| \int_{0}^{a} |K_{4}(N, z, sz)|\ ds.
  \end{equation*}
  To bound the remaining integral we make use of
  Lemma~\ref{lemma:K4-bounds}, from which we have
  \begin{align*}
    |K_{4}(N, z, t)| & \leq L_0(N)
    + L_1(N) \left|\log\left(\frac{t}{z}\right)\right|
    + L_2(N) \left|\log\left(\frac{t}{z}\right)\right|^2
    + L_3(N) \left|\log\left(\frac{t}{z}\right)\right|^3
    + L_4(N) \left|\log\left(\frac{t}{z}\right)\right|^4 \\
    & \quad + L_6(N) \left|\log\left(\frac{t}{z}\right)\right|^6,
  \end{align*}
  with explicit bounds for the coefficients given in
  Lemma~\ref{lemma:K4-bounds-constants}. Since \(0 < t / z < 1\), by changing variables to $s=t/z$ and integrating, we get that the integral is bounded by
  \begin{multline*}
    \int_{0}^{a} |K_{4}(N, z, sz)|\ ds
    \leq L_{0}(N)\int_{0}^{a} \ ds
    + L_{1}(N)\left|\int_{0}^{a} \log\left(s\right)\ ds\right|
    + L_{2}(N)\left|\int_{0}^{a} \log^{2}\left(s\right)\ ds\right|\\
    + L_{3}(N)\left|\int_{0}^{a} \log^{3}\left(s\right)\ ds\right|
    + L_{4}(N)\left|\int_{0}^{a} \log^{4}\left(s\right)\ ds\right|
    + L_{6}(N)\left|\int_{0}^{a} \log^{6}\left(s\right)\ ds\right|.
  \end{multline*}
  These integrals can then be computed explicitly, see
  Lemma~\ref{lemma:basic-integrals}.

  For \(I_{K,l,2}\) the process is exactly the same as for
  \(I_{k,l,2}\), the only difference being that we take
  \(b = \codenumber{1 - 10^{-6}}\) (for all values of \(l\)) and we
  have to compute enclosures of \(K_{4}(N, z, t)\) instead of
  \(d_{k}(z, t)\). These enclosures are computed using Taylor models,
  similar to how \(T_{6}\) was handled in
  Lemma~\ref{lemma:bounds_and_splitting:first_term_J0:upper-bounds-for-terms-from-g-function}.
  Taylor models are discussed in Appendix~\ref{sec:taylor-models},
  which also contains more details about how to compute \(K_{4}\)
  using them.

  With this we get the following enclosure for \(I_{K,l}\)
  \begin{align*}
    \max_{\theta \in [0, \pi]} |I_{K,1}(e^{i\theta})| &\in \resultnumber{[27.039, 39.953]},\\
    \max_{\theta \in [0, \pi]} |I_{K,2}(e^{i\theta})| &\in \resultnumber{[0, 292]},\\
    \max_{\theta \in [0, \pi]} |I_{K,3}(e^{i\theta})| &\in \resultnumber{[0, 620]},\\
    \max_{\theta \in [0, \pi]} |I_{K,4}(e^{i\theta})| &\in \resultnumber{[0, 4900]},
  \end{align*}
  which implies the required bounds.
\end{proof}

As a consequence, we get the following bound for the error:
\begin{corollary}
  \label{cor:bounds_and_splitting:second_term_integral_K:error-term-from-integral}
  We can write
  \begin{equation*}
    \frac{1}{N} \real \int_0^z \frac{1}{t} V(t) K(z,t) dt =
    \frac{1}{N^2} G_2(z) + \frac{1}{N^3} G_3(z) + \frac{1}{N^4} G_4(z) +
    \frac{1}{N^5} G_5(z) + E_K(N, z),
  \end{equation*}
  where for \(|z| = 1\) and \(N \geq N_{0} = \codenumber{64}\) we have
  the bound
  \begin{align*}
    |E_K(N, z)|
    &\leq \frac{1}{N^6}  \left( \left| \real\int_0^z  \frac{1}{t} d_1(z,t)V_4(t) dt \right| + \left|\real \int_0^z  \frac{1}{t} d_2(z,t)V_3(t) dt \right| \right.\\
    &\quad\quad\quad\quad + \left|\real \int_0^z  \frac{1}{t} d_3(z,t)V_2(t) dt \right|  +  \left|\real \int_0^z  \frac{1}{t} V_1(t)K_4(N,z,t) dt \right|   \\
    & \quad\quad\quad\quad + \frac{1}{N} \left(  \left|\real \int_0^z  \frac{1}{t} d_2(z,t) V_4(t) dt \right| + \left|\real \int_0^z  \frac{1}{t} d_3(z,t) V_3(t) dt \right| + \left|\real \int_0^z  \frac{1}{t} V_2(t) K_4(N,z,t)  dt \right| \right) \\
    & \quad\quad\quad\quad + \frac{1}{N^2} \left(  \left|\real \int_0^z  \frac{1}{t} d_3(z,t) V_4(t) dt \right| + \left|\real \int_0^z  \frac{1}{t} V_3(t) K_4(N,z,t)  dt \right| \right) \\
    & \quad\quad\quad\quad \left.+ \frac{1}{N^3} \left(  \left|\real \int_0^z  \frac{1}{t} V_4(t) K_4(N,z,t)  dt \right| \right) \right) \\
    &\leq \frac{1}{N^{6}}\left(
      C_{I,1,4} + C_{I,2,3} + C_{I,3,2} + C_{I,K,1}
      + \frac{C_{I,2,4} + C_{I,3,3} + C_{I,K,2}}{N}
      + \frac{C_{I,3,4} + C_{I,K,3}}{N^{2}}
      + \frac{C_{I,K,4}}{N^{3}}
      \right).
  \end{align*}
\end{corollary}

Finally, returning to Equation~\eqref{eq:supreme_for_which_we_want_an_upper_bound}, we have that for $|z|=1$
\begin{multline}
\label{eq:upper_bound_error}
    \left | a_0 J_0 \left ( \rho^{1/2}  |F_N(z)| \right )  + \frac{c_N}{N} \real \int_0^{z} V(z) K(z,t) \frac{dt}{t} \right | = \\
    \left |  a_0 \left (  \frac{g'(1) b_2(z)}{N^2} + \frac{g'(1)b_3(z)}{N^3} + \frac{g'(1)b_4(z) + \frac{g''(1)}{2} b_2^2(z)}{N^4} + \frac{g'(1)b_5(z) + g''(1)b_2(z)b_3(z)}{N^5} + E_{J_0}(N, z)  \right ) \right. \\
    \left. +  c_N \left (  \frac{G_2(z)}{N^2} + \frac{G_3(z)}{N^3} + \frac{G_4(z)}{N^4} + \frac{G_5(z)}{N^5} +  E_K(N, z)  \right ) \right | \, .
\end{multline}
Our own implementation of the algorithm described in~\cite{Berghaus-Georgiev-Monien-Radchenko:dirichlet-eigenvalues-regular-polygons} yields that all terms up to order $O(N^{-5})$ cancel out when $|z| = 1$.
\addjoana{
The algorithm is based on encoding the iterated integrals defining multiple polylogarithms by words in two letters. Operations between these polylogarithms, such as multiplication or integration, can then be expressed via concatenation or shuffling of words. In this representation, recurrence relations appear, originating a construction algorithm that produces linear combinations of multiple polylogarithms whose coefficients are in polynomials in $\lambda$ that allow for zeta values.
Then, the computed quantities are compared to the ansatz of $V_i$ and $\lambda_{app}$ from Definition~\ref{def:u_app-lambda_app}. Matching quantities give the cancellation in~\eqref{eq:upper_bound_error}. In other words, the algorithm verifies that the candidates for $V_i$ and $\lambda_{app}$ are indeed the exact expressions that give cancellation of order up to $N^{-5}$ in the expansion in $1/N$.
}
Thus,
\begin{equation*}
  |v_{app}(z)|
  \leq |a_0| \left | E_{J_0}(N, z)\right | + |c_N|  \left |E_K(N, z)   \right | \, .
\end{equation*}

All that we are left to do is estimate $c_N$ and $a_0$, for which we
have the following lemma:
\begin{lemma}\label{lemma:c_N-monotone}
    The constant $c_N = \sqrt{\dfrac{\Gamma(1-\frac{1}{N})^2 \Gamma(1+\frac{2}{N})}{ \Gamma(1+\frac{1}{N})^2 \Gamma(1-\frac{2}{N}) }}$, is monotonically increasing with respect to $N$, its limit is $1$ as $N$ tends to infinity \addjoana{and there exists a finite constant $C_c>0$ independent of $N$ such that for all $N\geq 3$
    $$
    1-c_N \leq \frac{C_c}{N^3}.
    $$}
\end{lemma}
\begin{proof}
    We start by writing an expression for $\frac{d}{dN} \log(c_N)$ as follows:
    \begin{equation*}
        \frac{d}{dN} \log(c_N) = \frac{1}{c_N}\frac{dc_N}{dN} = \frac{1}{N^2} \left ( -\psi^0\left(1-\frac{2}{N}\right) + \psi^0\left(1-\frac{1}{N}\right) + \psi^0\left(1+\frac{1}{N}\right) - \psi^0\left(1+\frac{2}{N}\right)  \right ) \, ,
    \end{equation*}
    where $\psi^0(z)$ is the digamma function, i.e. the logarithmic derivative of the gamma function. Using the integral representation of the digamma function, $\psi^0(z) = \int_0^1 \dfrac{1-t^{z-1}}{1-t} dt - \gamma_0$ for $Re(z) > 0$ and where $\gamma_0$ stands for the Euler constant, then
    \begin{equation*}
        \frac{dc_N}{dN} =  \frac{c_N}{N^2} \left ( \int_0^1 \frac{t^{\frac{-2}{N}} - t^{\frac{-1}{N}} - t^{\frac{1}{N}} + t^{\frac{2}{N}}}{1-t} dt  \right ) = \frac{c_N}{N^2}\left ( \int_0^1 \frac{t^{\frac{-2}{N}} (t^{\frac{1}{N}} - 1)^2 (t^{\frac{2}{N}} + t^{\frac{1}{N}} + 1)}{1-t} dt  \right ) > 0 ,
    \end{equation*}
    as $c_N$ is also positive for all $N\geq3$.

    Finally, taking the limit as $N$ goes to infinity,
    $$
    \lim_{N \to \infty} c_N = \lim_{N \to \infty} \sqrt{\dfrac{\Gamma(1-\frac{1}{N})^2 \Gamma(1+\frac{2}{N})}{ \Gamma(1+\frac{1}{N})^2 \Gamma(1-\frac{2}{N}) }} = \sqrt{\dfrac{\Gamma(1)^2 \Gamma(1)}{ \Gamma(1)^2 \Gamma(1) }} = 1 \, .
    $$
    \addjoana{
    For the quantitative expression, we start from the last equality for $\frac{dc_N}{dN}$. We bound $1-t^{1/N} \leq \frac{-\log t}{N}$, for $0<t<1$, using $1-e^{-u} \leq u$ with $u=-\log t/N \leq 0$. Moreover, since $N\geq 3$, then $t^{-2/N} \leq t^{-2/3}$ and $t^{2/N} + t^{1/N} + 1 \leq 3$.
    Then, using also $0 < c_N \leq 1$, we obtain
    $$
    1-c_N = \int_N^{\infty} c_M' dM \leq \int_N^{\infty} \left(\int_0^1\frac{t^{-2/3}(\log t)^2}{1-t} dt \right) \frac{3dM}{M^4} = \frac{C_c}{N^3},
    $$
    where $C_c = \int_0^1\frac{t^{-2/3}(\log t)^2}{1-t} dt < \infty$ since the integral converges at both endpoints.
    }
\end{proof}

We can now get the following numerical bounds:
\begin{lemma}
  \label{lemma:bounds_and_splitting:bound_uniform_with_N-for_a0_and_cN}
  A bound for all $N \geq 3$ for $c_N$ and $a_0$ is
  \begin{equation*}
    |c_N| \leq 1 \, ,
    |a_0| \leq C_{a_{0}} \, ,
  \end{equation*}
  with \(C_{a_{0}} = \codenumber{0.801}\).
\end{lemma}
\begin{proof}
  The bound for $c_N$ follows immediately from the above lemma. The
  bound for $a_0$ follows from the enclosure
  \begin{equation}
    |a_0| \leq \frac{1}{\lambda^{1/2}J_1(\lambda^{1/2})} \in \resultnumber{[0.8009873485 \pm 3.6 \cdot 10^{-11}]}.
  \end{equation}
\end{proof}

Putting all results from this subsection together, we have the following Lemma:
\begin{lemma}
  \label{lemma:upper-bound-on-the-defect_in-the-boundary}
  Let $u_{app}$ be the function defined in
  Definition~\ref{def:u_app-lambda_app}, then for
  \(N \geq N_{0} = \codenumber{64}\) we have the bound
  \begin{equation*}
    \sup_{x \in \partial \mathcal{P}_N} |u_{app}(x)| \leq \varepsilon(N),
  \end{equation*}
  with
  \begin{multline}\label{eq:varepsilon-N}
    \varepsilon(N) = C_{a_{0}}\frac{1}{N^6} \left(
      |g'(1)|C_{T,6}
      + \frac{|g''(1)|}{2}C_{b,3}^{2}
      + |g''(1)|C_{b,2}C_{T,4}
      + \frac{1}{6}C_{T,2}^{3}C_{g'''}
      + \frac{|g''(1)|}{N}C_{b,3}C_{T,4}
      + \frac{|g''(1)|}{2 N^2}C_{T,4}^{2}
    \right)\\
    + \frac{1}{N^{6}}\left(
      C_{I,1,4} + C_{I,2,3} + C_{I,3,2} + C_{I,K,1}
      + \frac{C_{I,2,4} + C_{I,3,3} + C_{I,K,2}}{N}
      + \frac{C_{I,3,4} + C_{I,K,3}}{N^{2}}
      + \frac{C_{I,K,4}}{N^{3}}
    \right).
  \end{multline}
\end{lemma}

\begin{proof}
  We have
  \begin{equation*}
    \sup_{x \in \partial \mathcal{P}_N} |u_{app}(x)|
    = \sup_{z \in \partial \mathbb{D}} |v_{app}(z)|
    \leq \sup_{z \in \partial \mathbb{D}} |a_0||E_{J_0}(N, z)| + |c_N||E_K(N, z)|.
  \end{equation*}
  The bound is then a direct consequence of
  Corollaries~\ref{cor:bounds_and_splitting:first_term_J0:error-bound-for-g-function}
  and~\ref{cor:bounds_and_splitting:second_term_integral_K:error-term-from-integral}
  and
  Lemma~\ref{lemma:bounds_and_splitting:bound_uniform_with_N-for_a0_and_cN}.
\end{proof}

\subsection{Lower bounds on $\|u_{app}\|_{L^2(\mathcal{P}_{N})}$}
\label{subsec:lower_bounds_on_L2_norm}
The goal of this subsection is to compute a lower bound for $\|u_{app}\|_{L^2(\mathcal{P}_N)}^2$.
Recall the expression for $u_{app}$ given in Definition~\ref{def:u_app-lambda_app}, from which we get
\begin{equation*}
    \begin{split}
    u_{app}(x)^2
    =&  \,  a_0^2 J_0^2\left(\sqrt{\lambda_{app}} |x|\right) + \\
    &     \left( \real \int_0^x U(t) J_0\left(\sqrt{\lambda_{app} \bar{x} (x-t)}\right) dt \right) \cdot
    \left ( 2 a_0 J_0\left(\sqrt{\lambda_{app}} |x|\right) +  \real \int_0^x U(t) J_0\left(\sqrt{\lambda_{app} \bar{x} (x-t)}\right) dt \right )
    \end{split}
\end{equation*}
We will treat the first term as the main one and the second one as an error. Therefore, we want to bound
\begin{multline}
\label{eq:lower-bound-inequality-for-|u_app|^2}
    |u_{app}(x)|^2 \geq \, a_0^2 J_0^2\left(\sqrt{\lambda_{app}} |x|\right)\\
    - \left | \real \int_0^x U(t) J_0\left(\sqrt{\lambda_{app} \bar{x} (x-t)}\right) dt  \right |
    \cdot \left |  2 a_0 J_0\left(\sqrt{\lambda_{app}}|x|\right) + \real \int_0^x U(t) J_0\left(\sqrt{\lambda_{app} \bar{x} (x-t)}\right) dt \right | \, .
\end{multline}

Integrating over a polygon leads to more difficult estimates due to their dependence with $N$. At the expense of a worse estimate, we will lower bound the norm on the polygon by the norm on a disk contained in all $N$-sided polygons for $N \geq N_0$, simplifying the integral due to the radiality of the main term.
Explicit bounds for the radii of the inscribed and circumscribed circles are found in the following Lemma, illustrated in Figure~\ref{fig:1st-quadrant-28-gon}.

\begin{lemma}
  \label{lemma-disk}
  Let $N \geq N_0 = \codenumber{64}$. Then the regular $N$-sided polygon
  of area $\pi$ contains a disk of radius
  \(R_{\text{inner}} = \codenumber{0.95}\) and is contained in a disk of
  radius \(R_{\text{outer}} = \codenumber{1.01}\).
  \begin{figure}[h]
    \centering
    \begin{tikzpicture}[scale=6]
      \def\n{28}
      \def\nn{7}
      \def\R{1.004212173}
      \begin{scope}
        \clip (0,0) rectangle (1.3,1.3);

        % Outer circle
        \draw[blue] (0,0) circle (1.01);

        % Regular polygon outline
        \draw[black, thin]
        ({\R*cos(0)}, {\R*sin(0)})
        \foreach \i in {1,2,...,\n} {
          -- ({\R*cos(\i*360/\n)}, {\R*sin(\i*360/\n)})
        } -- cycle;

        % Fill polygon
        \fill[gray!15, opacity=0.7]
        ({\R*cos(0)}, {\R*sin(0)})
        \foreach \i in {1,2,...,\n} {
          -- ({\R*cos(\i*360/\n)}, {\R*sin(\i*360/\n)})
        } -- cycle;

        % Inner circle
        \draw[red, very thin] (0,0) circle (0.95);
      \end{scope}

      % Mark polygon vertices
      \foreach \i in {0,1,...,\nn} {
        \fill[black] ({\R*cos(\i*360/\n)}, {\R*sin(\i*360/\n)}) circle (0.008);
      }

      % Origin
      \fill[black] (0,0) circle (0.008);

      % Axes (drawn outside clip to be fully visible)
      \draw[->] (-0.1,0) -- (1.2,0) node[right] {$x$};
      \draw[->] (0,-0.1) -- (0,1.2) node[above] {$y$};

      % Labels
      \node[blue] at (0.7,0.9) {$r = \codenumber{1.01}$};
      \node[red] at (0.7,0.45) {$r = \codenumber{0.95}$};
    \end{tikzpicture}
    \caption{A regular $28$-gon of area $\pi$ centered at the origin shown only in the first quadrant. The red circle of radius $\codenumber{0.95}$ lies entirely inside the polygon, and the blue circle of radius $\codenumber{1.01}$ fully contains the polygon. Neither circle touches the polygon.
    }
    \label{fig:1st-quadrant-28-gon}
  \end{figure}
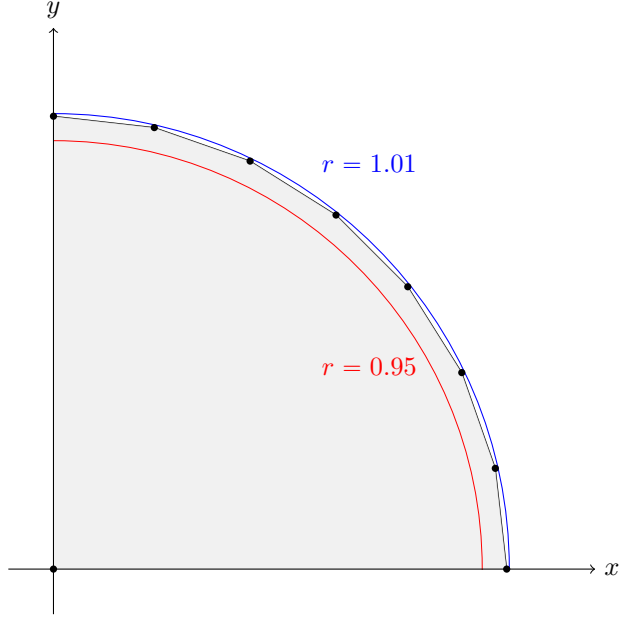
\end{lemma}

\begin{proof}
  The area, \(A_{N}\), and the inradius, \(I_{N}\), of a regular
  polygon with \(N\) sides is related as
  \begin{equation*}
    A_{N} = NI_{N}^2\tan\left(\pi/N\right).
  \end{equation*}
  In our case \(A_{N} = \pi\), giving us
  \begin{equation*}
    I_{N} = \sqrt{\frac{\pi}{N \tan \left(\pi/N\right)}}.
  \end{equation*}
  The function $I_N$ is increasing in $N$.
  Thus, it suffices to
  verify that \(I_{N_{0}} > R_{\text{inner}}\), which follows from the
  enclosure
  \begin{equation*}
    I_{N_{0}} \in \resultnumber{[0.99960 \pm 1.75 \cdot 10^{-6}]}.
  \end{equation*}

  In a similar way, the area and the circumradius, \(C_{N}\), is
  related by
  \begin{equation*}
    A_{N} = \frac{N}{2} \sin\left(\frac{2\pi}{N}\right) C_{N}^2,
  \end{equation*}
  giving (with \(A_{N} = \pi\))
  \begin{equation*}
    C_{N} = \sqrt{\frac{\pi}{\frac{N}{2} \sin\left(\frac{2\pi}{N}\right)}}.
  \end{equation*}
  In this case, the function $C_N$ is decreasing in $N$.
  Thus, it suffices to verify \(C_{N_{0}} < R_{\text{outer}}\), which follows
  from the enclosure
  \begin{equation*}
    C_{N_{0}} \in \resultnumber{[1.0008 \pm 3.78 \cdot 10^{-6}]}.
  \end{equation*}

  Figure~\ref{fig:8-polygon-radius-inradius} illustrates how the previous formulas are derived: by decomposing the polygon into triangles formed by connecting the center to the point of contact of the circumferences.

  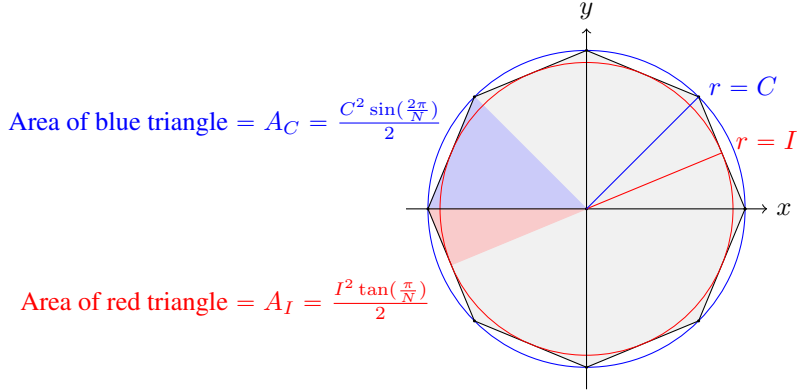
\begin{figure}[h]
    \centering
    \begin{tikzpicture}[scale=2]
      \def\n{8}
      \def\I{0.9736834439}
      \def\R{1.053907365}
      \def\C{1.053907365}
      \begin{scope}
        % Fill polygon
        \fill[gray!15, opacity=0.7]
        ({\R*cos(0)}, {\R*sin(0)})
        \foreach \i in {1,2,...,\n} {
          -- ({\R*cos(\i*360/\n)}, {\R*sin(\i*360/\n)})
        } -- cycle;

        % Blue triangle (second quadrant)
        \filldraw[blue!30, opacity=0.6]
        (0,0) --
        ({\C*cos(180)}, {\C*sin(180)}) --
        ({\C*cos(180-360/\n)}, {\C*sin(180-360/\n)}) -- cycle;
        \node[blue] at (-2.4,0.6)
        {$\text{Area of blue triangle} = A_C = \frac{C^2 \sin(\frac{2\pi}{N})}{2}$};

        % Red triangle (third quadrant)
        \filldraw[red!30, opacity=0.6]
        (0,0) --
        ({\C*cos(180)}, {\C*sin(180)}) --
        ({\I*cos(180+180/\n)}, {\I*sin(180+180/\n)}) -- cycle;
        % \node[red] at (-1.5,-0.6)
        \node[red] at (-2.4,-0.6)
        {$\text{Area of red triangle} = A_I = \frac{I^2 \tan(\frac{\pi}{N})}{2}$};

        % Outer circle
        \draw[blue] (0,0) circle (\C);

        % Regular polygon outline
        \draw[black, thin]
        ({\R*cos(0)}, {\R*sin(0)})
        \foreach \i in {1,2,...,\n} {
          -- ({\R*cos(\i*360/\n)}, {\R*sin(\i*360/\n)})
        } -- cycle;

        % Inner circle
        \draw[red, very thin] (0,0) circle (\I);
      \end{scope}

      % Mark polygon vertices
      \foreach \i in {0,1,...,\n} {
        \fill[black] ({\R*cos(\i*360/\n)}, {\R*sin(\i*360/\n)}) circle (0.01);
      }

      % Origin
      \fill[black] (0,0) circle (0.01);

      % Axes
      \draw[->] (-1.2,0) -- (1.2,0) node[right] {$x$};
      \draw[->] (0,-1.2) -- (0,1.2) node[above] {$y$};

      % Inscribed radius
      \draw[red, thin] (0,0) -- ({\I*cos(180/\n)}, {\I*sin(180/\n)});

      % Circumscribed radius
      \draw[blue, thin] (0,0) -- ({\C*cos(360/\n)}, {\C*sin(360/\n)});

      % Labels
      \node[blue, right] at (0.75,0.82) {$r = C$};
      \node[red, right] at (0.93,0.46) {$r = I$};
    \end{tikzpicture}
    \caption{An octagon ($8$-gon) of area $\pi$ centered at the origin. The polygon contains an inscribed disk of radius $I$ (red circle) and is contained within a circumscribed disk of radius $C$ (blue circle). The figure also highlights the areas of the triangles, into which the polygon is decomposed, to derive two equivalent formulas for its total area.}
    \label{fig:8-polygon-radius-inradius}
  \end{figure}
\end{proof}

By virtue of Lemma~\ref{lemma-disk}, we can obtain a lower bound of $\|u_{app}\|_{L^2(\mathcal{P}_N)}$ by computing a lower bound in $L^2(\mathbb{D}_{R_{\text{inner}}})$. We start by bounding the error term $ | \real \int_0^x U(t) J_0(\sqrt{\lambda_{app} \bar{x} (x-t)}) dt|$ in $L^{\infty}(\mathbb{D}_{R_{\text{inner}}})$.

\begin{align*}
\left | \real \int_0^x U(t) J_0 \left(\sqrt{\lambda_{app} \bar{x} (x-t)} \right) dt  \right |
& \leq
\left | \int_0^x U(t) J_0 \left(\sqrt{\lambda_{app} \bar{x} (x-t)} \right) dt  \right |
=
\left | \int_0^1 U(s x) J_0 \left(\sqrt{\lambda_{app} \bar{x} (x-s x)} \right) x ds  \right |
\\
& \leq
\int_0^1 | U(s x) | \left |J_0 \left(\sqrt{\lambda_{app} \bar{x} (x-s x)} \right) \right |  |x| ds
\leq
|x| \int_0^1 | U(s x) |  ds \, ,
\end{align*}
where in the last step we have used that the argument inside the Bessel function is real, and thus, the Bessel function can be bounded by $1$.

We use the relationship between $U(t)$ and $V(s)$ stated in Definition~\ref{def:u_app-lambda_app}, giving us $ V(s^N) c_N = U(f_N(s)) s f_N'(s)$. From this we get
\begin{equation*}
  |x| \int_0^1 | U(s x) |  ds
  =
  |x| \int_0^1 \left | \frac{V((f_N^{-1}(s x))^N) c_N}{f_N'(f_N^{-1}(s x)) f_N^{-1}(s x)} \right |  ds
  \leq R_{\text{inner}} \left\| \frac{V((f_N^{-1}( x))^N) c_N}{f_N'(f_N^{-1}( x)) f_N^{-1}( x)} \right\|_{L^{\infty}\left(\mathbb{D}_{R_{\text{inner}}}\right)}
  ,
\end{equation*}
where $f_N^{-1}:\mathcal{P}_N\to\mathbb{D}$ denotes the analytic inverse of the Schwarz Christoffel map $f_N:\mathbb{D}\to\mathcal{P}_N$, which is biholomorphic (i.e. bijective holomorphic function whose inverse is also holomorphic).
To further reduce the problem, note that the $V(z)$ function can be divided into its smaller parts, giving rise to the following Lemma:

\begin{lemma}
  \label{lemma:bounds_and_splitting:lower_bound:bounds-independent-of-beta-x_in_D-0.95-N-for-V_i}
  Let $V_1(z), V_2(z), V_3(z), V_4(z)$ be the functions defined in
  Definition~\ref{def:u_app-lambda_app}. Then, for all
  $x \in \mathbb{D}_{R_{\text{inner}}}$ and
  $N \geq N_0 = \codenumber{64}$, we have the following bounds:
  \begin{align*}
    \left | \frac{V_1\left(f_N^{-1}(x)^N\right) c_N}{f_N'(f_N^{-1}(x)) f_N^{-1}(x)} \right | \leq C_{V,1} \, , \quad &
    \left | \frac{V_2\left(f_N^{-1}(x)^N\right) c_N}{f_N'(f_N^{-1}(x)) f_N^{-1}(x)} \right | \leq C_{V,2} \, , \\
    \left | \frac{V_3\left(f_N^{-1}(x)^N\right) c_N}{f_N'(f_N^{-1}(x)) f_N^{-1}(x)} \right | \leq C_{V,3} \, , \quad &
    \left | \frac{V_4\left(f_N^{-1}(x)^N\right) c_N}{f_N'(f_N^{-1}(x)) f_N^{-1}(x)} \right | \leq C_{V,4} \, ,
  \end{align*}
  where
  \begin{equation*}
    C_{V,1} = \codenumber{0.1},\quad
    C_{V,2} = \codenumber{0.05},\quad
    C_{V,3} = \codenumber{0.08},\quad
    C_{V,4} = \codenumber{1}.
  \end{equation*}
\end{lemma}
\begin{proof}
  We start by giving bounds for \(f_N^{-1}(x)\). The goal is to find a
  radius \(R < 1\) such that for
  \(x \in \mathbb{D}_{R_{\text{inner}}}\) we have
  \(f_N^{-1}(x) \in \mathbb{D}_{R}\). Since \(f_N\) is continuous it
  suffices to find \(R\) such that the boundary of \(\mathbb{D}_{R}\)
  is mapped by \(f_N\) to a curve lying strictly outside of
  \(\mathbb{D}_{R_{\text{inner}}}\). It then follows from continuity
  that the image of \(\mathbb{D}_{R_{\text{inner}}}\) under \(f_N^{-1}\)
  lies strictly inside \(\mathbb{D}_{R}\). We claim that taking
  \(R = \codenumber{0.951}\) works, i.e that for
  \(x \in \mathbb{R}_{R_{\text{inner}}}\) we have
  \(f_N^{-1}(x) \in \mathbb{D}_{R} = \mathbb{D}_{\codenumber{0.951}}\).
  To prove this, it suffices to show that for \(|z| = R\) we have
  \begin{equation*}
    |f_N(z)| > R_{\text{inner}} = \codenumber{0.95}.
  \end{equation*}
  Recall that
  \begin{equation*}
    |f_N(z)| = c_{N}R\left|{}_2F_1 \left(\frac{2}{N}, \frac{1}{N}, 1 + \frac{1}{N}; z^N\right)\right|.
  \end{equation*}
  For \(N \geq N_{0}\) we have \(N^{-1} \in (0, N_{0}^{-1}]\) and
  \(z^{N} \in \mathbb{D}_{R^{N_{0}}}\), enclosing the \({}_{2}F_{1}\)
  function in this set gives us
  \begin{equation*}
    \left|{}_2F_1 \left(\frac{2}{N}, \frac{1}{N}, 1 + \frac{1}{N}; z^{N}\right)\right|
    \in \resultnumber{[1.000 \pm 2.01 \cdot 10^{-5}]},
  \end{equation*}
  valid for all \(N \geq N_{0}\) and \(|z| = R\). Combining this with
  Lemma~\ref{lemma:c_N-monotone} gives us that
  \begin{equation*}
    |f_N(z)|
    \geq c_{N_{0}}R\left|{}_2F_1 \left(\frac{2}{N}, \frac{1}{N}, 1 + \frac{1}{N}; z^N\right)\right|
    \in \resultnumber{[0.9510 \pm 2.78 \cdot 10^{-5}]} > \codenumber{0.95} = R_{\text{inner}},
  \end{equation*}
  which is exactly what we needed to verify.

  By the above it hence suffices to bound
  \begin{equation*}
    \left|\frac{V_l\left(z^N\right) c_N}{z f_N'(z)}\right|
  \end{equation*}
  for \(z \in \mathbb{D}_{R}\). Using that
  \(f_N'(z) = c_N (1-z^N)^{-2/N}\) we can write it as
  \begin{equation*}
    \left|\frac{V_l\left(z^N\right) c_N}{z f_N'(z)}\right| = \left|\frac{V_l\left(z^N\right)}{z}(1-z^N)^{2/N}\right|
  \end{equation*}
  For the second factor we get the upper bound
  \begin{equation*}
    \left|(1-z^N)^{2/N}\right| \leq (1 + R^{N})^{2/N} \leq (1 + R^{N_0})^{2/N_0}.
  \end{equation*}
  From Lemma~\ref{lemma:V_log_bounds} in Appendix~\ref{app:aux_lemmas}
  we have
  \begin{equation*}
    \left|\frac{V_{l}(z^{N})}{z}\right| \leq D_{l,R^{N_{0}}}|z^{N-1}| \leq D_{l,R^{N_{0}}}R^{N_{0}-1}.
  \end{equation*}
  This gives us
  \begin{equation*}
    \left|\frac{V_{l}(z) c_N}{z f_N'(z)}\right| \leq D_{l,R^{N_{0}}}R^{N_{0}-1}(1+R^{N_0})^{2/N_0}.
  \end{equation*}
  With this we can compute the bounds
  \begin{align*}
    \left|\frac{V_{1}(z) c_N}{z f_N'(z)}\right| &\leq D_{1,R^{N_{0}}}R^{N_{0}-1}(1+R^{N_0})^{2/N_0} \in \resultnumber{[0.086258 \pm 4.86 \cdot 10^{-7}]},\\
    \left|\frac{V_{2}(z) c_N}{z f_N'(z)}\right| &\leq D_{2,R^{N_{0}}}R^{N_{0}-1}(1+R^{N_0})^{2/N_0} \in \resultnumber{[0.041987 \pm 3.53 \cdot 10^{-7}]},\\
    \left|\frac{V_{3}(z) c_N}{z f_N'(z)}\right| &\leq D_{3,R^{N_{0}}}R^{N_{0}-1}(1+R^{N_0})^{2/N_0} \in \resultnumber{[0.079409 \pm 3.14 \cdot 10^{-7}]},\\
    \left|\frac{V_{4}(z) c_N}{z f_N'(z)}\right| &\leq D_{4,R^{N_{0}}}R^{N_{0}-1}(1+R^{N_0})^{2/N_0} \in \resultnumber{[0.95572 \pm 3.87 \cdot 10^{-7}]},
  \end{align*}
  which gives us the result.
\end{proof}

Therefore, we can conclude the following $L^{\infty}(\mathbb{D}_{R_{\text{inner}}})$-norm bound for the integral part
\begin{align*}
    \left | \real \int_0^x U(t) J_0 \left(\sqrt{\lambda_{app} \bar{x} (x-t)} \right) dt  \right |
    & \leq \frac{C_{V,1}}{N} + \frac{C_{V,2}}{N^2} + \frac{C_{V,3}}{N^3} + \frac{C_{V,4}}{N^4} =: E_I(N) \, .
\end{align*}
And thus,
\begin{equation}
\label{eq:lower-bound-u_app^2}
    |u_{app}(x)|^2 \geq a_0^2J_0^2\left(\sqrt{\lambda_{app}}|x|\right) - 2 a_0E_I(N) \left|J_0\left(\sqrt{\lambda_{app}}|x|\right)\right| - E_I(N)^2 \, .
\end{equation}
By integrating over the disk $\mathbb{D}_{R_{\text{inner}}}$ and applying a change of variables to polar coordinates, we can compute the following two integrals:
\begin{lemma}
\label{lemma:Bessel-integrals}
  The following integral identities hold for $R>0$ and $R < \dfrac{j_{0,1}}{\sqrt{\lambda_{app}}}$:
  \begin{equation*}
    \int_0^R J_0\left(\sqrt{\lambda_{app}}\, r\right)^2r dr =
    \frac{R^2}{2} \left(
      J_0\left(R \sqrt{\lambda_{app}}\right)^2
      + J_1\left(R \sqrt{\lambda_{app}}\right)^2
    \right) \, ,
  \end{equation*}
  \begin{equation*}
    \int_0^R J_0(\sqrt{\lambda_{app}}\, r) \, r dr =
    \frac{R}{\sqrt{\lambda_{app}}} J_1(R\sqrt{\lambda_{app}}) \, .
  \end{equation*}
\end{lemma}

We can now get a lower bound for $L^2$ norm of $u$.
\begin{corollary}
\label{cor:lower-bound-for-L2-norm-k(N)}
  For \(N \geq N_{0} = 64\) we have
  \begin{equation*}
    \|u_{app}\|_{L^2(\mathcal{P}_N)}^2 \geq \eta(N),
  \end{equation*}
  with
  \begin{multline}\label{eq:eta-N}
    \eta(N) =\pi a_0^2 R_{\text{inner}}^2 \left(
      J_0\left(R_{\text{inner}} \sqrt{\lambda_{app}}\right)^2
      + J_1\left(R_{\text{inner}} \sqrt{\lambda_{app}}\right)^2
    \right)\\
    - 4\pi a_0E_I(N) \frac{R_{\text{inner}}}{\sqrt{\lambda_{app}}}
      J_1(R_{\text{inner}}\sqrt{\lambda_{app}})
    - \pi R_{\text{inner}}^2 E_I(N)^2.
  \end{multline}
\end{corollary}
\begin{proof}
    This lower bound on the $L^2$-norm of $u_{app}$ comes from
    \begin{align*}
        \|u_{app}\|_{L^2(\mathcal{P}_N)}^2 & \geq \|u_{app}\|_{L^2(\mathbb{D}_{R_{\text{inner}}})}^2 = \int_{\mathbb{D}_{R_{\text{inner}}}} \left|u_{app}(x)\right|^2 dx \\
        & \geq
          \int_{\mathbb{D}_{R_{\text{inner}}}} a_0^2 J_0\left(\sqrt{\lambda_{app}}|x|\right)^2 - 2 a_0E_I(N) \left|J_0\left(\sqrt{\lambda_{app}}|x|\right)\right| - E_I(N)^2 dx \\
        & = 2\pi a_0^2 \int_0^{R_{\text{inner}}}  J_0\left(\sqrt{\lambda_{app}}r\right)^2 r dr - 4\pi a_0E_I(N) \int_0^{R_{\text{inner}}} \left|J_0\left(\sqrt{\lambda_{app}}r\right)\right| r dr -  2\pi E_I(N)^2\int_0^{R_{\text{inner}}} r dr \\
        & = \eta(N),
    \end{align*}
    using Equation~\eqref{eq:lower-bound-u_app^2} and Lemma~\ref{lemma:Bessel-integrals} after changing to polar coordinates.
    To apply the Lemma, we need to prove that $\sqrt{\lambda_{app}} R_{\text{inner}} < j_{0,1}$, so the value of the Bessel function is positive for $0<r<R_{\text{inner}} < j_{0,1}$ as $J_0$ is decreasing in this range. This follows from the enclosures
    \begin{equation*}
      \sqrt{\lambda_{app}} R_{\text{inner}} \in \resultnumber{[2.2846 \pm 3.15 \cdot 10^{-5}]} < \resultnumber{[2.404825558 \pm 3.05 \cdot 10^{-10}]} \ni j_{0,1},
    \end{equation*}
    valid for \(N \geq N_{0}\).
\end{proof}

Note that the limit as $N\to\infty$ of $\eta(N)$ does not vanish, as the following Lemma shows:

\begin{lemma}
\label{lemma:expansion-for-eta-N}
    There exists a finite constant $C_{\eta}$ such that
    \begin{equation*}
        \left| \eta(N) - \pi R_{\text{inner}}^2\frac{J_0\left(\sqrt{\lambda} R_{\text{inner}}\right)^2 + J_1\left(\sqrt{\lambda} R_{\text{inner}}\right)^2}{\lambda J_1\left(\sqrt{\lambda}\right)^2} \right| \leq \frac{C_{\eta}}{N}.
    \end{equation*}
\end{lemma}
\begin{proof}
  Using that $\dfrac{J_1(R_{\text{inner}}\sqrt{\lambda_{app}})}{R_{\text{inner}}\sqrt{\lambda_{app}}}$ is bounded in $N$, the terms from $\eta(N)$ containing $E_I(N)$ \addjoana{have a factor $1/N$ or better}. For the first term, define $g_0(w) = J_0(w R_{\text{inner}} \sqrt{\lambda})^2$ and $g_1(w) = J_1(w R_{\text{inner}} \sqrt{\lambda})^2$. Then, the subtraction becomes
  \begin{equation*}
      \pi a_0^2 R_{\text{inner}}^2 \left( g_0\left( \frac{\sqrt{\lambda_{app}}}{\sqrt{\lambda}} \right) - g_0(1) + g_1\left( \frac{\sqrt{\lambda_{app}}}{\sqrt{\lambda}} \right) - g_1(1) \right)  ,
  \end{equation*}
  which is bounded by the Taylor expansion of the functions $g_0$ and $g_1$.
  Moreover, as $a_0 = \frac{1}{\sqrt{\lambda} J_1\left(\sqrt{\lambda}\right)} c_N$ and \addjoana{$1 \geq c_N \geq 1 - C_c/N^3$} from Lemma~\ref{lemma:c_N-monotone}, the result follows.
\end{proof}

\subsection{Conclusion}

To avoid any confusion, in this section we explicitly write out the dependency on $N$ of $\lambda_{app}$ and $u_{app}$ from Definition~\ref{def:u_app-lambda_app}, i.e. $\lambda_{app}(N) = \lambda_{app}$ and $u_{app}(N) = u_{app}$.

  In the previous subsections we have constructed an approximate
  solution $u_{app}(N): \mathcal{P}_N \to \mathbb{R}$ in
  $\mathcal{C}^2(\Omega) \cap \mathcal{C}(\overline{\Omega})$ that
  solves $\Delta u_{app}(N) + \lambda_{app}(N) u_{app}(N) = 0$ in
  $\mathcal{P}_N$ for a suitable $\lambda_{app}(N)$. Using Lemma~\ref{lem:FoxHenriciMoler} together with the bounds from
  Lemma~\ref{lemma:upper-bound-on-the-defect_in-the-boundary} and Corollary~\ref{cor:lower-bound-for-L2-norm-k(N)}
  we conclude that there exists a Dirichlet
  eigenvalue $\lambda'$ of $\mathcal{P}_N$ such that
  \begin{equation}\label{eq:lambda_app_error_bound}
    |\lambda_{app}(N) - \lambda'| \leq \lambda' \sqrt{\pi}\frac{\varepsilon(N)}{\sqrt{\eta(N)}} =: \lambda' \hat{\varepsilon}(N) .
  \end{equation}
We then have the following qualitative behavior of $\hat{\varepsilon}(N)$:
  \begin{lemma}
  \label{lemma:asymptotics-of-epsilon-hat(N)}
      The term $\hat{\varepsilon}(N)$ satisfies:
      $$
      \hat{\varepsilon}(N) = O\left(\frac{1}{N^6} \right).
      $$
  \end{lemma}
  \begin{proof}
      This result is a direct consequence of the asymptotic properties of $\varepsilon(N)$ and $\eta(N)$, see Lemmas~\ref{lemma:expansion-for-eta-N} and~\ref{lemma:upper-bound-on-the-defect_in-the-boundary}.
  \end{proof}

  The previous result give us the existence of an eigenvalue, but with a priori no guarantee of it being the first $\lambda^{(N)} = \lambda_1(\mathcal{P}_N)$. We now prove that:
  \begin{lemma}
    Let \(N \geq N_{0} = \codenumber{64}\). We have the following chain of
    inequalities:
    \begin{align*}
      \lambda'
      \leq \frac{\lambda_{app}(N)}{1 - \hat{\varepsilon}(N)}
      < \lambda_{2}(\mathbb{D}_{R_{\text{outer}}})
      \leq \lambda_2(\mathcal{P}_N).
    \end{align*}
  \end{lemma}

  \begin{proof}
    The first inequality follows from \eqref{eq:lambda_app_error_bound}. For the second inequality
    we compute the enclosures
    \begin{equation*}
      \lambda_{2}(\mathbb{D}_{R_{\text{outer}}})
      = \frac{j_{1,1}^2}{R_{\text{outer}}^2} \in \resultnumber{[14.393 \pm 3.23 \cdot 10^{-4}]}
    \end{equation*}
    and
    \begin{equation*}
      \frac{\lambda_{app}(N)}{1 - \hat{\varepsilon}(N)} \in \resultnumber{[5.783 \pm 2.93 \cdot 10^{-4}]},
    \end{equation*}
    where the last enclosure is valid for \(N \geq N_{0}\). For the
    third inequality we note that by Lemma~\ref{lemma-disk},
    \(\mathcal{P}_{N} \subsetneq \mathbb{D}_{R_{\text{outer}}}\). Hence the
    inequality follows from monotonicity of eigenvalues with respect
    to the domain.

  \end{proof}

  This shows that $\lambda' = \lambda^{(N)}$. In particular, we can
  bound
  $\lambda_{inf}^{(N)} \leq \lambda^{(N)} \leq \lambda_{sup}^{(N)}$
  for $N \geq N_0$, where
  \begin{align*}
    \lambda_{inf}^{(N)} = \dfrac{\lambda_{app}(N)}{1 + \hat{\varepsilon}(N)} \, , \quad
    \lambda_{sup}^{(N)} = \dfrac{\lambda_{app}(N)}{1 - \hat{\varepsilon}(N)} \, .
  \end{align*}
  Which further gives us
  $q_{inf}^{(N)} \leq q_N = \frac{\lambda^{(N)}}{\lambda^{(N+1)}} \leq
  q_{sup}^{(N)}$ for $N \geq N_0$, where
  \begin{align*}
    q_{inf}^{(N)} = \dfrac{\lambda_{inf}^{(N)}}{\lambda_{sup}^{(N+1)}} \, , \quad
    q_{sup}^{(N)} = \dfrac{\lambda_{sup}^{(N)}}{\lambda_{inf}^{(N+1)}} \, .
  \end{align*}

  The proof of Proposition~\ref{prop:N_large} concludes by virtue of the following Lemma:

  \begin{lemma}\label{lemma:final-inequalities}
    For all $N \geq N_0 = \codenumber{64}$ it holds that
    $\lambda_{inf}^{(N)} > \lambda_{sup}^{(N+1)}$ and
    $q_{inf}^{(N)} > q_{sup}^{(N+1)}$. Thus, for all $N \geq N_0$ we
    obtain that
    \begin{equation*}
      \lambda_{sup}^{(N)} \geq \lambda^{(N)} \geq \lambda_{inf}^{(N)} > \lambda_{sup}^{(N+1)} \geq \lambda^{(N+1)} \geq \lambda_{inf}^{(N+1)} ,
    \end{equation*}
    \begin{equation*}
      q_{sup}^{(N)} \geq q^{(N)} \geq q_{inf}^{(N)} > q_{sup}^{(N+1)} \geq q^{(N+1)} \geq q_{inf}^{(N+1)} .
    \end{equation*}
  \end{lemma}
  \begin{proof}
    We need
    to check the following conditions for $\lambda$ and $q$:
    \begin{equation*}
      \frac{\lambda_{app}(N)}{1 + \hat{\varepsilon}(N)} > \frac{\lambda_{app}(N+1)}{1 - \hat{\varepsilon}(N+1)}
      \quad \text{and} \quad
      \frac{\lambda_{app}(N)(1-\hat{\varepsilon}(N+1))}{(1+\hat{\varepsilon}(N))\lambda_{app}(N+1)} > \frac{\lambda_{app}(N+1)(1+\hat{\varepsilon}(N+2))}{(1-\hat{\varepsilon}(N+1))\lambda_{app}(N+2)} \, .
    \end{equation*}

      Let us define the functions $P_1(N)$ and $P_2(N)$ as the differences between the left and right-hand side of these inequalities, respectively. Our goal is to show that $P_1(N)>0$ and $P_2(N)>0$ for $N\geq N_0 = \codenumber{64}$.

      The strategy is to analyze these expressions by changing variables to $\nu=1/N$. Let $p_1(\nu) = P_1(1/\nu)$ and $p_2(\nu)=P_2(1/\nu)$, so that $P_1(N)=p_1(1/N)$ and $P_2(N) = p_2(1/N)$. \addjoana{It is immediate from their expressions that these functions are analytic at $\nu=0$.} We now aim to prove their positivity for $\nu\in(0,1/N_0]$.

      \addjoana{Let us start by proving that $p_1(\nu)$ vanishes up to third order at $\nu=0$ and that $p_2(\nu)$ vanishes up to fourth order.}
      We rely on the known asymptotic behavior of the involved terms. From Lemma~\ref{lemma:asymptotics-of-epsilon-hat(N)}, we know that $\hat{\varepsilon}(N)$ is of order $O(N^{-6})$. Consequently, the terms $1 \pm \hat{\varepsilon}(N)$ are $1 + O(N^{-6})$. Moreover, from~\eqref{eq:expansion_lambda_N}, we have $\lambda_{app}(N) = \lambda \left (  1 + \frac{4 \zeta(3)}{N^3} + \frac{(12 - 2\lambda) \zeta(5)}{N^5}\right )$.
      Using the Taylor expansion for $(1+u)^{-1}$, we can write
      $$\addjoana{p_1(\nu) =} P_1(N)=\lambda_{app}(N)-\lambda_{app}(N+1) + O\left( \frac{1}{N^6} \right) = \lambda 4 \zeta(3) \left( \frac{1}{N^3}-\frac{1}{(N+1)^3} \right) + O\left( \frac{1}{N^5} \right) = \addjoana{O\left( \nu^4 \right)} \, ,$$
      and
      \begin{multline*}
          \addjoana{p_2(\nu) =} P_2(N)
        =\frac{\lambda_{app}(N)}{\lambda_{app}(N+1)}-\frac{\lambda_{app}(N+1)}{\lambda_{app}(N+2)} + O\left( \frac{1}{N^6} \right) = \frac{\lambda_{app}(N)\lambda_{app}(N+2)-\lambda_{app}(N+1)^2}{\lambda_{app}(N+1)\lambda_{app}(N+2)} + O\left( \frac{1}{N^6} \right) \\
        = \left( \frac{4\zeta(3)}{N^3} + \frac{4\zeta(3)}{(N+2)^3}\right) - \left(2\frac{4\zeta(3)}{(N+1)^3}\right) + O\left( \frac{1}{N^5} \right) \\
        = 4 \zeta(3) \left( \frac{2}{N^3} - \frac{6}{N^4} \right)- 4 \zeta(3) \left( 2\frac{1}{N^3} - 2\frac{3}{N^4} \right) + O\left( \frac{1}{N^5} \right) = \addjoana{O\left( \nu^5 \right)} \, .
      \end{multline*}
      \addjoana{
      Since $p_1$ and $p_2$ are analytic functions, it follows immediately that the required derivatives vanish at $\nu=0$.}

      By Taylor's theorem, we can then express these functions for some $\xi_{1},\xi_{2}\in(0,1/N_0)$ as
      \begin{equation*}
          p_1(\nu) = \frac{p_1^{(4)}(\xi_{1})}{4!}\nu^4 \, , \quad \text{and} \quad p_2(\nu) = \frac{p_2^{(5)}(\xi_{2})}{5!}\nu^5 \, .
      \end{equation*}
      We compute a lower bound of $p_1^{(4)}(\nu)$ over the interval
      $[0,1/N_0]$, using the methods in
      Appendix~\ref{sec:enclosing-extrema} for enclosing the minimum,
      combined with Taylor arithmetic to compute the derivative
      automatically, this gives us
      \begin{equation*}
        \min_{\nu \in [0, 1 / N_{0}]} p_1^{(4)}(\nu) \in \resultnumber{[1002.2 \pm 0.0471]}.
      \end{equation*}
      Since $p_1^{(4)}(\nu)$ is strictly positive on $[0,1/N_0]$, Taylor's theorem shows that $p_1(\nu)>0$ for all $\nu\in(0,1/N_0]$, hence $P_1(N) = p_1(1/N) > 0$ for all $N\geq N_0$.

      For the function $p_2(\nu)$, we can't do the same since $p_2^{(5)}(\nu)$ is not always positive in the interval, in particular, it is negative at $\nu=1/N_0$.
      More precisely, we have
      \begin{equation*}
         p_2^{(5)}(1 / N_{0}) \in \resultnumber{[-25014.73225 \pm 4.58 \cdot 10^{-6}]}.
      \end{equation*}
      Hence, we split the interval $[0,1/N_0]$ into \([0, a]\) and \([a, 1 / N_{0}]\) with \(a = \codenumber{1 / 1024}\). Following the same approach as for \(p_{1}\) we have the enclosures
      \begin{equation*}
        \min_{\nu \in [0, a]} p_2^{(5)}(\nu) \in \resultnumber{[5060.5 \pm 0.0267]}
      \end{equation*}
      and
      \begin{equation*}
        \min_{\nu \in [a, 1 / N_{0}]} p_2(\nu) \in \resultnumber{[4.8958 \cdot 10^{-14} \pm 2.86 \cdot 10^{-19}]}.
      \end{equation*}
      On the interval $[0,a]$, the Taylor expansion shows $p_2(\nu)>0$ for $\nu\in(0,a]$. For the interval $[a,1/N_0]$, we have computed a direct lower bound for the function $p_2(\nu)$ itself, directly establishing that $p_2(\nu)$ is positive on this part of the interval as well.
      Since $p_2(\nu)$ is positive on both $[0,a]$ and $[a,1/N_0]$, we conclude that $p_2(\nu)>0$ for all $\nu\in[0,1/N_0]$, which proves that $P_2(N)= p_2(1/N) > 0$ for $N\geq N_0$.
  \end{proof}

\section{The small $N$ case}
\label{sec:the-small-N-case}
For the approximate eigenpair \((u_{app}, \lambda_{app})\) that is
used in the above section to be a good approximation, \(N\) needs to
be sufficiently large. For smaller values of \(N\), \(\mathcal{P}_N\)
is far away from the circle and the same approximation produces very
poor results. To tackle this issue, a different type of approximation
is required. In this section we construct approximate solutions
following the classical MPS approach, it allows us to prove the
following proposition:

\begin{proposition}
  \label{prop:N_small}
  For \(N_{0} = \codenumber{64}\) we have
  \begin{equation*}
    \lambda_{1}(\mathcal{P}_3) > \lambda_{1}(\mathcal{P}_4) >
    \dots > \lambda_1(\mathcal{P}_N) > \lambda_1(\mathcal{P}_{N+1}) >
    \dots > \lambda_{1}(\mathcal{P}_{N_{0} - 1}) > \lambda_{1}(\mathcal{P}_{N_{0}}).
  \end{equation*}
  Furthermore, for \(q_N\) as in Theorem~\ref{main_thm}, it also holds
  that
  \begin{equation*}
    q_{3} > q_{4} > \dots > q_N > q_{N+1} > \dots >  q_{N_{0} - 1} > q_{N_{0}}.
  \end{equation*}
\end{proposition}

Note that the first two eigenvalues are known explicitly and given by
\(\lambda_{1}(\mathcal{P}_3) = \frac{4\pi}{\sqrt{3}}\) and
\(\lambda_{1}(\mathcal{P}_4) = 2\pi\). What remains to prove the
result is to get sufficiently accurate enclosures for the remaining
eigenvalues. Since \(q_{N_{0}}\) depends on
\(\lambda_{1}(\mathcal{P}_{N_{0} + 1})\) we have to compute enclosures
of \(\lambda_{1}(\mathcal{P}_{5})\) to
\(\lambda_{1}(\mathcal{P}_{N_{0} + 1})\).

The version of the MPS that we make use of here is due to Betcke and
Trefethen~\cite{Betcke-Trefethen:method-particular-solutions}, see
also~\cite[Section 4]{Dahne-GomezSerrano-Hou:counterexample-payne}.
The starting point is to write the eigenfunction as a linear
combination of functions \(\{\phi_{i}\}\) that satisfy the equation
\(-\Delta\phi_{i} = \lambda\phi_{i}\) in the domain, but with no
boundary conditions. The coefficients are then chosen to approximate
the boundary condition we want \(u\) to satisfy. In the case of a zero
Dirichlet boundary condition, this amounts to finding a non-zero
linear combination for which the boundary values are as close to zero
as possible. The linear combination is determined by taking \(m_{b}\)
collocation points on the boundary, then choosing the linear
combination to minimize its values on the collocation points in the
least squares sense. This alone is not quite enough, since increasing
the number of elements in the basis leads to the existence of linear
combinations very close to 0 inside the domain. The version by Betcke
and Trefethen handles this by also adding a number \(m_{i}\) of
interior points and taking the linear combination to stay close to
unit norm on these. This is accomplished by considering the two
matrices \(A_{B} = (\phi_{i}(x_{k}))_{k = 1}^{m_{b}}\) and
\(A_{I} = (\phi_{i}(y_{l}))_{l = 1}^{m_{i}}\) which are combined into
a matrix whose \emph{QR} factorization gives an orthonormal basis of
these function evaluations
\begin{equation}
  \label{eq:QR}
  A = \begin{bmatrix} A_{B} \\ A_{I} \end{bmatrix} =
  \begin{bmatrix} Q_{B} \\ Q_{I} \end{bmatrix} R =: Q R.
\end{equation}
The right singular vector \(v\) corresponding to the smallest singular
value \(\sigma = \sigma(\lambda)\) of \(Q_{B}\) for a given
\(\lambda\) is a good candidate for the eigenfunction when
\(\sigma(\lambda)\) is small. We refer to the work by Betcke and
Trefethen~\cite{Betcke-Trefethen:method-particular-solutions} for more
details about why this gives a good approximation.

To get a good numerical approximation, the choice of the basis
functions \(\{\phi_{i}\}\) is essential. In our case we make use of
two types of basis functions, both given in terms of the Bessel
function, \(J\). The first type is the one used in the original
version of MPS, and is responsible for approximating the eigenfunction
near the vertices of the domain. They are centered at the vertices of
the domain, and for a vertex with angle \(\pi/\alpha\) they take the
form
\begin{equation}\label{eq:vertex-expansion}
  \phi_{\alpha,k} = J_{k \alpha}\left(r\sqrt{\lambda}\right)\sin k\alpha\theta
\end{equation}
in polar coordinates centered around the vertex and \(\theta = 0\)
being taken along one of the boundary segments. The second type is an
interior expansion, which in our case is placed at the center of the
domain and in polar coordinates given by
\begin{equation}\label{eq:interior-expansion}
  \phi_0(r, \theta) = J_0\left(r \sqrt{\lambda}\right), \qquad
  \phi_j^\text{c}(r, \theta) = J_j\left(r \sqrt{\lambda}\right) \cos j\theta, \qquad
  \phi_j^\text{s}(r, \theta) = J_j\left(r \sqrt{\lambda}\right) \sin j\theta.
\end{equation}

For our purposes it suffices to take two terms in the expansion at
each vertex and two terms from the interior expansion, for a total of
\(2N + 2\) terms. The symmetry of the domain forces the expansions at
the vertices to all have identical coefficients. For the interior
expansion the symmetry makes only the \(\phi_j^\text{c}\) terms
relevant and also forces \(j\) to be a multiple of \(N\). By
normalizing the expansion so that the coefficient for the first
interior term is one, it can thus be written as
\begin{equation}
  \label{eq:approximation-small-N}
  u_{app}(x, y) = J_{0}\left(r\sqrt{\lambda}\right)
  + a_{2}J_{N}\left(r\sqrt{\lambda}\right)\cos N\theta
  + \sum_{n = 1}^{N} \left(
    b_{1}J_{\alpha}\left(r_{n}\sqrt{\lambda}\right)\sin \alpha\theta_{n}
    + b_{2}J_{2\alpha}\left(r_{n}\sqrt{\lambda}\right)\sin 2\alpha\theta_{n}
  \right).
\end{equation}
Here \(r\) is the distance from \((x, y)\) to the center of the
domain, and \((r_{n}, \theta_{n})\) are the polar coordinates of the
point \((x, y)\) when centered around vertex \(n\) and the orientation
taken so that \(\theta_{n} = 0\) corresponds to the boundary segment
between vertex \(n\) and \(n + 1\).

For the approximations of the eigenvalues we split into two different
cases. For \(N \geq 12\) we use the approximate values computed
in~\cite{Berghaus-Jones-Monien-Radchenko:computation-eigenvalues-2d-shapes},
note that these are computed to much higher precision than needed for
our purposes. For \(5 \leq N \leq 11\) approximations are computed by
searching for \(\lambda\) that minimizes the value of
\(\sigma(\lambda)\). This search is done in the interval between
\(\lambda_{1}(\mathcal{P}_{4}) = 2\pi\) and the approximation for
\(\lambda_{1}(\mathcal{P}_{12})\), and is performed using Brent's
method, a hybrid method which combines bisection with the secant
method and inverse quadratic interpolation to accelerate
convergence~\cite{Brent:algorithms-book}.

Given the approximation of \(\lambda\), we compute the coefficients as
the right singular vector of \(Q_{B}\) in Equation~\eqref{eq:QR}.
After normalizing so that the first coefficient for
\(J_{0}(r\sqrt{\lambda})\) is one, this gives us the coefficients
given in Figure~\ref{fig:coefficients-small-N}. Note that the
magnitudes of the coefficients are decreasing with \(N\), which is
expected since in the limit as \(N \to \infty\) the eigenfunction is
exactly given by \(J_{0}(r\sqrt{\lambda})\).

\begin{figure}
  \centering
  \begin{subfigure}[t]{0.45\textwidth}
    \includegraphics[width=\textwidth,alt={Log-linear plot of the
      magnitude of coefficient a subscript 2 scaled by J subscript N
      of 1 versus N. The data shows a rapid exponential decay
      (appearing linear on the log-linear scale) from around 10 to the
      power of -5 to down to 10 to the power of -35. There is a slight
      oscillation in the trend for the first few points where N <
      10.}]{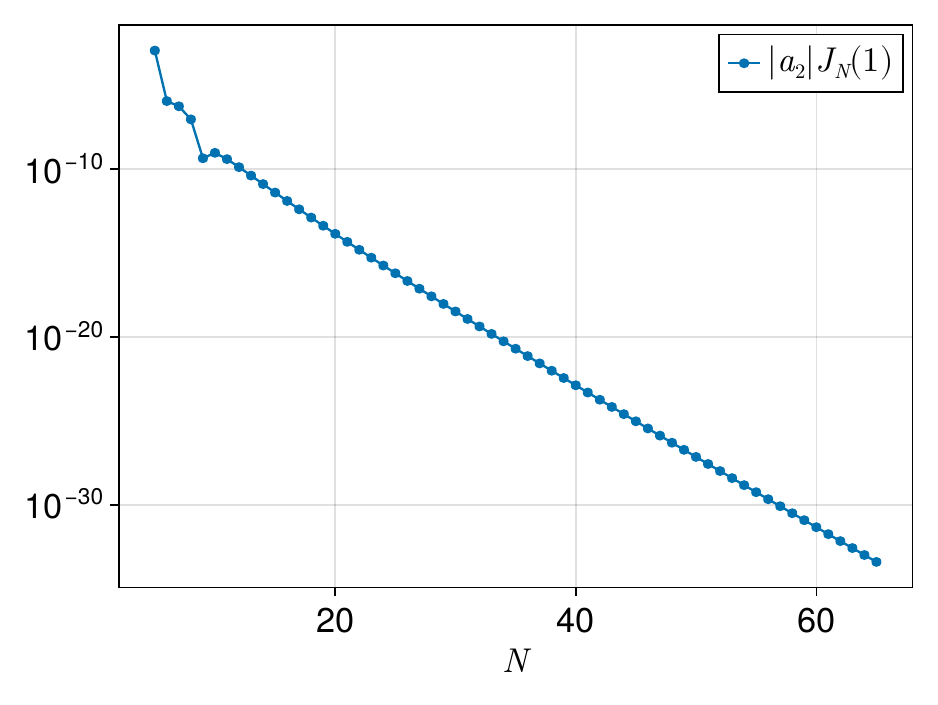}
    \caption{Coefficient \(a_{2}\).}
  \end{subfigure}
  \hspace{0.05\textwidth}
  \begin{subfigure}[t]{0.45\textwidth}
    \includegraphics[width=\textwidth,alt={Log-log plot of
      coefficients negative b subscript 1 (blue circles) and b
      subscript 2 (yellow crosses) against N ranging from 10 to 65.
      Both coefficients follow a power-law decay (appearing linear on
      the log-log scale). The coefficient b subscript 2 is
      consistently larger in magnitude, ranging from approx 10 to the
      power of -0.5 to 10 to the power of -1.2, while b subscript 1 is
      smaller and decays more steeply, ranging from 10 to the power of
      -0.8 down to 10 to the power of
      -2.1.}]{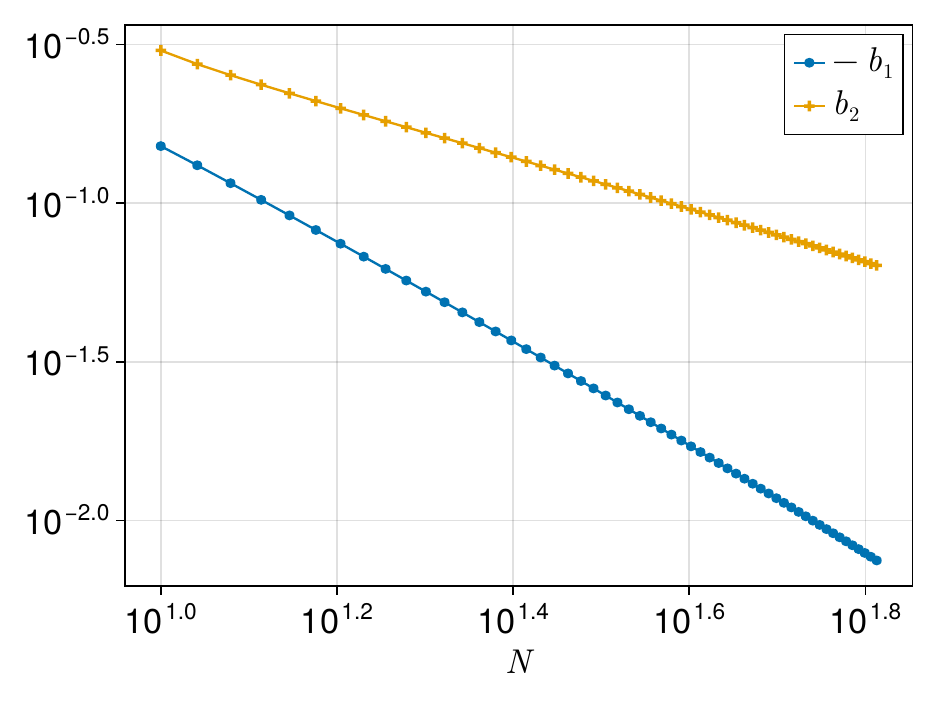}
    \caption{Coefficients \(b_{1}\) and \(b_{2}\).}
  \end{subfigure}
  \caption{Coefficients \(a_{2}\), \(b_{1}\) and \(b_{2}\) used in the
    approximate eigenfunction~\eqref{eq:approximation-small-N} and
    their dependence on \(N\). For \(a_{2}\) we scale the coefficients
    by \(J_{N}(1)\) to take into account the term's dependence on
    \(N\). For low values of \(N\) the domains are relatively
    different from each other and there are more oscillations in the
    coefficients, to better see the asymptotic behavior for \(b_{1}\)
    and \(b_{2}\) the graph therefore starts at \(N = 10\).}
  \label{fig:coefficients-small-N}
\end{figure}

To apply Lemma~\ref{lem:FoxHenriciMoler} we need to compute upper
bounds of the defect and lower bounds of the norm. The approach is
largely the same as in
~\cite{Dahne-GomezSerrano-Hou:counterexample-payne}:
\begin{description}
\item[Defect] An upper bound for the defect is computed using the
  approach discussed in Appendix~\ref{sec:enclosing-extrema}. This is
  very similar to the approach used in~\cite[Section
  2.1]{Dahne-Salvy:enclosures-eigenvalues}. The \(N\)-fold symmetry of
  the domain and the approximate eigenfunction means that we only have
  to bound the defect on one of the sides.
\item[Norm] For lower bounding the norm we use the same procedure as
  in~\cite[Section
  4]{GomezSerrano-Orriols:negative-hearing-shape-triangle}. Consider a
  subset of the domain, \(\Omega \subset \mathcal{P}_{N}\). If
  \(u_{app}\) does not vanish anywhere in \(\Omega\) then, without
  loss of generality, it can be assumed to be positive there. Then,
  since \(-\Delta u_{app} = \lambda_{app} u_{app} > 0\), \(u_{app}\)
  is superharmonic in \(\Omega\) and satisfies
  \(\inf_{\Omega} u_{app} \geq \inf_{\partial\Omega} u_{app}\). Thus a
  lower bound for \(|u_{app}|\) on \(\partial\Omega\) yields a lower
  bound for \(|u_{app}|\) inside \(\Omega\). To determine that
  \(u_{app}\) does not vanish on \(\Omega\), a lower bound for
  \(|u_{app}|\) on \(\partial\Omega\) is computed with the same
  techniques as when upper bounding the maximum. Once it is determined
  that \(u_{app}\) has a fixed sign on \(\partial\Omega\) (which we
  assume to be positive) the key observation is that \(u_{app}\)
  cannot be negative inside \(\Omega\) if \(\Omega\) is small enough.
  Indeed, if \(\Omega' \subset \Omega\) is a maximal domain where
  \(u < 0\), then \(u_{app} = 0\) on \(\partial\Omega'\) and thus
  \(\lambda_{app}\) is an eigenvalue for \(\Omega'\). If the area of
  \(\Omega\) is small enough this scenario can be ruled out using the
  Faber-Krahn inequality. \addjoel{More precisely, if
    \(\mathbb{D}_{*}\) is the disk with the same area as \(\Omega\),
    then \(\lambda_{app} < \lambda_{1}(\mathbb{D}_{*})\) ensures us
    that \(u_{app}\) has a constant sign inside \(\Omega\). We take
    \(\Omega\) to be \(\mathcal{P}_{N}\) scaled down by a fixed factor
    \(c\). Since \(\mathcal{P}_{N}\) has area \(\pi\) this gives}
  \begin{equation*}
    \lambda_{1}(\mathbb{D}_{*}) = \lambda_{1}(\mathbb{D}_{c})
    = \frac{\lambda_{1}(\mathbb{D}_{1})}{c^{2}}
    = \frac{j_{0,1}^{2}}{c^{2}}.
  \end{equation*}
  \addjoel{The value for \(c\) thus has to be taken such that
    \(\lambda_{app} < \frac{j_{0,1}^{2}}{c^{2}}\). Apart from this
    condition on \(c\), what value to use is a trade-off between the
    area of \(\Omega\) and the lower bound of \(u_{app}\) on it. In
    the implementation we take \(c = \codenumber{0.65}\), which is
    reasonably close to the optimal value for the disc.}
\end{description}

With upper bounds for the defects and lower bounds for the norms we
can compute enclosures of the eigenvalues using
Lemma~\ref{lem:FoxHenriciMoler}. \addjoel{The last step is to verify
  that the computed enclosures indeed correspond to the first
  eigenvalue. For this we recall that the circumradius of
  \(\mathcal{P}_{N}\) is given by the formula}
\begin{equation*}
  \addjoel{C_N = \sqrt{\frac{\pi}{\frac{N}{2} \sin\left(\frac{2\pi}{N}\right)}},}
\end{equation*}
\addjoel{see Lemma~\ref{lemma-disk}. Since \(C_{N}\) is decreasing in
  \(N\) we have that \(\mathcal{P}_{N}\) is contained in
  \(\mathbb{D}_{C_{5}}\) for all \(N \geq 5\). By monotonicity of
  eigenvalues with respect to the domain, we obtain
  \(\lambda_2(\mathbb{D}_{C_5}) \leq \lambda_2(\mathcal{P}_N)\).
  Computing the enclosure}
\begin{equation*}
  \addjoel{\lambda_2(\mathbb{D}_{C_5}) = \frac{j_{1,1}^2}{C_5^2} \in \resultnumber{[11.112 \pm 2.93 \cdot 10^{-4}]}}
\end{equation*}
\addjoel{it thus suffices to verify that the enclosures of the
  eigenvalues for the polygons are smaller than this number to ensure
  that they are the first eigenvalue.}

\addjoel{With the above we get enclosures of
  \(\lambda_{1}(\mathcal{P}_{N})\) for \(5 \leq N \leq N_{0} + 1\).
  For example, we have the enclosures
  \(\lambda_{1}(\mathcal{P}_{5}) \in \resultnumber{[6.022138 \pm 3.95
    \cdot 10^{-7}]}\),
  \(\lambda_{1}(\mathcal{P}_{6}) \in \resultnumber{[5.917418 \pm 3.53
    \cdot 10^{-7}]}\),
  \(\lambda_{1}(\mathcal{P}_{63}) \in \resultnumber{[5.7832972 \pm
    3.37 \cdot 10^{-8}]}\) and
  \(\lambda_{1}(\mathcal{P}_{64}) \in \resultnumber{[5.7832920 \pm
    5.35 \cdot 10^{-8}]}\). From these enclosures, the condition}
\begin{equation*}
  \lambda_{1}(\mathcal{P}_3) > \lambda_{1}(\mathcal{P}_4) >
  \cdots > \lambda_{1}(\mathcal{P}_{N_{0} - 1}) > \lambda_{1}(\mathcal{P}_{N_{0}})
\end{equation*}
\addjoel{reduces to verifying that the enclosures do not overlap and
  form a decreasing sequence, which indeed is the case.} See
Figure~\ref{fig:difference-eigenvalues-small-N} for how the distance
between the eigenvalues compares with the accuracy of the computed
enclosures. We can also compute enclosures of \(q_{N}\); \addjoel{for
  example, we obtain
  \(q_{4} \in \resultnumber{[1.043348 \pm 1.05 \cdot 10^{-7}]}\),
  \(q_{5} \in \resultnumber{[1.017697 \pm 1.64 \cdot 10^{-7}]}\),
  \(q_{63} \in \resultnumber{[1.00000089 \pm 5.92 \cdot 10^{-9}]}\)
  and
  \(q_{64} \in \resultnumber{[1.00000083 \pm 6.89 \cdot 10^{-9}]}\).
  Just as with the eigenvalues, it is then straightforward to verify
  their monotonicity. Figure~\ref{fig:difference-q-N-small-N} shows
  how the distance compares with the accuracy of the enclosures.} This
concludes the proof of Proposition~\ref{prop:N_small}.

\begin{remark}
  The tolerances in verifying the above inequalities for
  \(\lambda_{N}\) and \(q_{N}\) are relatively tight. This is however
  not a limitation in the method. It is possible to significantly
  improve the accuracy of the enclosures by increasing the number of
  basis functions in \(u_{app}\). In fact, the inequality can be
  verified at least all the way up to \(N = 660\) (which is how far
  the computations
  in~\cite{Berghaus-Jones-Monien-Radchenko:computation-eigenvalues-2d-shapes}
  go) by using more basis functions.
\end{remark}

\begin{figure}
  \centering
  \begin{subfigure}[t]{0.45\textwidth}
    \includegraphics[width=\textwidth,alt={Log-linear plot comparing
      the difference between successive eigenvalue approximations
      (blue circles) versus the error bound (yellow crosses) as N
      increases from 5 to 65. The difference decays from 1 down to
      approximately 10 to the power of −5. The error bound is
      consistently lower than the difference, mostly staying below 10
      to the power of −6 except for a brief spike around N equal to
      8.}]{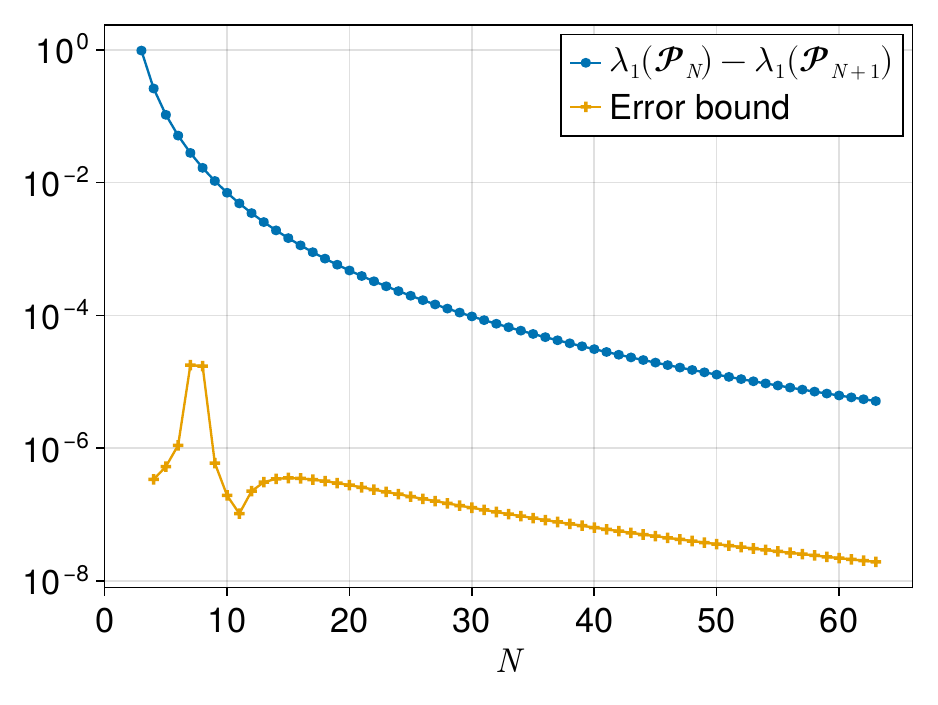}
    \caption{Approximations of eigenvalues.}
    \label{fig:difference-eigenvalues-small-N}
  \end{subfigure}
  \hspace{0.05\textwidth}
  \begin{subfigure}[t]{0.45\textwidth}
    \includegraphics[width=\textwidth,alt={Log-linear plot comparing
      the difference between successive approximations of q subscript
      N (blue circles) versus the error bound (yellow crosses). The
      difference curve decays from roughly 10 to the power of −1 to 10
      to the power of −7. The error bound remains strictly below the
      difference curve throughout the entire
      range.}]{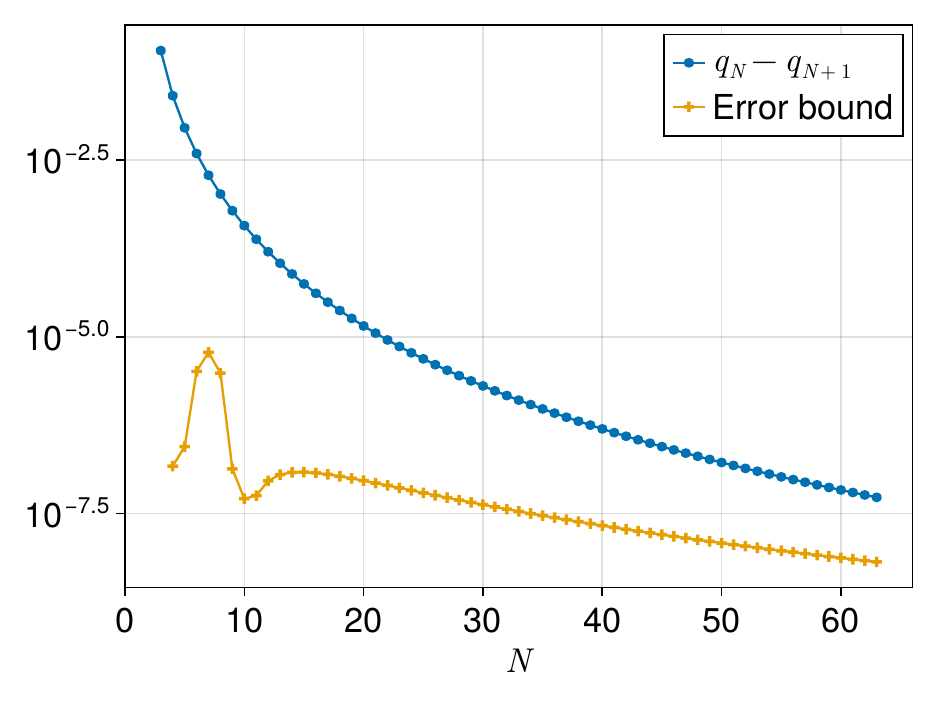}
    \caption{Approximations of \(q_{N}\).}
    \label{fig:difference-q-N-small-N}
  \end{subfigure}
  \caption{Distance between successive approximations and the computed
    error bounds. As long as the error is smaller than the distance,
    the monotonicity can be verified. For the eigenvalues the error
    bounds are computed as the sum of the radii for the enclosures of
    \(\lambda_{1}(\mathcal{P}_{N})\) and
    \(\lambda_{1}(\mathcal{P}_{N + 1})\), for the \(q_{N}\)'s it is
    the sum of the radii for the enclosures for \(q_{N}\) and
    \(q_{N + 1}\).}
\end{figure}

\appendix

\section{Auxiliary Lemmas}
\label{app:aux_lemmas}

This appendix collects several auxiliary results for bounds of polylogarithms and useful expressions required throughout the main text. We begin by establishing bounds for multiple polylogarithms $\Li_{m_1,\dots,m_k}(z)$, which appear in the functions $V_l(z)$.

\begin{lemma}\label{lemma:polylog_bound}
  Let \(|z| \leq a < 1\), for any positive integers
  \(m_{1}, \dots, m_{k}\) we have the bound
  \begin{equation*}
    |\Li_{m_{1},\dots,m_{k}}(z)| \leq \Li_{1,\dots,1}(|z|) \leq \Li_{1,\dots,1}(a),
  \end{equation*}
  where the number of \(1\)'s for \(\Li_{1,\dots,1}\) is \(k\).
\end{lemma}

\begin{proof}
  The first inequality follows directly from the series expansion
  \begin{equation*}
    \Li_{m_1,\dots,m_k}(z) = \sum_{0<n_1<n_2<\dots<n_k} \frac{z^{n_k}}{n_1^{m_1} n_2^{m_2} \cdots n_k^{m_k}},
  \end{equation*}
  given that for fixed $n_1,\dots,n_k$, the term coming from the
  $\Li_{m_1,\dots,m_k}(z)$ series expansion is less than or equal to
  the term arising from $\Li_{1,\dots,1}(|z|)$. The second inequality
  follows from that \(\Li_{1,\dots,1}(|z|)\) is increasing for
  \(|z| \in [0, 1)\), which is a consequence of the terms in the
  series expansion all being positive.
\end{proof}

The following lemma gives bounds for \(V_{l}(z)\) in terms of
$\log(1-|z|)$.
\begin{lemma}\label{lemma:V_log_bounds}
  Let \(|z| < 1\), for the \(V_{l}\)'s in Equation~\eqref{eq:Vs} we
  have the bound
  \begin{equation*}
    \left|V_{l}(z)\right| \leq \sum_{j = 1}^{l} C_{V,l,j} \frac{|\log(1 - |z|)|^{j}}{j!},
  \end{equation*}
  with
  \begin{equation*}
    C_{V,1,1} = 2,
  \end{equation*}
  \begin{equation*}
    C_{V,2,1} = |\lambda / 2 - 2|,\quad
    C_{V,2,2} = 4,
  \end{equation*}
  \begin{equation*}
    C_{V,3,1} = |\lambda^{2} / 16 - \lambda + 2|,\quad
    C_{V,3,2} = |3\lambda - 12| + |\lambda - 4|,\quad
    C_{V,3,3} = 8,
  \end{equation*}
  as well as
  \begin{equation*}
    C_{V,4,1} = |\lambda^{3} / 192 - \lambda^{2} / 8 - \lambda / 2 - 2| + |2\lambda\zeta(3)|,\quad
    C_{V,4,2} = |\lambda^{2} / 8 - 2\lambda + 4|
    + |\lambda^{2} / 4 - 4\lambda + 12|
    + |5\lambda^{2} / 8 - 8\lambda + 28|
  \end{equation*}
  \begin{equation*}
    C_{V,4,3} = |2\lambda - 8| + |6\lambda - 24| + |14\lambda - 56|,\quad
    C_{V,4,4} = 16.
  \end{equation*}
  Moreover, for \(|z| \leq a < 1\) we have the bound
  \begin{equation*}
    \left|V_{l}(z)\right| \leq D_{l,a}|z|,
  \end{equation*}
  with
  \begin{equation*}
    D_{l,a} = \sum_{j = 1}^{l} C_{V,l,j} \frac{|\log(1 - a)|^{j}}{a \cdot j!}.
  \end{equation*}
\end{lemma}

\begin{proof}
  The polylogarithms in the definition of the $V_l$'s can be bounded
  by \(\Li_{1,\dots,1}(|z|)\) using Lemma~\ref{lemma:polylog_bound}.
  The first bound then follows using the identity
  $\Li_{1,\dots,1}(z) = \frac{1}{k!} (-\log(1-z))^k = \frac{1}{k!}
  |\log(1-z)|^k$, for \(\Li_{1,\dots,1}\) with the number of \(1\)'s
  given by \(k\), and choosing the appropriate $C_{V,l,j}$ constants
  using the definition of the $V_l$'s from Equation~\eqref{eq:Vs}.

  For the second bound we note that the first inequality gives us
  \begin{equation*}
    \left|V_{l}(z)\right| \leq |z|\sum_{j = 1}^{l} C_{V,l,j} \frac{|\log(1 - |z|)|^{j}}{|z| \cdot j!}.
  \end{equation*}
  The result follows from that
  \begin{equation*}
    \frac{|\log(1 - |z|)|^{j}}{|z|} = \frac{\Li_{1,\dots,1}(|z|)}{|z|}
  \end{equation*}
  is increasing for \(|z|\) in the interval \([0, 1)\). The fact that
  \(\frac{\Li_{1,\dots,1}(|z|)}{|z|}\) is increasing in \(|z|\)
  follows, similar to in the proof of Lemma~\ref{lemma:polylog_bound},
  from that the series expansion has only positive coefficients.
\end{proof}

The following lemma provides bounds for the integrals of the $d_l(z,t)$ kernels appearing in Lemma~\ref{lemma:bounds-I_k_l}.

\begin{lemma}\label{lemma:integral_d}
  Let $0<a<1$. For \(d_{0}\) and \(d_{1}\) in
  Lemma~\ref{lemma:exact-expression-for-K-up-to-order_-3}
  we have
  \begin{equation*}
    \int_{0}^{a} |d_{0}(z, s z)|\ ds = a
  \end{equation*}
  and
  \begin{equation*}
    \int_{0}^{a} |d_{1}(z, s z)|\ ds = \frac{\lambda}{4} \int_{0}^{a} \left|\log(s)\right|\ ds.
  \end{equation*}
  For \(d_{2}\) we have
  \begin{equation*}
    \int_{0}^{a} |d_{2}(z, s z)|\ ds \leq
    \left(\frac{\lambda^{2}}{64} + \frac{\lambda}{8}\right)\int_{0}^{a}\left|\log^{2}(s)\right|\ ds
    + \frac{\lambda}{4}a\left|S_{2}(\xi_{0}z) - S_{2}(z)\right|
  \end{equation*}
  for some \(\xi_{0} \in (0, a)\). For \(d_{3}\) we have
  \begin{multline*}
    \int_{0}^{a} |d_{3}(z, s z)|\ ds \leq
    \left(\frac{\lambda^{3}}{2304} + \frac{\lambda^{2}}{64} + \frac{\lambda}{24}\right)\int_{0}^{a}\left|\log^{3}(s)\right|\ ds\\
    + \left(
    \frac{\lambda^{2}}{32} \left|S_{2}(\xi_{1}z) - S_{2}(z)\right|+
    \frac{\lambda}{4} \left|S_{2}(\xi_{2}z) + S_{2}(\bar{z})\right|
    \right)\int_{0}^{a}\left|\log(s)\right|\ ds
    + \frac{\lambda}{4}a\left|S_{3}(\xi_{0}z) - S_{3}(z)\right|
  \end{multline*}
  for some \(\xi_{0}\), \(\xi_{1}\) and  \(\xi_{2}\) in $(0, a)$.

  The remaining integrals can be explicitly computed from
  \begin{equation*}               \int_{0}^{a}\left|\log^{m}(s)\right|\ ds = \left| \int_{0}^{a}\log^{m}(s) ds \right|\
    = \left|  (-1)^m a \sum_{n = 0}^{m} (-1)^{n} \frac{m!}{n!}\log(a)^{n} \right|.
  \end{equation*}
\end{lemma}

\begin{proof}
  The expressions for the integral of $d_0$ and $d_1$ are a direct consequence of their expression in Lemma~\ref{lemma:exact-expression-for-K-up-to-order_-3}. For $d_2$ and $d_3$ we need to further show that for $l=1,2$:
  \begin{align*}
      \int_0^a |S_l(sz) - S_l(z)| ds &= a |S_l(\xi_0z) -S_l(z)| \,  \text{ for some $\xi_0\in(0,a)$}, \\
      \int_0^{a} \left|\log(s)\right| |S_2(sz) - S_2(z)| ds &= |S_2(\xi_1 z) - S_2(z)| \int_0^{a} \left|\log(s)\right| ds \, \text{ for some $\xi_1\in(0,a)$, and} \\
      \int_0^{a} \left|\log(s)\right| |S_2(sz) + S_2(\bar{z})| ds &= |S_2(\xi_2 z) + S_2(\bar{z})| \int_0^{a} \left|\log(s)\right| ds \, \text{ for some $\xi_2\in(0,a)$.}
  \end{align*}
  These follow from the generalized mean value theorem for integrals. The functions $s \mapsto |S_l(sz) - S_l(z)|$ and $s \mapsto |S_l(sz) + S_l(\bar{z})|$ are continuous and real-valued, and both $s \mapsto 1$ and $s \mapsto |\log(s)|$ are real-valued integrable functions that do not change sign on $(0,a)$ for $a<1$, so the theorem can be applied to obtain the desired result.

  For the last remaining integrals involving only logarithms, since $0<a<1$ the logarithms do not change sign so the absolute value can be written outside the integral. Then, the last equality holds from Lemma~\ref{lemma:basic-integrals}.
\end{proof}

Finally, we explicitly compute some basic logarithmic integrals, which
are used in different parts of the paper.
\begin{lemma}\label{lemma:basic-integrals}
  Let \(0 < a, b < 1\) and \(m \in \mathbb{Z}_{\geq 0}\). We have
  \begin{equation*}
    \int_{0}^{a} \log^{m}\left(s\right)\ ds
    = (-1)^{m}a \sum_{n = 0}^{m} (-1)^{n}\frac{m!}{n!}\log(a)^{n}.
  \end{equation*}
  Moreover, we have
  \begin{equation*}
    \int_{b}^{1}\log^{m}(1 - sz)\ ds = L_{1}(1) - L_{1}(b),
  \end{equation*}
  and
  \begin{equation*}
    \int_{b}^{1}\log^{m}(1 - sz)(1 - sz)^{y}\ ds = L_{2}(1) - L_{2}(b),
  \end{equation*}
  where
  \begin{equation*}
    L_{1}(s) = (-1)^mm!\left(
      s - \frac{1}{z}\sum_{n = 1}^{m} \frac{(-1)^n}{n!}\log(1 - sz)^n(1 - sz)
    \right)
  \end{equation*}
  and
  \begin{equation*}
    L_{2}(s) = \frac{(-1)^{m + 1}m!}{z}\left(
      \sum_{n = 0}^{m} \frac{(-1)^{n}}{n!}(1 + y)^{n - m - 1}\log(1 - sz)^{n}(1 - sz)^{1 + y}
    \right).
  \end{equation*}
  Note that for \(z = 1\), the functions \(L_{1}(s)\) and \(L_{2}(s)\)
  both have removable singularities at \(s = 1\) coming from the
  \(\log(1 - sz)^n(1 - sz)\) factors.
\end{lemma}

\section{Implementation details}
\label{app:implementation}
The full code on which the computer-assisted parts of the proofs are
based, as well as notebooks presenting the results and figures, are
available at the repository~\cite{SpectralRegularPolygon.jl}. The code
is implemented in Julia~\cite{Bezanson-Edelman-Karpinski-Shah:julia}.
For the rigorous numerics it uses FLINT~\cite{Flint}/Arb\footnote{In
  2023 Arb was merged with the FLINT library}~\cite{Johansson:Arb}
through the Julia wrapper \textit{Arblib.jl}~\cite{Arblib.jl}, with
many of the basic rigorous numerical algorithms, such as isolating
roots or enclosing maximum values of functions, implemented in a
separate package, \textit{ArbExtras.jl}~\cite{ArbExtras.jl}.

The FLINT library implements ball (intervals represented as a midpoint
and a radius) arithmetic, which allows for computation of rigorous
enclosures of functions. It has built-in support for many of the
special functions that are required, including many hypergeometric
functions such as Bessel functions, \({}_{0}F_{1}\) and
\({}_{2}F_{1}\)~\cite{Johansson2019}. There are efficient
implementations of Taylor arithmetic (see
e.g.~\cite{Johansson2015reversion}), which are important for the
automatic computation of high order derivatives. Moreover, it supports
rigorous numerical integration~\cite{Johansson2018numerical}, which we
comment more on in Section~\ref{sec:rigorous-integration} below.

The computations were done on an AMD Ryzen 9 5900X processor with 32
GB of RAM using 12 threads and 128 bits of precision for the ball
arithmetic. The computational cost is dominated by
Lemma~\ref{lemma:bounds-I_k_l}, which takes around 25 minutes to run.
All of the other computations in Section~\ref{sec:the-large-N-case}
take less than a minute each and Proposition~\ref{prop:N_small} takes
just over one minute to run.

\subsection{Enclosing extrema}
\label{sec:enclosing-extrema}
In several cases we have to compute enclosures (or at least
lower/upper bounds) of expressions of the form
\begin{equation*}
  \min_{x \in I} f(x) \quad \text{or} \quad \max_{x \in I} f(x),
\end{equation*}
for a finite interval \(I\). For example this occurs in the proofs for
lemmas~\ref{lemma:bounds_and_splitting:first_term_J0:upper-bounds-for-terms-from-g-function},
\ref{lemma:g_function_derivatives}, \ref{lemma:bounds-I_k_l}
and~\ref{lemma:final-inequalities} as well as when upper bounding the
defects and lower bounding the norms of the approximate eigenfunctions
in Section~\ref{sec:the-small-N-case}.

Methods for computing such enclosures are implemented in the
\textit{ArbExtras.jl} package (under the names
\textit{minimum\_enclosure} and \textit{maximum\_enclosure}). The
procedure is the same as discussed in e.g.~\cite[Section
2.1]{Dahne-Salvy:enclosures-eigenvalues} and~\cite[Section
7.1]{Dahne:highest-cusped-waves-fkdv} (see
also~\cite{Dahne-GomezSerrano:highest-wave-burgershilbert}) and we
therefore only briefly describe it here.

Let us consider the case of enclosing the maximum, with the minimum
being similar. The main idea is to iteratively bisect the interval
\(I\) into smaller and smaller subintervals. At every iteration we
compute an enclosure of \(f\) on each subinterval. From these
enclosures a lower bound of the maximum can be computed. We then
discard all subintervals for which the enclosure is less than the
lower bound of the maximum, as the maximum cannot be attained there.
For the remaining subintervals we check if their enclosure satisfies
some prespecified tolerance, in that case we don't bisect them
further. If there are any subintervals left, we bisect them and
continue with the next iteration. In the end, either when there are no
subintervals left to bisect or we have reached some maximum number of
iterations (to guarantee that the procedure terminates), we return the
maximum of all subintervals that were not discarded. This is
guaranteed to give an enclosure of the maximum of \(f\) on the
interval.

If we are able to compute Taylor expansions of the function \(f\) we
can improve the performance of the procedure significantly. In our
case this is primarily relevant in Section~\ref{sec:the-small-N-case},
where Taylor expanding the approximate eigenfunctions is
straightforward. In Section~\ref{sec:the-large-N-case}, most of the
bounds are for functions given by the absolute value of some complex
valued function, and are hence in general not differentiable.

In the case that we are able to compute Taylor expansions of \(f\), we
make the following adjustments to the above method. Consider a
subinterval \(I_{i}\), instead of directly computing an enclosure of
\(f(I_{i})\) we compute a Taylor polynomial \(P\) (usually of degree
\(8\)) at the midpoint and an enclosure \(R\) of the remainder term
such that \(f(x) \in P(x) + R\) for \(x \in I_{i}\). We then have
\begin{equation}
  \label{eq:taylor-bound}
  \max_{x \in I_{i}} f(x) \in \max_{x \in I_{i}} P(x) + R.
\end{equation}
To compute \(\sup_{x \in I_{i}} P(x)\) we isolate the roots of \(P'\)
(using an adaptive bisection approach) on \(I_{i}\) and evaluate \(P\)
on the roots as well as the endpoints of the interval. In practice the
computation of \(R\) involves computing an enclosure of the Taylor
expansion of \(f\) on the full interval \(I_{i}\). Since this includes
the derivative we can as an extra optimization check if the derivative
is non-zero, in which case \(f\) is monotone, and it is enough to
evaluate \(f\) on either the left or the right endpoint of \(I_{i}\),
depending on the sign of the derivative.

\subsection{Rigorous integration}
\label{sec:rigorous-integration}
For the computation of the generalized polylogarithms as well as the
integrals in Lemma~\ref{lemma:bounds-I_k_l} we need to compute
rigorous enclosures of integrals. For this we make use of the
integrator implemented in FLINT~\cite{Johansson2018numerical}, which
allows for efficient evaluation of integrals of functions that are
analytic in a neighborhood of the path of integration. The method
automatically handles the computation of error bounds by bounding the
integrand on an ellipse enclosing the path of integration. For the
error bounds to be valid the function needs to be analytic on this
ellipse, ensuring that this is the case is up to the user. In our
case, what we have to check for is branch cuts that intersect this
ellipse. Below we go through the two different types of integrands
that appear in the paper and analyze their branch cuts. Checks for
these branch cuts are then implemented in the integration code.

The first type of integrands comes from the evaluation of the
generalized polylogarithms and is of the form
\begin{equation*}
  \frac{d^{n}}{d\nu^{n}} \nu t^{\nu}((1 - tz)^{-2\nu} - 1)\frac{1}{t},
\end{equation*}
with \(0 \leq \nu \leq 1 / N_{0}\) and \(n \geq 2\). The factor
\(t^{\nu}\) has a branch cut along the negative real axis and the term
\((1 - tz)^{-2\nu}\) has a branch cut whenever \(1 - tz\) lies on the
negative real axis. When required to verify analyticity we therefore
check that \(t\) and \(1 - tz\) do not overlap the negative real axis.
An exception to this is the case when \(n = 2\) and \(\nu = 0\), in
this case the integrand is given by
\begin{equation*}
  -4\frac{\log(1 - tz)}{t},
\end{equation*}
which has a removable singularity at \(t = 0\) and no branch cut for
\(t\) lying on the negative real axis.

The second type of integrands comes from
Lemma~\ref{lemma:bounds-I_k_l}, where the integrands that occur are
\begin{equation*}
  d_k(z,t) V_l(t) \text{ with } (k, l) = (1, 4),\ (2, 3),\ (2, 4),\ (3, 2),\ (3, 3) \text{ and } (3, 4)
\end{equation*}
as well as \(K_4(N,z,t) V_l(t)\) with \(l = 1,\ 2,\ 3\) and \(4\). The
functions \(d_{1}(z, t)\), \(d_{2}(z, t)\) and \(d_{3}(z, t)\) all
have branch cuts when \(t / z\) lies on the negative real axis, due to
the \(\log(t / z)\) factors. In addition to this, \(d_{2}(z, t)\) and
\(d_{3}(z, t)\) have branch cuts along the real axis from \(t = 1\) to
\(t = +\infty\) coming from \(S_{2}(t)\) and \(S_{3}(t)\). Note that
the branch cuts in \(z\) are not important since we are only
integrating in \(t\). For the function \(K_{4}(N, z, t)\), the
situation is the same. In this case the branch cuts come from
\((t / z)^{1/N}\) and \(F_{N}(t)\) appearing in the expression for
\(K(z, t)\) in the proof of
Lemma~\ref{lemma:exact-expression-for-K-up-to-order_-3}. For the
\(V_{l}\) functions, the only branch cut is along the real axis from
\(t = 1\) to \(t = +\infty\).

The FLINT integrator does not perform well if the enclosures of the
endpoints of the integration are not very precise. If the endpoints
are wide (complex) intervals, then the result will be significantly
worse. This is relevant in Lemma~\ref{lemma:bounds-I_k_l} where we are
computing integrals with endpoints that depend on \(z = e^{i\theta}\),
where \(\theta\) is represented by wide intervals. To improve the
performance in this case we integrate to the midpoint of the interval
that represents \(z\) and add the resulting error term separately.
More precisely, consider the problem of computing an enclosure of the
integral
\begin{equation*}
  \int_{\inter{a}}^{\inter{b}} f(t)\ dt,
\end{equation*}
for an analytic function \(f\) and with \(\inter{a}\) and
\(\inter{b}\) representing two (complex) intervals. If we let
\(\midint(\inter{a})\) and \(\midint(\inter{b})\) denote the midpoints
of \(\inter{a}\) and \(\inter{b}\) respectively, then we use that
\begin{equation*}
  \int_{\inter{a}}^{\inter{b}} f(t)\ dt
  \subseteq \int_{\midint(\inter{a})}^{\midint(\inter{b})} f(t)\ dt
  + (\inter{a} - \midint(\inter{a}))f(\inter{a})
  + (\inter{b} - \midint(\inter{b}))f(\inter{b}).
\end{equation*}
Here \(f(\inter{a})\) and \(f(\inter{b})\) are complex intervals
enclosing the sets \(\{f(a): a \in \inter{a}\}\) and
\(\{f(b): b \in \inter{b}\}\) respectively, which are straightforward
to compute using interval arithmetic.
left
\subsection{Computing generalized polylogarithms}
\label{sec:polylogs}
Several of the computer-assisted lemmas in
Section~\ref{sec:the-large-N-case} require computing enclosures of
certain generalized polylogarithms. While FLINT has an implementation
of the standard polylogarithm, \(L_{s}(z)\), it does not implement the
more general versions that we require here. This section contains the
details for how we compute these.

Let us start with the multiple polylogarithms, that occur in the
expressions for the \(V\)'s in Equation~\eqref{eq:Vs}. The ones that
occur are
\begin{equation*}
  \Li_{1,1},\quad
  \Li_{1,2},\quad
  \Li_{1,3},\quad
  \Li_{2,1},\quad
  \Li_{2,2},\quad
  \Li_{3,1},\quad
  \Li_{1,1,1},\quad
  \Li_{1,2,1},\quad
  \Li_{2,1,1},\quad\text{and}\quad
  \Li_{1,1,1,1}.
\end{equation*}
The ones with the parameters given by only \(1\)'s are given by
\begin{equation*}
  \Li_{1,\dots,1}(z) = (-1)^{n}\frac{\log^{n}(1 - z)}{n!},
\end{equation*}
where \(n\) is the number of \(1\)'s. Expressions for the other ones
are given in Table~\ref{table:multiple-polylog-expressions}. For
\(\Li_{1,3}\) we have used that it is equal to the Nielsen generalized
polylogarithm \(S_{2,2}\), combined with the expression
from~\cite[Proposition 5]{Charlton2021}. For all the other entries in
the table we have computed their iterative integral representation
following the recursive differential equations in
\eqref{eq:multiple-polylogs-differential-equations}. Note that these
expressions are in general not suitable for evaluation near \(z = 0\),
in this case bounds are computed using
Lemma~\ref{lemma:polylog_bound}. Also observe that the branch cuts of
the functions appearing in the expressions from
Table~\ref{table:multiple-polylog-expressions} cancel out when the
primary branch is taken, i.e. the expressions for the multiple
polylogarithms are single-valued in \(\mathbb{D}\) even though they
are written in terms of multivalued functions like the $\log$.

\begin{table}[h!]
  \centering
  \begin{tabular}{r|r}
    \hline
    $\Li_{1,2}$ & $\frac{1}{2} \log^2(1 - z) \log(z) + \log(1 - z)\Li_2(1 - z) - \Li_3(1 - z) + \zeta(3)$ \\
    \hline
    $\Li_{1,3}$ & $\begin{array}{r}
      -\Li_{4}(1 - z)
      + \Li_{4}(z)
      + \Li_{4}((1 - 1/z)^{-1})
      - \Li_{3}(z)\log(1 - z) \\ [0.5em]
      + \frac{1}{4!}\log^{4}(1 - z)
      - \frac{1}{3!}\log(z)\log^{3}(1 - z)
      + \frac{\zeta(2)}{2!}\log^{2}(1 - z)
      + \zeta(3)\log(1 - z)
      + \zeta(4)
    \end{array}$ \\
    \hline
    $\Li_{2,1}$ & $-\frac{1}{6} \log(1-z) (\pi^2 + 6\Li_2(1-z)) + 2\Li_3(1-z) - 2\zeta(3)$ \\
    \hline
    $\Li_{2,2}$ &  $\begin{array}{r}
      - \log^3(1 - z) \log(z)
      + \frac{1}{6} \log^2(1 - z) \left( \pi^2 + 9 \log^2(z)\right) \\ [0.5em]
      - \frac{1}{6} \log(1 - z) \left(\pi^2 \log(z) + 6 \log^3(z) \right)
      + \frac{1}{4} \log^4(z) \\ [0.5em]
      - \frac{1}{2} \Li_2^2(1 - z)
      + \log^2\left( -1 + \frac{1}{z} \right) \Li_2\left( 1 - \frac{1}{z} \right) \\ [0.5em]
      - \left( \log^2\left( -1 + \frac{1}{z} \right) + \log(1 - z) \log(z) \right) \Li_2(z)
      - 2 \log\left( -1 + \frac{1}{z} \right) \Li_3(1 - z) \\ [0.5em]
      - 2 \log\left( -1 + \frac{1}{z} \right) \Li_3\left( 1 - \frac{1}{z} \right)
      + 2 \log(z) \Li_3(z) \\ [0.5em]
      + 2 \Li_4(1 - z)
      + 2 \Li_4\left( 1 - \frac{1}{z} \right) \\ [0.5em]
      - 2 \Li_4(z)
      + \Li_2^2(1 - z)
      + \frac{\pi^2}{6} \Li_2(z)
      - 2 \log(z) \zeta(3)
      + \frac{\pi^4}{360}
    \end{array}$ \\
    \hline
    $\Li_{3,1}$ & $-\log(1 - z) - \frac{1}{2}\Li_2^2(z)\Li_3(z)$ \\
    \hline
    $\Li_{1,1,2}$ & $-\frac{1}{6} \log^3(1 - z) \log z  - \frac{1}{2} \log^2(1 - z)\Li_2(1 - z) + \log(1 - z)\Li_3(1 - z) - \Li_4(1 - z) + \frac{\pi^4}{90}$ \\
    \hline
    $\Li_{1,2,1}$ & $\frac{1}{2} \log^2(1 - z)\Li_2(1 - z) - \log(1 - z) \left(2\Li_3(1 - z) + \zeta(3)\right) + 3\Li_4(1 - z) - \frac{\pi^4}{30}$ \\
    \hline
    $\Li_{2,1,1}$ & $\frac{\pi^2}{12}\log^2(1 - z) + \log(1 - z) (\Li_3(1 - z) + 2\zeta(3)) - 3\Li_4(1 - z) + \frac{\pi^4}{30}$ \\
    \hline
    \end{tabular}
    \caption{Expressions of the multiple polylogarithms \(\Li_{1,2}\),
      \(\Li_{1,3}\), \(\Li_{2,1}\), \(\Li_{2,2}\), \(\Li_{3,1}\),
      \(\Li_{1,2,1}\) and \(\Li_{2,1,1}\) that are used for
      computations.}
    \label{table:multiple-polylog-expressions}
\end{table}

The other type of generalized polylogarithms we have to deal with are
the functions \(S_{n}(z)\) coming from the expansion of \(F_{N}\) in
Equation~\eqref{eq:FNz_expansion_powers_1/N}. These functions can be
written in terms of finite sums of Nielsen generalized polylogarithms.
For computational purposes we have, however, found it better to
compute the \(S_{n}\) functions directly through their integral
formula. If we combine Equation~\eqref{eq:FNz_integral}
and~\eqref{eq:FNz_expansion_powers_1/N} and let \(\nu = 1 / N\) we
have
\begin{equation*}
  S_{n}(z) = \frac{1}{n!}\int_{0}^{1}\frac{d^{n}}{d\nu^{n}} \nu t^{\nu}((1 - tz)^{-2\nu} - 1)\bigg|_{\nu = 0}\frac{dt}{t}.
\end{equation*}
We will compute the functions \(S_{n}\) by computing an enclosure of
this integral. The approach for this is similar to how the integrals
in Lemma~\ref{lemma:bounds-I_k_l} are bounded. We are interested in
the cases \(n = 2, 3, 4, 5\), for which we can explicitly compute
\begin{align*}
  \frac{d^{2}}{d\nu^{2}} \nu t^{\nu}((1 - tz)^{-2\nu} - 1)\bigg|_{\nu = 0}
  &= -4\log(1 - tz)\frac{1}{t},\\
  \frac{d^{3}}{d\nu^{3}} \nu t^{\nu}((1 - tz)^{-2\nu} - 1)\bigg|_{\nu = 0}
  &= \left(-12\log(t)\log(1 - tz) + 12\log^{2}(1 - tz)\right)\frac{1}{t},\\
  \frac{d^{4}}{d\nu^{4}} \nu t^{\nu}((1 - tz)^{-2\nu} - 1)\bigg|_{\nu = 0}
  &= \left(-24\log^{2}(t)\log(1 - tz) + 48\log(t)\log^{2}(1 - tz) - 32\log^{3}(1 - tz)\right)\frac{1}{t},\\
  \frac{d^{5}}{d\nu^{5}} \nu t^{\nu}((1 - tz)^{-2\nu} - 1)\bigg|_{\nu = 0}
  &= \Big(-40\log^{3}(t)\log(1 - tz) + 120\log^{2}(t)\log^{2}(1 - tz)\\
  &\qquad- 160\log(t)\log^{3}(1 - tz) + 80\log^{4}(1 - tz)\Big)\frac{1}{t},\\
\end{align*}

In general the integral has an integrable singularity at \(t = 0\), to
handle this we split the integral as
\begin{equation*}
  S_{n}(z) = \frac{1}{n!}\int_{0}^{a}\frac{d^{n}}{d\nu^{n}} \nu t^{\nu}((1 - tz)^{-2\nu} - 1)\bigg|_{\nu = 0}\frac{dt}{t}
  + \frac{1}{n!}\int_{a}^{1}\frac{d^{n}}{d\nu^{n}} \nu t^{\nu}((1 - tz)^{-2\nu} - 1)\bigg|_{\nu = 0}\frac{dt}{t}
  = \frac{1}{n!}S_{n,1}(z) + \frac{1}{n!}S_{n,2}(z),
\end{equation*}
with \(0 < a < 1\) (in the computations we take
\(a = \codenumber{10^{-8}}\)). In the special case \(n = 2\) the
singularity at \(t = 0\) is removable, and we can take \(a = 0\).

To compute \(S_{n,1}\) we use that \(\log(1 - tz) / t\) and
\(\log(1 - tz)\) are both bounded near zero. We then apply the
generalized mean value theorem for integrals applied to the real and
imaginary parts separately. For \(n = 3\) this gives us that
\begin{multline*}
  S_{3,1}(z) = -12\left(
    \real\left(\frac{\log(1 - \xi_{1} z)}{\xi_{1}}\right)
    + i\imag\left(\frac{\log(1 - \xi_{2} z)}{\xi_{2}}\right)
  \right)\int_{0}^{a}\log t\ dt\\
  + 12\left(
    \real\left(\frac{\log(1 - \xi_{3} z)}{\xi_{3}}\log(1 - \xi_{3}z)\right)
    + i\imag\left(\frac{\log(1 - \xi_{4} z)}{\xi_{4}}\log(1 - \xi_{4}z)\right)
  \right)\int_{0}^{a}\ dt,
\end{multline*}
for some \(0 < \xi_{1}, \xi_{2}, \xi_{3}, \xi_{4} < a\). Since FLINT
represents complex intervals using rectangles it is possible to write
this in a simpler way using interval arithmetic notation. If we let
\(f_{1}(t) = \frac{\log(1 - tz)}{t}\) and \(f_{2}(t) = \log(1 - tz)\), then we have
\begin{equation*}
  S_{3,1}(z) \in f_{1}([0, a])\left(
    -12\int_{0}^{a}\log t\ dt
    + 12f_{2}([0, a])\int_{0}^{a}\ dt
  \right),
\end{equation*}
where \(f_{1}([0, a])\) and \(f_{2}([0, a])\) will be complex
rectangles that enclose the images of \(f_{1}\) and \(f_{2}\) on the
interval \([0, a]\). For \(f_{2}\) it is straightforward to compute an
enclosure using regular interval arithmetic, for \(f_{1}\) an
enclosure can be computed by handling the removable singularity. For
\(S_{4,1}\) and \(S_{5,1}\) we similarly get
\begin{align*}
  S_{4,1}(z) &\in f_{1}([0, a])\left(
               -24 \int_{0}^{a}\log^{2}(t)\ dt
               + 48f_{2}([0, a])\int_{0}^{a}\log(t)\ dt
               - 32f_{2}([0, a])^{2}\int_{0}^{a}\ dt
               \right),\\
  S_{5,1}(z) &\in f_{1}([0, a])\Bigg(
               -40 \int_{0}^{a}\log^{3}(t)\ dt
               + 120f_{2}([0, a])\int_{0}^{a}\log^{2}(t)\ dt\\
             &\qquad\qquad\qquad\qquad- 160f_{2}([0, a])^{2}\int_{0}^{a}\log(t)\ dt
               + 80f_{2}([0, a])^{3}\int_{0}^{a}\ dt
               \Bigg).
\end{align*}
The remaining integrals can be computed explicitly using
Lemma~\ref{lemma:basic-integrals}.

For \(S_{n,2}\) the integrand is in general finite on the interval of
integration, the exception is when \(z = 1\) in which case the
integrand has an integrable singularity at \(t = 1\). When our
interval enclosure of \(z\) does not contain \(1\) we compute the
entire integral using the techniques for rigorous numerical
integration discussed in Appendix~\ref{sec:rigorous-integration}. If
\(z\) overlaps \(1\) then we further split \(S_{n,2}\) as
\begin{equation*}
  S_{n,2} = \int_{a}^{b}\frac{d^{n}}{d\nu^{n}} \nu t^{\nu}((1 - tz)^{-2\nu} - 1)\bigg|_{\nu = 0}\frac{dt}{t}
  + \int_{b}^{1}\frac{d^{n}}{d\nu^{n}} \nu t^{\nu}((1 - tz)^{-2\nu} - 1)\bigg|_{\nu = 0}\frac{dt}{t}
  = S_{n,2,1}(z) + S_{n,2,2}(z),
\end{equation*}
with \(a < b < 1\) (in the computations we take
\(b = \codenumber{1 - 10^{-8}}\)). The integral \(S_{n,2,1}\) is
computed directly.

For \(S_{n,2,2}\) the idea is to use that \(\log(t)\) and \(1 / t\)
are bounded near \(t = 1\) to factor them out from the integrals. For
this we want to split into real and imaginary parts and use the
generalized mean value theorem for integrals. Assuming that the real
and imaginary parts of \(\log(1 - tz)\) do not change sign, then for
\(S_{2,2,2}\) we have that
\begin{equation*}
  S_{2,2,2}(z) = -4 \frac{1}{\xi_{1}} \int_{b}^{1} \real(\log(1 - tz))\ dt
  + -4 \frac{1}{\xi_{2}} \int_{b}^{1} \imag(\log(1 - tz))\ dt,
\end{equation*}
for some \(b < \xi_{1}, \xi_{2} < 1\). If we let
\(g_{1}(t) = \frac{1}{t}\) then we can write this in interval
arithmetic notation as
\begin{equation*}
  S_{2,2,2}(z) \in g_{1}([b, 1])\left(
    -4 \int_{b}^{1} \real(\log(1 - tz))\ dt
    + -4 \int_{b}^{1} \imag(\log(1 - tz))\ dt,
  \right)
  = -4g_{1}([b, 1])\int_{b}^{1} \log(1 - tz)\ dt.
\end{equation*}
We can apply the same idea for \(n = 3, 4, 5\) if we assume that the
real and imaginary parts of \(\log^{m}(1 - tz)\) do not change sign
for \(m = 1, 2, 3, 4\). If we let \(g_{2}(t) = \log(t)\) we have
\begin{align*}
  S_{3,2,2}(z) &\in g_{1}([b, 1])\left(
                 -12g_{2}([b, 1])\int_{b}^{1} \log(1 - tz)\ dt
                 + 12\int_{b}^{1} \log^{2}(1 - tz)\ dt
                 \right),\\
  S_{4,2,2}(z) &\in g_{1}([b, 1])\left(
                 -24g_{2}([b, 1])^{2}\int_{b}^{1} \log(1 - tz)\ dt
                 + 48g_{2}([b, 1])\int_{b}^{1} \log^{2}(1 - tz)\ dt
                 - 32\int_{b}^{1} \log^{3}(1 - tz)\ dt
                 \right),\\
  S_{5,2,2}(z) &\in g_{1}([b, 1])\Bigg(
                 -40g_{2}([b, 1])^{3}\int_{b}^{1} \log(1 - tz)\ dt
                 + 120g_{2}([b, 1])^{2}\int_{b}^{1} \log^{2}(1 - tz)\ dt\\
               &\qquad\qquad\qquad- 160g_{2}([b, 1])\int_{b}^{1} \log^{3}(1 - tz)\ dt
                 + 80\int_{b}^{1} \log^{4}(1 - tz)\ dt
                 \Bigg).
\end{align*}
The remaining integrals can then be computed using
Lemma~\ref{lemma:basic-integrals}.

To make use of this we need to verify that the real and imaginary
parts of \(\log^{m}(1 - tz)\) do not change sign for \(t \in (b, 1)\).
We formulate the result in the following lemma.

\begin{lemma}\label{lemma:log-m-signs}
  Let \(z \in \mathbb{C}\) and \(0 < b < 1\) and assume that for all
  \(t \in [b, 1]\) we have \(|1 - tz| \leq C\) for some \(C < 1\). Let
  \(\theta = \arctan\left(\pi/ |\log(C)|\right)\), if \(m \geq 1\) is
  an integer such that \(m\theta \leq \pi/2\) then the real and
  imaginary parts of \(\log^{m}(1 - tz)\) do not change sign for all
  \(t \in [b, 1]\).
\end{lemma}

\begin{proof}
  Let us start with the case when \(\imag z \leq 0\). Geometrically,
  showing that the real and imaginary parts of \(\log^{m}(1 - tz)\) do
  not change sign is equivalent to showing that the values of the
  function are confined to a single quadrant in the complex plane.

  Consider the base term \(\log(1 - tz)\). Decomposing this into real
  and imaginary parts, we have
  \begin{equation*}
    \log(1 - tz) = \log|1 - tz| + i\arg(1 - tz).
  \end{equation*}
  From the assumption that \(|1 - tz| < C < 1\) we get that
  \(\log|1 - tz| \in (-\infty, \log(C)]\), with \(\log(C) < 0\). Using
  that \(\imag z \leq 0\) we get that \(\imag(1 - tz) \geq 0\) for all
  \(t \in [b, 1]\), hence \(\arg(1 - tz) \in [0, \pi]\). It follows
  that the values of \(\log(1 - tz)\) are contained within the
  semi-infinite strip
  \begin{equation*}
    (-\infty, \log(C)] \times i[0, \pi].
  \end{equation*}
  As illustrated in Figure~\ref{fig:log-sector}, this strip is
  subtended by an angle \(\theta\) at the origin, measured from the
  negative real axis. By simple trigonometry on the vertex of the
  strip \((\log(C), \pi)\), we get that this angle is
  \begin{equation*}
    \theta = \arctan\left(\frac{\pi}{|\log(C)|}\right).
  \end{equation*}
  The argument of \(\log(1 - tz)\) is therefore restricted to the
  interval \([\pi - \theta, \pi]\). Extending this to the \(m\)-th
  power, the argument of \(\log^m(1 - tz)\) lies in the interval
  \([m\pi - m\theta, m\pi]\). For this sector to fit entirely within a
  single quadrant, the total opening angle must satisfy
  \(m\theta \leq \frac{\pi}{2}\), which is precisely the condition on
  \(m\) in the statement.

  The case \(\imag z > 0\) follows in the same way.

  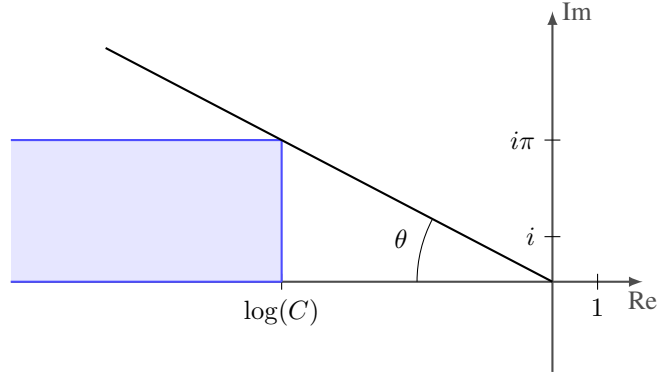
\begin{figure}[h]
    \centering
    \begin{tikzpicture}[
      scale = 0.6,
      >=latex, % For arrow style
      ]

      % Define C
      \def\logC{-6}
      \def\thetadeg{27.6364} % Approximate value for theta in degrees
      \def\PiVal{3.14159} % Approximate value for pi

      % Axis limits
      \def\xmin{-12}
      \def\xmax{2}
      \def\ymin{-2}
      \def\ymax{6}

      % Draw the axes
      \draw[thick,black!70,->] (\xmin,0) -- (\xmax,0) node[below] {Re};
      \draw[thick,black!70,->] (0,\ymin) -- (0,\ymax) node[right] {Im};

      % Draw axis ticks and labels
      % \foreach \x in {\xmin,...,-1,1,...,\xmax} \draw (\x,1pt) -- (\x,-1pt) node[below] {\x};
      % \draw (0,0) node[below left] {$0$};
      \draw (1,5pt) -- (1,-5pt) node[below] {$1$};
      \draw (\logC,5pt) -- (\logC,-5pt) node[below] {$\log(C)$};
      % \foreach \y in {1,2,3} \draw (1pt,\y) -- (-1pt,\y) node[left] {\y};
      \draw (5pt,1) -- (-5pt,1) node[left] {$i$};
      \draw (5pt,\PiVal) -- (-5pt,\PiVal) node[left] {$i\pi$};

      % Draw the strip region $(-\infty, -C] \times i(0, \pi)$
      % We approximate -infinity with the left edge of our canvas (xmin)
      \draw[fill=blue!10, draw=none] (\xmin, 0) rectangle (\logC, \PiVal);

      % Draw the boundary lines

      % Boundary at x=-C
      \draw[thick, blue!70] (\logC, 0) -- (\logC, \PiVal);
      % Boundaries at y=0 and y=pi
      \draw[thick, blue!70] (\xmin, 0) -- (\logC, 0);
      \draw[thick, blue!70] (\xmin, \PiVal) -- (\logC, \PiVal);

      % Draw sector
      \draw[thick] (0, 0) -- (1.65 * \logC, 1.65 * \PiVal);

      \draw (0.5*\logC, 0) arc (180:180-\thetadeg:-0.5*\logC);
      \draw (0.5 * \logC, 0.6 * 0.5 * \PiVal) node[left] {\(\theta\)};
    \end{tikzpicture}
    \caption{Geometric bounds for \(\log(1 - tz)\) when
      \(|1 - tz| \leq C < 1\) and \(\imag z \leq 0\). The values are
      confined to the blue semi-infinite strip
      \((-\infty, \log(C)] \times i[0, \pi]\). This region is enclosed
      by a cone of angle \(\theta\) relative to the negative real
      axis, where \(\theta = \arctan(\pi / |\log C|)\).}
    \label{fig:log-sector}
  \end{figure}
\end{proof}

\subsection{Taylor models}
\label{sec:taylor-models}
On several occurrences we expand functions in terms of \(N^{-1}\) up
to some finite degree and need a bound on the remainder term that is
valid for all \(N \geq N_{0}\) for some \(N_{0}\). Examples of such
remainder terms are the functions \(T_{6}(N, z)\) and \(K_{4}(N, z, t)\)
from
Lemmas~\ref{lemma:bounds_and_splitting:first-term-J0:expansion-of-argument-inside-g-function}
and~\ref{lemma:exact-expression-for-K-up-to-order_-3} respectively, as
well as the function bounded by \(C_{F_{N},4}\) in
Lemma~\ref{lemma:K4-bounds-preliminary-constants}. The main tool to
bound such remainder terms will be the use of Taylor models. We here
give a brief introduction to Taylor models, for a more thorough
introduction we refer to~\cite{Joldes2011}, and discuss the details on
how to use them to bound the relevant remainder terms.

With a Taylor model we mean what in~\cite{Joldes2011} is referred to
as a \emph{Taylor model with relative error}. For a real valued
function \(f\) we use the following definition (compare
with~\cite[Definition 2.3.2]{Joldes2011}).
\begin{definition}
  A Taylor model \(M = (p, \Delta)\) of degree \(n\) for a function
  \(f\) on an interval \(I\) centered at a point \(x_{0} \in I\) is a
  polynomial \(p\) of degree \(n\) together with an interval
  \(\Delta\), satisfying that for all \(x \in I\) there is
  \(\delta \in \Delta\) such that
  \begin{equation*}
    f(x) - p(x - x_{0}) = \delta (x - x_{0})^{n + 1}.
  \end{equation*}
\end{definition}
Many of the functions we make use of are complex valued, for this case
we have the following definition. Note that even if the function is
complex valued its domain is still real (in our case the domain will
correspond to \(N^{-1}\)).
\begin{definition}
  A complex valued Taylor model \(M = (p, \Delta)\) of degree \(n\)
  for a complex valued function \(f\) on a real interval \(I\)
  centered at a point \(x_{0} \in I\) is a (complex) polynomial \(p\)
  of degree \(n\) together with a complex rectangle
  \(\Delta = \Delta_{\real} \times i\Delta_{\imag}\), satisfying that
  for all \(x \in I\) there is
  \(\delta = \delta_{\real} + i\delta_{\imag} \in \Delta_{\real}
  \times i\Delta_{\imag}\) such that
  \begin{equation*}
    f(x) - p(x - x_{0}) = \delta (x - x_{0})^{n + 1}.
  \end{equation*}
\end{definition}

In the case that \(f\) is an \(n + 1\) times differentiable real
function, the polynomial \(p\) is the Taylor polynomial of degree
\(n\) of \(f\) centered at \(x_{0}\)~\cite[Lemma 2.3.3]{Joldes2011}. A
Taylor model is thus given by a truncated Taylor expansion plus a
bound on the remainder term valid on some interval \(I\). From
Taylor's theorem we see that we can take \(\Delta\) to be an enclosure
of \(\frac{f^{n + 1}(x)}{(n + 1)!}\) on the interval \(I\). The
situation is the same for a complex valued function \(f\), which can
be seen by considering the real and imaginary parts separately. Note
that in this case \(\Delta_{\real}\) and \(\Delta_{\imag}\) are
enclosures of \(\frac{\real(f^{n + 1}(x))}{(n + 1)!}\) and
\(\frac{\imag(f^{n + 1}(x))}{(n + 1)!}\) respectively. In case \(f\)
is holomorphic one can use more powerful tools, but these ones
suffices for our purposes.

We can perform arithmetic on Taylor models. Given two functions \(f\)
and \(g\) with corresponding Taylor models
\(M_{f} = (p_{f}, \Delta_{f})\) and \(M_{g} = (p_{g} + \Delta_{g})\)
we can compute a Taylor model of \(f + g\) as
\(M_{f + g} = (p_{f} + p_{g}, \Delta_{f} + \Delta_{g})\), similarly
for \(f - g\). With slightly more work we can compute a Taylor model
of \(f \cdot g\), see~\cite[Algorithm 2.3.6]{Joldes2011}, as well as
of \(f / g\), see~\cite[Algorithm 2.3.12]{Joldes2011}. We can also
compose Taylor models with arbitrary functions, given a Taylor model
\(M_{f}\) of \(f\) and a function \(g\) we can compute a Taylor model
\(M_{g \circ f}\) of \(g \circ f\), see~\cite[Algorithm
2.3.8]{Joldes2011}. Note that while the algorithms
in~\cite{Joldes2011} are described for real valued Taylor models, they
work equally well for complex valued Taylor models.

One particular case that will be relevant for us is to compute a (real
valued) Taylor model of \(|f|\) from a complex valued Taylor model of
\(f\). If \(M_{f} = (p_{f}, \Delta_{f})\) is a complex valued Taylor
model for the function \(f\), then we want to find a Taylor model
\(M_{|f|} = (p_{|f|}, \Delta_{|f|})\) for the function \(|f|\). For
this we use that
\begin{equation}\label{eq:taylor-model-abs}
  |f(x)| = \sqrt{f(x)\overline{f(x)}}.
\end{equation}
Since \(x \in I\) is real valued, we have
\begin{equation*}
  \overline{f(x)} = \overline{p_{f}(x) + \delta (x - x_{0})^{n + 1}}
  = \overline{p_{f}}(x) + \bar{\delta}(x - x_{0})^{n + 1},
\end{equation*}
where \(\overline{p_{f}}\) denotes the conjugate polynomial of
\(p_{f}\). It follows that
\(M_{\bar{f}} = (\overline{p_{f}}, \bar{\Delta})\) is a Taylor model
for \(\bar{f}\). We can get a Taylor model for \(f(x)\overline{f(x)}\)
by simplify multiplying the Taylor models \(M_{f}\) and
\(M_{\bar{f}}\), note further that this Taylor model is real valued.
Finally, we can get a Taylor model for \(|f|\) composing the product
of \(M_{f}\) and \(M_{\bar{f}}\) with the square-root function.

\subsubsection{Taylor model for \(F_N(z)\)}
\label{sec:taylor-model-F_N}
For many functions we can automatically compute Taylor models using
the tools for Taylor arithmetic that are available in FLINT. For
\(F_{N}(z)\), no such tools are however available and computation of
the Taylor model requires a separate implementation.

For a fixed \(z\) we want to compute a Taylor model of \(F_{N}(z)\) in
the variable \(\nu = N^{-1}\) that is centered at \(\nu = 0\) and
valid on the interval \([0, 1 / N_{0}]\) for some \(N_{0}\) (in
practice \(N_{0} = \codenumber{64}\)). For our purposes it suffices
with a Taylor model of degree \(5\). From
\eqref{eq:FNz_expansion_powers_1/N} we immediately get that
the polynomial for the Taylor model is
\begin{equation*}
  p(\nu) = 1 + \sum_{n = 1}^{5} S_{n}(z)\nu^{n}.
\end{equation*}
The only thing that remains is to compute an enclosure for \(\Delta\).
For this we want an enclosure of
\begin{equation*}
  \frac{d^{6}}{d\nu^{6}} F_{\nu^{-1}}(z)
\end{equation*}
for \(\nu \in [0, 1 / N_{0}]\). From Equation~\eqref{eq:FNz_integral}
we have
\begin{equation*}
  \frac{d^{6}}{d\nu^{6}} F_{\nu^{-1}}(z)
  = \int_{0}^{1} \frac{d^{6}}{d\nu^{6}}  \nu t^{\nu}((1 - tz)^{-2\nu} - 1)\frac{dt}{t}.
\end{equation*}
Let us denote the integral by \(R_{F}(\nu, z)\). To compute an
enclosure of \(R_{F}(\nu, z)\) we will follow the same approach as
when computing \(S_{n}(z)\) in Appendix~\ref{sec:polylogs}. That is,
we write
\begin{multline*}
  R_{F}(\nu, z) =
  \int_{0}^{a} \frac{d^{6}}{d\nu^{6}}  \nu t^{\nu}((1 - tz)^{-2\nu} - 1)\frac{dt}{t}
  + \int_{a}^{b} \frac{d^{6}}{d\nu^{6}}  \nu t^{\nu}((1 - tz)^{-2\nu} - 1)\frac{dt}{t}\\
  + \int_{b}^{1} \frac{d^{6}}{d\nu^{6}}  \nu t^{\nu}((1 - tz)^{-2\nu} - 1)\frac{dt}{t}
  = R_{F,1}(\nu, z) + R_{F,2,1}(\nu, z) + R_{F,2,2}(\nu, z),
\end{multline*}
where \(0 < a < b \leq 1\). In the computations we take
\(a = \codenumber{10^{-8}}\) and if \(z\) doesn't overlap \(1\) we
take \(b = 1\) (so \(R_{F,2,2}\) is zero), otherwise we take
\(b = 1 - \codenumber{10^{-8}}\). To enclose \(R_{F,2,1}(\nu, z)\) we
use the rigorous numerical integrator discussed in
Appendix~\ref{sec:rigorous-integration}, with the sixth order
derivative in the integrator computed automatically using Taylor
arithmetic.

To enclose \(R_{F,1}\) we note that the integrand can be written as
\begin{align*}
  \frac{d^{6}}{d\nu^{6}}  \nu t^{\nu}((1 - tz)^{-2\nu} - 1)\frac{1}{t}
  &= t^{\nu}\Bigg(\nu\frac{(1 - tz)^{-2\nu} - 1}{t}\log^{6}(t)\\
  &\qquad\qquad- 6\left(2\nu(1 - tz)^{-2\nu}\frac{\log(1 - tz)}{t} - \frac{(1 - tz)^{-2\nu} - 1}{t}\right)\log^{5}(t)\\
  &\qquad\qquad+ 60(\nu\log(1 - tz) - 1)(1 - tz)^{-2\nu}\frac{\log(1 - tz)}{t}\log^{4}(t)\\
  &\qquad\qquad- 80(2\nu\log(1 - tz) - 3)\log(1 - tz)(1 - tz)^{-2\nu}\frac{\log(1 - tz)}{t}\log^{3}(t)\\
  &\qquad\qquad+ 240(\nu\log(1 - tz) - 2)\log^{2}(1 - tz)(1 - tz)^{-2\nu}\frac{\log(1 - tz)}{t}\log^{2}(t)\\
  &\qquad\qquad- 96(2\nu\log(1 - tz) - 5)\log^{3}(1 - tz)(1 - tz)^{-2\nu}\frac{\log(1 - tz)}{t}\log(t)\\
  &\qquad\qquad+ 64(\nu\log(1 - tz) - 3)\log^{4}(1 - tz)(1 - tz)^{-2\nu}\frac{\log(1 - tz)}{t}\Bigg).
\end{align*}
At \(t = 0\) the terms \(\frac{(1 - tz)^{-2\nu} - 1}{t}\) and
\(\frac{\log(1 - tz)}{t}\) have removable singularities that can be
handled, and the factor \(t^{\nu}\) is contained in \([0, 1]\). All
other parts except the \(\log(t)\) factors are also bounded near
\(t = 0\). To enclose the integral we thus follow the same procedure
as for enclosing \(S_{n,1}\) in Appendix~\ref{sec:polylogs}.

For \(R_{F,2,2}\) we instead write the derivative as
\begin{align*}
  \frac{d^{6}}{d\nu^{6}}  \nu t^{\nu}((1 - tz)^{-2\nu} - 1)\frac{1}{t}
  &= t^{\nu - 1}\Big(64\nu \log^{6}(1 - tz)(1 - tz)^{-2\nu}\\
  &\qquad\qquad- 192(\nu\log(t) + 1)\log^{5}(1 - tz)(1 - tz)^{-2\nu}\\
  &\qquad\qquad+ 240(\nu\log(t) + 2)\log(t)\log^{4}(1 - tz)(1 - tz)^{-2\nu}\\
  &\qquad\qquad- 160(\nu\log(t) + 3)\log^{2}(t)\log^{3}(1 - tz)(1 - tz)^{-2\nu}\\
  &\qquad\qquad+ 60(\nu\log(t) + 4)\log^{3}(t)\log^{2}(1 - tz)(1 - tz)^{-2\nu}\\
  &\qquad\qquad- 12(\nu\log(t) + 5)\log^{4}(t)\log(1 - tz)(1 - tz)^{-2\nu}\\
  &\qquad\qquad+ (\nu\log(t) + 6)\log^{5}(t)((1 - tz)^{-2\nu} - 1)\Big).
\end{align*}
We integrate this termwise and factor out everything except the
factors \(\log^{m}(1 - tz)(1 - tz)^{-2\nu}\). The remaining integrals
are of the form
\begin{equation*}
  \int_{b}^{1} \log^{m}(1 - tz)(1 - tz)^{-2\nu}\ dt,
\end{equation*}
and can be computed using Lemma~\ref{lemma:basic-integrals}.

For this to be valid we must, similar to for \(S_{n,2,2}\) in
Appendix~\ref{sec:polylogs}, ensure that the real and imaginary parts
of \(\log^{m}(1 - tz)(1 - tz)^{-2\nu}\) do not change sign. For this
we have the following lemma, that builds upon
Lemma~\ref{lemma:log-m-signs}.

\begin{lemma}
  Let \(z \in \mathbb{C}\), \(\nu \in [0, 1 / N_{0}]\), \(0 < b < 1\) and
  assume that for all \(t \in [b, 1]\) we have \(|1 - tz| \leq C\) for
  some \(C < 1\). Let \(\theta = \arctan\left(\pi/ |\log(C)|\right)\),
  if \(m \geq 1\) is an integer such that
  \(m\theta + 2\pi / N_{0} \leq \pi/2\) then the real and imaginary
  parts of \(\log^{m}(1 - tz)(1 - tz)^{-2\nu}\) do not change sign for
  all \(t \in [b, 1]\).
\end{lemma}

\begin{proof}
  Similar to in Lemma~\ref{lemma:log-m-signs} we give a proof for
  \(\imag z \leq 0\), with the case \(\imag z > 0\) following in a
  similar way.

  From the proof of Lemma~\ref{lemma:log-m-signs} we have that the
  argument of \(\log^m(1 - tz)\) lies in the interval
  \([m\pi - m\theta, m\pi]\). For the second factor we note that
  \(\imag z \leq 0\) means that the argument of \(1 - tz\) lies in the
  interval \([0, \pi]\), hence the argument of \((1 - tz)^{-2\nu}\)
  lies in \([-2\nu\pi, 0]\). It follows that
  \begin{equation*}
    \arg(\log^{m}(1 - tz)(1 - tz)^{-2\nu}) \in [m\pi - m\theta - 2\nu\pi, m\pi],
  \end{equation*}
  and for this to lie in a single quadrant we must have
  \(m\theta + 2\nu\pi \leq \frac{\pi}{2}\). Since
  \(\nu \in [0, 1 / N_{0}]\), the result follows.
\end{proof}

\subsubsection{Bounding remainder terms using Taylor models}
\label{sec:taylor-models-remainders}
With an implementation of Taylor models for \(F_{N}\) it is relatively
straightforward to compute enclosures of the required remainder terms.
They are all given by the remainder term of a Taylor model of the
associated function.

The function \(T_{6}(N, z)\) is given in Equation~\eqref{eq:T_246}.
Consider a degree 5 Taylor model \(M_{z} = (p_{z}, \Delta_{z})\) of
the function
\begin{equation*}
  \frac{\sqrt{\lambda_{app}}}{\sqrt{\lambda}} c_N |F_N(z)|
\end{equation*}
in the variable \(\nu = N^{-1}\) that is valid on the interval
\([0, 1 / N_{0}]\). By definition of \(\Delta_{z}\) we then have
\begin{equation*}
  T_{6}(N, z) \in \Delta_{z}
\end{equation*}
for any \(\nu \in [0, 1 / N_{0}]\). For \(K_{4}(N, z, t)\) the approach
is similar, in this case it is enclosed by the remainder term of a
degree 3 Taylor model for
\begin{equation*}
  K(z, t) = J_0 \left(\rho^{1/2} |z|^{1/N} \sqrt{F_N(\bar{z}) \left( F_N(z) - (t/z)^{1/N} F_N(t) \right)}\right) \, ,
\end{equation*}
given in Equation~\eqref{eq:definition-of-K(z,t)}. Using that
\(J_{0}(z) = {}_0\tilde{F}_1(1; -z^{2} / 4)\) we can write this as
\begin{equation*}
  K(z, t) = {}_0\tilde{F}_1\left(1; -\frac{1}{4}\rho |z|^{2/N} \left(F_N(\bar{z}) \left( F_N(z) - (t/z)^{1/N} F_N(t) \right)\right)\right) \, .
\end{equation*}
To bound \(C_{F_{N},4}\) in
Lemma~\ref{lemma:K4-bounds-preliminary-constants} we need to compute
a degree 3 Taylor model of
\begin{equation*}
  \frac{\rho^{1/2}}{\lambda^{1/2}}F_N(t).
\end{equation*}

In all cases we compute Taylor models of \(F_{N}\) using the approach
discussed in the above subsection. For the rest, Taylor models can be
computed automatically using Taylor arithmetic and composition of
Taylor models. In a couple of cases we initially compute higher degree
Taylor models which are then truncated to the appropriate degree, this
generally improves the quality of the enclosure of the remainder term.

\section{Bounds for $K_4(N,z,t)$}
\label{app:bounds-for-K4}

This section is devoted to computing bounds on the residue function
\begin{equation*}
  K_4(N,z,t) = \left(K(z,t) - d_0(z,t) - \frac{d_1(z,t)}{N} -
  \frac{d_2(z,t)}{N^2} - \frac{d_3(z,t)}{N^3}\right)N^{4}
\end{equation*}
from
Lemma~\ref{lemma:exact-expression-for-K-up-to-order_-3}
that are valid for \(t\) close to zero.

We first begin by establishing bounds on some functions that will
appear later.
\begin{lemma}
  \label{lemma:K4-bounds-preliminary-constants}
  Let \(N \geq N_{0} = \codenumber{64}\). We have the enclosure
  \begin{equation}\label{eq:rho-lambda-enclosure}
    0 \leq 1 - \frac{\rho^{1/2}}{\lambda^{1/2}} \leq \frac{C_{\rho,\lambda}}{N^5},
  \end{equation}
  for
  \begin{equation*}
    C_{\rho,\lambda} = \codenumber{7}.
  \end{equation*}
  For \(|t| \leq 1\) we have the bounds
  \begin{equation*}
    |S_2(t)| \leq C_{S_2},\quad
    |S_3(t)| \leq C_{S_3}
  \end{equation*}
  with
  \begin{equation*}
    C_{S_{2}} = \codenumber{3.3},\quad
    C_{S_{3}} = \codenumber{2.5}.
  \end{equation*}
  Moreover, we have the bounds
  \begin{align*}
    \left|\frac{\rho^{1/2}}{\lambda^{1/2}}F_N(t)\right| &\leq C_{F_N,0},\\
    \left|\frac{\rho^{1/2}}{\lambda^{1/2}}F_N(t) - 1\right| &\leq \frac{C_{F_N,2}}{N^2},\\
    \left| \frac{\rho^{1/2}}{\lambda^{1/2}} \left(F_N(t) - \frac{S_2(t)}{N^2}\right) - 1 \right| &\leq \frac{C_{F_N,3}}{N^3},\\
    \left| \frac{\rho^{1/2}}{\lambda^{1/2}} \left(F_N(t) - \frac{S_2(t)}{N^2} - \frac{S_3(t)}{N^3} \right) - 1 \right| &\leq \frac{C_{F_N,4}}{N^4}
  \end{align*}
  with
  \begin{equation*}
    C_{F_{N},0} = \codenumber{1.001},\quad
    C_{F_{N},2} = \codenumber{3.5},\quad
    C_{F_{N},3} = \codenumber{3},\quad
    C_{F_{N},4} = \codenumber{11}.
  \end{equation*}
  Finally, we have the bound
  \begin{equation*}
    \|J_{0}^{(8)}\|_{L^{\infty}([0, \infty))} \leq C_{J_{0},8}
  \end{equation*}
  with \(C_{J_{0},8} = \codenumber{0.2734376}\).
\end{lemma}

\begin{proof}
  Let us start with the bound involving \(C_{\rho,\lambda}\).
  Recall that \(\rho = \rho(N) = c_{N}^{2}\lambda_{app}(N)\). Let us
  denote \(\nu = N^{-1}\), we are then interested in controlling
  \(\rho(1 / \nu)\) for \(\nu \in [0, 1 / N_{0}]\). By construction,
  \(\lambda_{app}\) is taken such that for the Taylor expansion at
  \(\nu = 0\), all terms up to \(\nu^{4}\) in
  \begin{equation*}
    1 - \frac{\rho(1 / \nu)^{1/2}}{\lambda^{1/2}}
  \end{equation*}
  vanish. To control \(1 - \frac{\rho(1 / \nu)^{1/2}}{\lambda^{1/2}}\),
  it therefore suffices to compute an enclosure of the fifth term in
  the Taylor expansion that is valid for \(\nu \in [0, 1 / N_{0}]\).
  More precisely we have for \(\nu \in [0, 1 / N_{0}]\) that
  \begin{equation*}
    1 - \frac{\rho(1 / \nu)^{1/2}}{\lambda^{1/2}} = \frac{1}{5!}\frac{d^{5}}{d\nu^{5}}\left(1 - \frac{\rho(1 / \nu)^{1/2}}{\lambda^{1/2}}\right)\Big|_{\nu = \xi}\nu^{5},
  \end{equation*}
  for some \(\xi \in [0, 1 / N_{0}]\). We enclose the derivative on
  the interval \([0, 1 / N_{0}]\) using the approach from
  \eqref{eq:taylor-bound} combined with automatic Taylor
  expansions, giving
  \begin{equation*}
    \frac{1}{5!}\frac{d^{5}}{d\nu^{5}}\left(1 - \frac{\rho(1 / \nu)^{1/2}}{\lambda^{1/2}}\right) \in \resultnumber{[6 \pm 0.633]}.
  \end{equation*}
  Since the enclosure is positive we get that
  \(1 - \frac{\rho^{1/2}}{\lambda^{1/2}}\) is lower bounded by zero.
  Since it is upper bounded by \(C_{\rho,\lambda}\) we get that
  \begin{equation*}
    1 - \frac{\rho^{1/2}}{\lambda^{1/2}} \leq C_{\rho,\lambda}\nu^{5} = \frac{C_{\rho,\lambda}}{N^{5}}.
  \end{equation*}

  Next we look at the bounds for \(S_{2}(t)\) and \(S_{3}(t)\). Since
  both of these functions are analytic on the unit disk it suffices to
  bound the modulus on the boundary of the disk. Moreover, they are
  conjugate symmetric and it hence suffices to bound them for
  \(t = e^{i\theta}\) with \(\theta \in [0, \pi]\). We make use of the
  algorithm discussed in Appendix~\ref{sec:enclosing-extrema} for
  enclosing the maximum on this interval. The procedure to compute
  enclosures of \(S_{2}\) and \(S_{3}\) is discussed in
  Appendix~\ref{sec:polylogs}. From this we get the enclosures
  \begin{equation*}
    \max_{\theta \in [0, \pi]} |S_{2}(e^{i\theta})| \in \resultnumber{[3.29 \pm 1.28 \cdot 10^{-3}]},\quad
    \max_{\theta \in [0, \pi]} |S_{3}(e^{i\theta})| \in \resultnumber{[2.44 \pm 1.98 \cdot 10^{-3}]},
  \end{equation*}
  which implies the proposed bounds.

  For the bounds related to \(F_{N}\) we start by noting that it
  suffices to prove the bound with \(C_{F_{N},4}\), the other bounds
  then follow from
  \begin{align*}
    \left|\frac{\rho^{1/2}}{\lambda^{1/2}}F_N(t)\right|
    &\leq 1 + \left|\frac{\rho^{1/2}}{\lambda^{1/2}}\frac{S_2(t)}{N^2}\right| + \left|\frac{\rho^{1/2}}{\lambda^{1/2}}\frac{S_3(t)}{N^3}\right| + \frac{C_{F_N,4}}{N^4}
      \leq 1 + \frac{C_{S_{2}}}{N_{0}^{2}} + \frac{C_{S_{3}}}{N_{0}^{3}} + \frac{C_{F_N,4}}{N_{0}^{4}}
      \leq C_{F_{N},0},\\
    \left|\frac{\rho^{1/2}}{\lambda^{1/2}}F_N(t) - 1\right|
    &\leq \left|\frac{\rho^{1/2}}{\lambda^{1/2}}\frac{S_2(t)}{N^2}\right| + \left|\frac{\rho^{1/2}}{\lambda^{1/2}}\frac{S_3(t)}{N^3}\right| + \frac{C_{F_N,4}}{N^4}
      \leq \left(C_{S_{2}} + \frac{C_{S_{3}}}{N_{0}} + \frac{C_{F_N,4}}{N_{0}^{2}}\right)\frac{1}{N^{2}}
      \leq \frac{C_{F_{N},2}}{N^{2}},\\
    \left| \frac{\rho^{1/2}}{\lambda^{1/2}} \left(F_N(t) - \frac{S_2(t)}{N^2} \right) - 1 \right|
    & \leq \left|\frac{\rho^{1/2}}{\lambda^{1/2}}\frac{S_3(t)}{N^3}\right| + \frac{C_{F_N,4}}{N^4}
      \leq \left(C_{S_{3}} + \frac{C_{F_N,4}}{N_{0}}\right)\frac{1}{N^{3}}
      \leq \frac{C_{F_{N},3}}{N^{3}}.
  \end{align*}
  Where we have used the bound
  \(\left|\frac{\rho^{1/2}}{\lambda^{1/2}}\right| \leq 1\), which
  follows from~\eqref{eq:rho-lambda-enclosure}.

  We next move to establishing the bound with \(C_{F_{N},4}\). Similar
  to for \(S_{2}\) and \(S_{3}\) we start by noticing that
  \begin{equation}\label{eq:scaled-F_N-remainder}
    \frac{\rho^{1/2}}{\lambda^{1/2}} \left(F_N(t) - \frac{S_2(t)}{N^2} - \frac{S_3(t)}{N^3} \right) - 1
  \end{equation}
  is analytic and conjugate symmetric. It hence suffices to bound it
  for \(t = e^{i\theta}\) with \(\theta \in [0, \pi]\). We make use of
  the algorithm discussed in Appendix~\ref{sec:enclosing-extrema} for
  enclosing the maximum on this interval. To evaluate the function we
  make use of Taylor models. Taylor models are discussed in
  Appendix~\ref{sec:taylor-models}, which also contains more details
  about how to compute the function in question using them. This gives
  us the enclosure
  \begin{equation*}
    \max_{\theta \in [0, \pi]} \left|\left(\frac{\rho^{1/2}}{\lambda^{1/2}} \left(F_N(t) - \frac{S_2(t)}{N^2} - \frac{S_3(t)}{N^3} \right) - 1\right)N^{4}\right|
    \in \resultnumber{[10 \pm 0.298]},
  \end{equation*}
  which implies the required bound.

  Finally, what remains is establishing the bound for \(J_{0}^{(8)}\).
  We start by noticing that the derivative can be explicitly computed
  to be
  \begin{equation*}
    J_{0}^{(8)}(x) = \frac{35J_{0}(x) - 56J_{2}(x) + 28J_{4}(x) - 8J_{6}(x) + J_{8}(x)}{128}.
  \end{equation*}
  To bound it on \([0, \infty)\) we split the interval into
  \([0, 50]\) and \((50, \infty)\). For \(x \in [0, 50]\) we make use
  of the algorithm discussed in Appendix~\ref{sec:enclosing-extrema}
  for enclosing the maximum, giving us:
  \begin{equation*}
    \max_{x \in [0, 50]} |J_{0}^{(8)}| \in \resultnumber{[0.2734375000 \pm 1.75 \cdot 10^{-11}]} < C_{J_{0},8}.
  \end{equation*}
  For \(x \in (50, \infty)\) we use the bound
  \begin{equation*}
    |J_{\nu}(x)| \leq 0.7858 x^{-1 / 3},
  \end{equation*}
  from~\cite{fungrim-7f3485}, giving
  \begin{equation*}
    |J_{0}^{(8)}(x)| \leq 0.7858x^{-1 / 3}
    \leq 0.7858 \cdot 50^{-1 / 3}
    = 0.213298... < C_{J_{0},8}.
  \end{equation*}
  %In reality \(J_{0}^{(8)}\) achieves its maximum at \(x = 0\), where it has the value \(56/128\), proving this does however require a little bit more work and the above bounds are good enough for our purposes in either case.
\end{proof}

With these preliminary constants established, we can now analyze the structure of $K_4(N,z,t)$. Let us denote by $l_0(N,z,t), l_1(N,z,t), l_2(N,z,t)$ and $l_3(N,z,t)$ the terms arising from the polynomial in $\lambda$ from $\sum_{j=0}^3 \frac{d_j(z,t)}{N^j}$, i.e.
\begin{align*}
    l_0(N,z,t) &= 1 \\
    l_1(N,z,t) &= \frac{\lambda  \log(\frac{t}{z}) }{4 N} + \frac{1}{8 N^2} \lambda \left( \log^2\left(\frac{t}{z}\right) +2 S_2(t) - 2S_2(z) \right) \\
        & + \frac{1}{24 N^3} \lambda \left( \log^3\left(\frac{t}{z}\right) + 6 \log\left(\frac{t}{z}\right) S_2(t) + 6 \log\left(\frac{t}{z}\right)S_2(\bar{z}) + 6S_3(t) - 6 S_3(z) \right) \\
    l_2(N,z,t) &= \frac{1}{64 N^2}\lambda^2 \log^2\left(\frac{t}{z}\right) + \frac{1}{64 N^3} \lambda^2 \log\left(\frac{t}{z}\right) \left( \log^2\left(\frac{t}{z}\right) + 2 S_2(t) - 2 S_2(z) \right) \\
    l_3(N,z,t) &= \frac{\log^3\left(\frac{t}{z}\right)}{2304 N^3}
\end{align*}
Thus, we can rewrite $K_4(N,z,t) = N^4 \left[ K(z,t) - l_0(N,z,t) - \lambda l_1(N,z,t) - \lambda^2 l_2(N,z,t) - \lambda^3 l_3(N,z,t) \right]$,
and expand $K(z,t)$ using the Taylor series of the Bessel function at the origin as in the proof of Lemma~\ref{lemma:exact-expression-for-K-up-to-order_-3}.
This decomposition allows us to express a bound for the remainder $K_4(N,z,t)$ as a polynomial in $\log(t/z)$, as shown in the following lemma.

\begin{lemma}
    \label{lemma:K4-bounds}
    Let $|z|=1$ and $t=az$ with $0<a<1$.
    The residue function $K_4(N,z,t) = N^4 [ K(z,t) - d_0(z,t) - \frac{d_1(z,t)}{N} - \frac{d_2(z,t)}{N^2} - \frac{d_3(z,t)}{N^3} ] = N^4 \left[ K(z,t) - l_0(N,z,t) - \lambda l_1(N,z,t) - \lambda^2 l_2(N,z,t) - \lambda^3 l_3(N,z,t) \right]$ can be bounded on $|z|=1$ as:
    \begin{align*}
    |K_4(N,z,t)|
        &= N^4 \left| \sum_{n=0}^{3} \left[ \left(-\frac{\rho}{4}\right)^n \frac{F_N(\bar{z})^n (F_N(z) -(t/z)^{1/N} F_N(t) )^n}{(n!)^2}  - \lambda^n l_n(N,z,t) \right] \right. \\
        & \quad \left. + \frac{J_0^{(8)}(\xi)}{8!} \left(\rho  F_N(\bar{z}) \left(F_N(z) -\left(\frac{t}{z}\right)^{1/N} F_N(t) \right) \right)^{4} \right| \\
        &= N^4 \left| \sum_{n=0}^{3} T_{K,n}(N,z,t)  +  \frac{J_0^{(8)}(\xi)}{8!} \left(\rho  F_N(\bar{z}) \left(F_N(z) -\left(\frac{t}{z}\right)^{1/N} F_N(t) \right) \right)^{4}  \right| \\
        &\leq |T_{K,1}(N,z,t)| + |T_{K,2}(N,z,t)| + |T_{K,3}(N,z,t)| + \frac{|J_0^{(8)}(\xi)|}{8!} \left| \rho  F_N(\bar{z}) \left(F_N(z) -\left(\frac{t}{z}\right)^{1/N} F_N(t) \right) \right|^{4} N^4 \\
        &\leq |T_{K,1}(N,z,t)| + |T_{K,2}(N,z,t)| + |T_{K,3}(N,z,t)|
        + \frac{C_{J_0,8}}{8!}\lambda^{4}C_{F_{N},0}^4 \left(2\frac{C_{F_N,2}}{N} + C_{F_{N},0}\left|\log\left(\frac{t}{z}\right)\right|\right)^4  .
    \end{align*}
    Moreover, the bound can be expressed as a polynomial in $\log\left(\frac{t}{z}\right)$, i.e.
    \begin{align*}
        |K_4(N,z,t)|
        &\leq L_0(N) + L_1(N) \left|\log\left(\frac{t}{z}\right)\right| + L_2(N) \log\left(\frac{t}{z}\right)^2 + L_3(N)\left|\log\left(\frac{t}{z}\right)\right|^3 + L_4(N) \log\left(\frac{t}{z}\right)^4 + L_6(N) \log\left(\frac{t}{z}\right)^6
    \end{align*}
\end{lemma}
\begin{proof}
    We begin with a detailed expansion of $T_{K,n}(N,z,t) = N^4 \left[ \left(-\frac{\rho}{4}\right)^n \frac{F_N(\bar{z})^n(F_N(z) -(t/z)^{1/N} F_N(t) )^n}{(n!)^2}  - \lambda^n l_n(N,z,t) \right]$ for $n=1,2,3$ in order to show that the multiplying factor $N^{4}$ is canceled by the internal structure of $T_{K,n}(N,z,t)$, resulting in an expression which is $O(1)$ in $N$.

\begin{enumerate}
    \item \textbf{Analysis of $T_{K,1}(N,z,t)$:}
    \begin{align*}
        & T_{K,1}(N,z,t) = N^4 \left[ \left(-\frac{\rho}{4}\right) F_N(\bar{z}) \left( {F}_N(z) - \left(\frac{t}{z}\right)^{1/N} {F}_N(t) \right) - \frac{\lambda}{4N} \log\left(\frac{t}{z}\right) - \frac{\lambda}{4} \left( \frac{\log^2(t/z)}{2N^2} + \frac{S_2(t) - S_2(z)}{N^2} \right) \right. \\
        & \quad \left. - \frac{\lambda}{4N^3} \left( \frac{\log^3(t/z)}{6} + \log\left(\frac{t}{z}\right) S_2(t) + \log\left(\frac{t}{z}\right) S_2(\bar{z}) + S_3(t) - S_3(z) \right) \right] \\
        & = N^4 \left[ -\frac{1}{4} \rho^{1/2}{F}_N(\bar{z}) \left( \rho^{1/2}\left( {F}_N(z) - \frac{S_2(z)}{N^2} - \frac{S_3(z)}{N^3} \right) \right. \right. \\
        & \qquad \left. - \left( \left(\frac{t}{z}\right)^{1/N} - \frac{1}{N} \log\left(\frac{t}{z}\right) - \frac{1}{2N^2} \log^2\left(\frac{t}{z}\right) - \frac{1}{6N^3} \log^3\left(\frac{t}{z}\right) \right)    \rho^{1/2}\left( {F}_N(t) - \frac{S_2(t)}{N^2} - \frac{S_3(t)}{N^3} \right) \right)  \\
        & \quad+ \left( \frac{\rho^{1/2}}{\lambda^{1/2}}{F}_N(\bar{z}) - 1 \right) \frac{\lambda}{4N^2} \left( \frac{\rho^{1/2}}{\lambda^{1/2}}\left(-S_2(z) - \frac{S_3(z)}{N}\right) \right. \\
        & \quad \left. + \frac{\rho^{1/2}}{\lambda^{1/2}}\left( S_2(t) + \frac{S_3(t)}{N} \right) \left( \left(\frac{t}{z}\right)^{1/N} - \frac{\log\left(\frac{t}{z}\right)}{N}  - \frac{\log^2\left(\frac{t}{z}\right)}{2N^2}  - \frac{\log^3\left(\frac{t}{z}\right)}{6N^3}  \right) \right) \\
        & \quad + \frac{\lambda}{4N^2} \frac{\rho^{1/2}}{\lambda^{1/2}} \left( S_2(t) + \frac{S_3(t)}{N} \right) \left[ \left(\frac{t}{z}\right)^{1/N} - 1 - \frac{1}{N} \log\left(\frac{t}{z}\right) - \frac{1}{2N^2} \log^2\left(\frac{t}{z}\right) - \frac{1}{6N^3} \log^3\left(\frac{t}{z}\right) \right] \\
        & \quad + \left(\frac{\rho^{1/2}}{\lambda^{1/2}} - 1\right) \frac{\lambda}{4N^2} \left( S_2(t) + \frac{S_3(t)}{N} \right) + \left(\frac{\rho^{1/2}}{\lambda^{1/2}} - 1\right) \left(-S_2(z) - \frac{S_3(z)}{N}\right) \frac{\lambda}{4N^2} \\
        & \quad + \left( \frac{\rho}{\lambda}{F}_N(\bar{z}) {F}_N(t) - 1 \right) \frac{\lambda}{4} \left( \frac{1}{2N^2} \log^2\left(\frac{t}{z}\right) + \frac{1}{6N^3} \log^3\left(\frac{t}{z}\right) \right) \\
        & \quad + \left( \frac{\rho}{\lambda}\left( {F}_N(\bar{z}) - \frac{S_2(\bar{z})}{N^2} \right) \left( {F}_N(t) - \frac{S_2(t)}{N^2} \right) - 1 \right) \frac{\lambda}{4} \cdot \frac{1}{N} \log\left(\frac{t}{z}\right) \\
        & \quad + \left( \frac{\rho^{1/2}}{\lambda^{1/2}}{F}_N(\bar{z}) - 1 \right) \frac{\rho^{1/2}}{\lambda^{1/2}} \frac{S_2(t)}{N^2} \frac{\lambda}{4} \frac{1}{N} \log\left(\frac{t}{z}\right)
        + \left( \frac{\rho^{1/2}}{\lambda^{1/2}}{F}_N(t) - 1 \right) \frac{\rho^{1/2}}{\lambda^{1/2}} \frac{S_2(\bar{z})}{N^2} \frac{\lambda}{4} \frac{1}{N} \log\left(\frac{t}{z}\right) \\
        & \quad \left. - \frac{\rho}{\lambda} \frac{S_2(\bar{z}) S_2(t)}{N^4} \frac{\lambda}{4} \frac{1}{N} \log\left(\frac{t}{z}\right)
        + \left( \frac{\rho^{1/2}}{\lambda^{1/2}} - 1 \right) \frac{S_2(t) + S_2(\bar{z})}{N^2} \frac{\lambda}{4} \frac{1}{N} \log\left(\frac{t}{z}\right)  \right]
    \end{align*}
    This expression for $T_{K,1}(N,z,t)$ enables us to obtain bounds for this function only depending on $\log\left(\frac{t}{z}\right)$, using the bounds from Lemma~\ref{lemma:K4-bounds-preliminary-constants}.
    \begin{align*}
        |T_{K,1}(N,z,t)| &\leq \frac{\lambda}{4} C_{F_{N},0} \left( 2C_{F_N,4} + \frac{1}{24} \log\left(\frac{t}{z}\right)^4 \left( 1 + \frac{C_{F_N,4}}{N^4} \right) \right) \\
        & \quad + C_{F_N,2} \frac{\lambda}{4} \left( C_{S_2} + \frac{C_{S_3}}{N} + \left( C_{S_2} + \frac{C_{S_3}}{N} \right) \left(1 + \frac{1}{24N^4}\log\left(\frac{t}{z}\right)^4 \right) \right)  \\
        & \quad + \frac{\lambda}{4} \left( C_{S_2} + \frac{C_{S_3}}{N} \right) \left( \frac{1}{24N^2}\log\left(\frac{t}{z}\right)^4 \right)
        + \frac{\lambda}{2}\frac{C_{\rho,\lambda}}{N^3}\left( C_{S_2} + \frac{C_{S_3}}{N} \right)  \\
        &\quad + \left( \left( 1 + C_{F_{N},0} \right) C_{F_N,2} \right) \frac{\lambda}{4}  \left( \frac{1}{2}\log\left(\frac{t}{z}\right)^2 + \frac{1}{6N}\left|\log\left(\frac{t}{z}\right)\right|^3 \right) \\
        &\quad + \left( \left( 1 + C_{F_{N},0} + \frac{C_{S_2}}{N^2} \right) C_{F_N,3} \right) \frac{\lambda}{4} \left|\log\left(\frac{t}{z}\right)\right|
        + 2 C_{F_N,2} C_{S_2} \frac{\lambda}{4N} \left|\log\left(\frac{t}{z}\right)\right|
        \\
        & \quad + C_{S_2}^2 \frac{\lambda}{4N} \left|\log\left(\frac{t}{z}\right)\right|
        + \frac{\lambda}{2} C_{S_2} \frac{C_{\rho,\lambda}}{N^4} \left|\log\left(\frac{t}{z}\right)\right|  \\
        & = \lambda\left(T_{K,1,0}(N) + T_{K,1,1}(N) \left|\log\left(\frac{t}{z}\right)\right| + T_{K,1,2}(N) \log\left(\frac{t}{z}\right)^2 + T_{K,1,3}(N) \left|\log\left(\frac{t}{z}\right)\right|^3 \right. \\
        & \qquad \left. + T_{K,1,4}(N) \log\left(\frac{t}{z}\right)^4\right)
    \end{align*}
    with
    \begin{align*}
        T_{K,1,0}(N)
            &= \frac{1}{2} C_{F_{N},0}C_{F_N,4}
            + \frac{1}{2}C_{F_N,2}\left(C_{S_2} + \frac{C_{S_3}}{N}\right)
            + \frac{1}{2}\frac{C_{\rho,\lambda}}{N^3}\left( C_{S_2} + \frac{C_{S_3}}{N} \right),
            \\
        T_{K,1,1}(N)
            &= \frac{1}{4}\left(1 + C_{F_{N},0} + \frac{C_{S_2}}{N^2}\right)C_{F_N,3}
            + \frac{1}{2N}C_{F_N,2} C_{S_2}
            + \frac{1}{4N}C_{S_2}^2 + \frac{1}{2 N^4} C_{S_2} C_{\rho,\lambda},
            \\
        T_{K,1,2}(N) &= \frac{1}{8}(1 + C_{F_{N},0})C_{F_N,2},
            \\
        T_{K,1,3}(N)
            &= \frac{1}{24 N} (1 + C_{F_{N},0})C_{F_N,2},
            \\
        T_{K,1,4}(N)
            &= \frac{1}{96}C_{F_{N},0}\left(1 + \frac{C_{F_N,4}}{N^4}\right)
            + \frac{1}{96N^{2}}C_{F_N,0}\left(C_{S_2} + \frac{C_{S_3}}{N}\right).
    \end{align*}

    \item \textbf{Analysis of $T_{K,2}(N,z,t)$:}
    \begin{align*}
        & T_{K,2}(N,z,t)  = N^4 \left[ \left( -\frac{\rho}{4} \right)^2 \frac{{F}_N(\bar{z})^2}{4} \left( {F}_N(z) - \left(\frac{t}{z}\right)^{1/N} {F}_N(t) \right)^2 - \frac{\lambda^2 \log^2(t/z)}{64 N^2} - \frac{\lambda^2}{64} \frac{\log^3\left(\frac{t}{z}\right)}{N^3} \right. \\
        & \quad \left.   - \frac{\lambda^2}{32} \frac{\log\left(\frac{t}{z}\right)}{N^3} \left( S_2(t) - S_2(z) \right) \right] \\
        & = N^4 \left[ \frac{\rho}{16} \cdot \frac{{F}_N(\bar{z})^2}{4} \left( \rho^{1/2}{F}_N(z) - \left[ \left(\frac{t}{z}\right)^{1/N} - \frac{1}{N} \log\left(\frac{t}{z}\right) \right] \rho^{1/2}{F}_N(t) \right)^2 \right. \\
        & \quad - \frac{\rho^2}{32} {F}_N(\bar{z})^2 F_N(t) \frac{1}{N} \log\left(\frac{t}{z}\right) \left( {F}_N(z) - \frac{S_2(z)}{N^2} - \left[ \left(\frac{t}{z}\right)^{1/N} - \frac{1}{N} \log\left(\frac{t}{z}\right) - \frac{1}{2N^2} \log\left(\frac{t}{z}\right)^2  \right] \left( {F}_N(t) - \frac{S_2(t)}{N^2} \right) \right) \\
        & \quad + \frac{\lambda^2}{64 N^2} \left( \frac{\rho^2}{\lambda^2} F_N(\bar{z})^2 F_N(t)^2 - 1 \right) \log\left(\frac{t}{z}\right)^2  +  \frac{\lambda^2}{32} \left( \frac{\rho^2}{\lambda^2} F_N(\bar{z})^2 F_N(t)^2 - 1  \right) \frac{1}{2N^3} \log\left(\frac{t}{z}\right)^3 \\
        & \quad + \frac{\lambda^2}{32 N^3}  \log\left(\frac{t}{z}\right) \left( \frac{\rho^{3/2}}{\lambda^{3/2}} F_N(\bar{z})^2 F_N(t) - 1 \right) \frac{\rho^{1/2}}{\lambda^{1/2}} (-S_2(z))
        + \frac{\lambda^2}{32N^3} \left( \frac{\rho^{1/2}}{\lambda^{1/2}} - 1 \right) \log\left(\frac{t}{z}\right) \left( -S_2(z) \right)  \\
        & \quad +  \frac{\lambda^2}{32} \left(\frac{\rho^{3/2}}{\lambda^{3/2}} F_N(\bar{z})^2 F_N(t) - 1 \right) \frac{1}{N} \log\left(\frac{t}{z}\right) \frac{\rho^{1/2}}{\lambda^{1/2}} \frac{S_2(t)}{N^2} \left( \left(\frac{t}{z}\right)^{1/N}  - \frac{1}{N}\log\left(\frac{t}{z}\right) - \frac{1}{2N^2}\log\left(\frac{t}{z}\right)^2 \right)  \\
        & \quad \left.  + \frac{\lambda^2}{32N} \log\left(\frac{t}{z}\right) \frac{\rho^{1/2}}{\lambda^{1/2}} \frac{S_2(t)}{N^2} \left( \left(\frac{t}{z}\right)^{1/N}  - 1 - \frac{1}{N}\log\left(\frac{t}{z}\right) - \frac{1}{2N^2}\log\left(\frac{t}{z}\right) \right) + \frac{\lambda^2}{32N} \left( \frac{\rho^{1/2}}{\lambda^{1/2}} - 1 \right) \frac{S_2(t)}{N^2} \log\left(\frac{t}{z}\right) \right]
    \end{align*}

    As before, we apply Lemma~\ref{lemma:K4-bounds-preliminary-constants} to bound $T_{K,2}(N,z,t)$ using only powers of $\log\left(\frac{t}{z}\right)$. This shows that the function $T_{K,2}$ is $O(1)$ in $N$.

    \begin{align*}
        |T_{K,2}(N,z,t)| &\leq \frac{\lambda^2}{64} C_{F_{N},0}^2 \left( 2 C_{F_N,2} + \frac{1}{2}\log\left(\frac{t}{z}\right)^2 C_{F_{N},0} \right)^2 \\
        & \quad+ \frac{\lambda^2}{32} C_{F_{N},0}^3 \left|\log\left(\frac{t}{z}\right)\right| \left( 2 C_{F_N,3} + \frac{1}{6} \left|\log\left(\frac{t}{z}\right)\right|^3 \left( C_{F_{N},0} + \frac{C_{S_2}}{N^2} \right)  \right)  \\
        &\quad+ \frac{\lambda^2}{64} \left( C_{F_{N},0}^2 + 1 \right) \left(1 + C_{F_{N},0}\right) C_{F_N,2} \log\left(\frac{t}{z}\right)^2
        + \frac{\lambda^2}{32} \left( C_{F_{N},0}^2 + 1 \right) \left(1+C_{F_{N},0}\right) \frac{C_{F_N,2}}{N} \frac{1}{2} \log\left(\frac{t}{z}\right)^3 \\
        & \quad+ \frac{\lambda^2}{32} \left|\log\left(\frac{t}{z}\right)\right| \frac{C_{F_N,2}}{N} \left( 1+C_{F_N,0}+C_{F_N,0}^2 \right) C_{S_2} + \frac{\lambda^2}{32} \frac{C_{\rho,\lambda}}{N^4} \left|\log\left(\frac{t}{z}\right)\right| C_{S_2}  \\
        &\quad+ \frac{\lambda^2}{32 N} \left|\log\left(\frac{t}{z}\right)\right|  \left( C_{F_{N},0} \left( 1 + C_{F_{N},0} \right) + 1\right) C_{F_N,2} C_{S_2} \left( 1 + \frac{1}{6N^3} \left|\log\left(\frac{t}{z}\right)\right|^3 \right)  \\
        & \quad+ \frac{\lambda^2}{32} \left|\log\left(\frac{t}{z}\right)\right| C_{S_2} \frac{1}{6N^2} \left|\log\left(\frac{t}{z}\right)\right|^3
        \quad+ \frac{\lambda^2}{32} \frac{C_{\rho,\lambda}}{N^4} C_{S_2} \left|\log\left(\frac{t}{z}\right)\right|
        \\
        & = \lambda^{2}\left(T_{K,2,0}(N) + T_{K,2,1}(N) \left|\log\left(\frac{t}{z}\right)\right| + T_{K,2,2}(N) \log\left(\frac{t}{z}\right)^2 + T_{K,2,3}(N) \left|\log\left(\frac{t}{z}\right)\right|^3 \right. \\
        & \qquad\quad \left. + T_{K,2,4}(N) \log\left(\frac{t}{z}\right)^4\right)
    \end{align*}
    with
    \begin{align*}
      T_{K,2,0}(N)
      &= \frac{1}{16} C_{F_{N},0}^2C_{F_N,2}^2
      \\
      T_{K,2,1}(N)
      &= \frac{1}{16} C_{F_{N},0}^3C_{F_N,3} + \frac{1}{16N^{4}}C_{\rho,\lambda}C_{S_2}
        + \frac{1}{16N} \left(1 + C_{F_{N},0}\left(1 + C_{F_{N},0}\right)\right)C_{F_N,2}C_{S_2}
      \\
      T_{K,2,2}(N)
      &= \frac{1}{32}C_{F_{N},0}^3 C_{F_N,2}
        + \frac{1}{64} (1 + C_{F_{N},0}^2)(1 + C_{F_{N},0})C_{F_N,2}
      \\
      T_{K,2,3}(N)
      &= \frac{1}{64N}(1 + C_{F_{N},0}^2)(1 + C_{F_{N},0})C_{F_N,2}
      \\
      T_{K,2,4}(N)
      &= \frac{1}{256} C_{F_{N},0}^4
        + \frac{1}{192}C_{F_{N},0}^3\left(C_{F_{N},0} + \frac{C_{S_2}}{N^2}\right)
        + \frac{1}{192N^4}(1 + C_{F_{N},0}(1 + C_{F_{N},0}))C_{F_N,2}C_{S_2}
        + \frac{1}{192N^2}C_{S_2}
    \end{align*}

        \item \textbf{Analysis of $T_{K,3}(N,z,t)$:}

    \begin{align*}
        T_{K,3}(N,z,t)
        &= N^4 \left[ \left( -\frac{\rho^3}{64} \right) \frac{{F}_N(\bar{z})^3}{36} \left( {F}_N(z) - \left(\frac{t}{z}\right)^{1/N} {F}_N(t) \right)^3 - \frac{\lambda^3 \log^3(t/z)}{64 \cdot 36 N^3} \right] \\
        & = N^4 \left[ -\frac{1}{64} \cdot \rho^{3/2}\frac{{F}_N(\bar{z})^3}{36} \left( \rho^{1/2}{F}_N(z) - \left[ \left(\frac{t}{z}\right)^{1/N} - \frac{1}{N} \log\left(\frac{t}{z}\right) \right] \rho^{1/2}{F}_N(t) \right)^3 \right. \\
        & \quad+ \left( \frac{\rho^3}{\lambda^3}{F}_N(\bar{z})^3 {F}_N(t)^3 - 1 \right) \frac{\lambda^3 \log^3(t/z)}{64 \cdot 36 \cdot N^3} \\
        & \quad+ 3 \left( \frac{\rho^{3/2} {F}_N(\bar{z})^3}{64 \cdot 36} \right) \left( \rho^{1/2}{F}_N(z) - \left(\frac{t}{z}\right)^{1/N} \rho^{1/2}{F}_N(t) \right) \frac{1}{N} \log\left(\frac{t}{z}\right) \rho^{1/2}{F}_N(t)   \\
        & \qquad \left. \left( \rho^{1/2}{F}_N(z) - \left[ \left(\frac{t}{z}\right)^{1/N} - \frac{\log\left(\frac{t}{z}\right)}{N}  \right] \rho^{1/2}{F}_N(t) \right) \right] \\
    \end{align*}

    Similar as before, using the bounds from Lemma~\ref{lemma:K4-bounds-preliminary-constants}, $T_{K,3}(N,z,t)$ is bounded by powers of $\log\left(\frac{t}{z}\right)$, giving an expression which is $O(1)$ in $N$.

    \begin{align*}
        |T_{K,3}(N,z,t)| &\leq \frac{\lambda^3}{2304} C_{F_{N},0}^3 \left( 2 C_{F_N,2} + \frac{1}{2}\log\left(\frac{t}{z}\right)^2 C_{F_{N},0} \right)^3 \frac{1}{N^2} \\
        &\quad + \frac{\lambda^3}{2304 N} \left|\log\left(\frac{t}{z}\right)\right|^3 \left( C_{F_{N},0}^3 +1 \right) \left( C_{F_{N},0}^2 +  \left( 1 + C_{F_{N},0} \right)  \right) C_{F_N,2} \\
        &\quad + 3\frac{\lambda^3}{2304}\left|\log\left(\frac{t}{z}\right)\right| C_{F_{N},0}^4 \left( 2 \frac{C_{F_N,2}}{N} + \left|\log\left(\frac{t}{z}\right)\right| C_{F_{N},0} \right)  \left( 2 C_{F_N,2} + \frac{1}{2}\log\left(\frac{t}{z}\right)^2 C_{F_{N},0} \right) \\
        & = \lambda^{3}\Bigg(T_{K,3,0}(N) + T_{K,3,1}(N) \left|\log\left(\frac{t}{z}\right)\right| + T_{K,3,2}(N) \log\left(\frac{t}{z}\right)^2 + T_{K,3,3}(N) \left|\log\left(\frac{t}{z}\right)\right|^3  \\
        &\qquad+ T_{K,3,4}(N) \log\left(\frac{t}{z}\right)^4
        + T_{K,3,6}(N) \log\left(\frac{t}{z}\right)^6\Bigg)
    \end{align*}
    with
    \begin{align*}
      T_{K,3,0}(N)
      &= \frac{1}{288N^{2}}C_{F_{N},0}^3C_{F_N,2}^3,
      \\
      T_{K,3,1}(N)
      &= \frac{1}{192}C_{F_{N},0}^4C_{F_N,2}^2,
      \\
      T_{K,3,2}(N)
      &= \frac{1}{384N^{2}}C_{F_{N},0}^4C_{F_N,2}^2
        + \frac{1}{384}C_{F_{N},0}^5C_{F_N,2},
      \\
      T_{K,3,3}(N)
      &= \frac{1}{2304N}(1 + C_{F_{N},0}^3)(1 + C_{F_{N},0} + C_{F_{N},0}^2)C_{F_N,2}
        + \frac{1}{768N}C_{F_{N},0}^5C_{F_{N},2},
      \\
      T_{K,3,4}(N)
      &= \frac{1}{1536N^{2}} C_{F_{N},0}^5C_{F_N,2}
        + \frac{1}{1536}C_{F_{N},0}^6,
      \\
      T_{K,3,6}(N)
      &= \frac{1}{18432N^{2}}C_{F_{N},0}^6.
    \end{align*}
\end{enumerate}
    Combining the bounds for $T_{K,1}(N,z,t), T_{K,2}(N,z,t), T_{K,3}(N,z,t)$, which are already polynomials in $\log\left(\frac{t}{z}\right)$, with the remaining error term that can also be written in terms of only $\log\left(\frac{t}{z}\right)$, we obtain that
    \begin{align*}
        |K_4(N,z,t)| \leq L_0(N) + L_1(N) \left|\log\left(\frac{t}{z}\right)\right| + L_2(N) \log\left(\frac{t}{z}\right)^2 + L_3(N)\left|\log\left(\frac{t}{z}\right)\right|^3 + L_4(N) \log\left(\frac{t}{z}\right)^4 + L_6(N) \log\left(\frac{t}{z}\right)^6 ,
    \end{align*}
    with
    \begin{align*}
      L_0(N)
      &= \lambda T_{K,1,0}(N) + \lambda^{2}T_{K,2,0}(N) + \lambda^{3} T_{K,3,0}(N) + \lambda^{4} \frac{16}{N^{4}}\frac{C_{J_0,8}}{8!}C_{F_{N},0}^4C_{F_N,2}^4,
      \\
      L_1(N)
      &= \lambda T_{K,1,1}(N) + \lambda^{2}T_{K,2,1}(N) + \lambda^{3} T_{K,3,1}(N) + \lambda^{4} \frac{32}{N^{3}}\frac{C_{J_0,8}}{8!}C_{F_{N},0}^5C_{F_N,2}^3,
      \\
      L_2(N)
      &= \lambda T_{K,1,2}(N) + \lambda^{2}T_{K,2,2}(N) + \lambda^{3} T_{K,3,2}(N) + \lambda^{4} \frac{24}{N^{2}}\frac{C_{J_0,8}}{8!}C_{F_{N},0}^6C_{F_N,2}^2,
      \\
      L_3(N)
      &= \lambda T_{K,1,3}(N) + \lambda^{2}T_{K,2,3}(N) + \lambda^{3} T_{K,3,3}(N) + \lambda^{4} \frac{8}{N}\frac{C_{J_0,8}}{8!}C_{F_{N},0}^7C_{F_N,2},
      \\
      L_4(N)
      &= \lambda T_{K,1,4}(N) + \lambda^{2} T_{K,2,4}(N) + \lambda^{3} T_{K,3,4}(N) + \lambda^{4} \frac{C_{J_0,8}}{8!}C_{F_{N},0}^8,
      \\
      L_6(N)
      &= \lambda^{3} T_{K,3,6}(N).
      \\
    \end{align*}
\end{proof}

Finally, the next lemma provides explicit bounds on the constants $L_j(N)$ appearing in the polynomial bound in $\log\left(\frac{t}{z}\right)$ for $K_4(N,z,t)$ stated in Lemma~\ref{lemma:K4-bounds}.

\begin{lemma}\label{lemma:K4-bounds-constants}
  For \(N \geq N_{0} = \codenumber{64}\) the constants \(L_{j}(N)\) in
  Lemma~\ref{lemma:K4-bounds} are bounded by \(L_{j}(N) \leq C_{L_{j}}\)
  with
  \begin{equation*}
    C_{L_{0}} = \codenumber{92},\quad
    C_{L_{1}} = \codenumber{18},\quad
    C_{L_{2}} = \codenumber{18},\quad
    C_{L_{3}} = \codenumber{2},\quad
    C_{L_{4}} = \codenumber{0.6},\quad
    C_{L_{6}} = \codenumber{3 \cdot 10^{-6}}.
  \end{equation*}
\end{lemma}

\begin{proof}
  We use the expression for the constants \(L_{j}(N)\) given in the proof
  of Lemma~\ref{lemma:K4-bounds}. Apart from the constants from
  Lemma~\ref{lemma:K4-bounds-preliminary-constants} we only need a
  bound for \(\lambda\), which is easily computed. For all divisions
  by \(N\) that occur in the expressions we replace them with division
  by \(N_{0}\), which still gives an upper bound. The computed values
  are then upper bounded by the \(C_{L_{j}}\) constants.
\end{proof}

\printbibliography

\begin{tabular}{l}
  \textbf{Joel Dahne} \\
  {School of Mathematics} \\
  {University of Minnesota} \\
  {127 Vincent Hall, 206 Church St. SE} \\
  {Minneapolis, MN 55455, USA} \\
  {Email: jdahne@umn.edu}  \\ \\
  \textbf{Javier G\'omez-Serrano}\\
  {Department of Mathematics} \\
  {Brown University} \\
  {314 Kassar House, 151 Thayer St.} \\
  {Providence, RI 02912, USA} \\ \\

{\textit{and}} \\ \\

{Institute for Advanced Study}\\
{1 Einstein Drive}\\
{Princeton, NJ 08540, USA} \\
  {Email: javier\_gomez\_serrano@brown.edu}
  \\ \\
  \textbf{Joana Pech-Alberich} \\
  {Department of Mathematics} \\
  {Universitat Politècnica de Catalunya} \\
  {Carrer Pau Gargallo, 14} \\
  {Barcelona, 08030, Spain} \\ \\

{\textit{and}} \\ \\

  {Department of Mathematics} \\
  {Brown University} \\
  {010 Kassar House, 151 Thayer St.} \\
  {Providence, RI 02912, USA} \\
  {Email: joana.pech@estudiantat.upc.edu, joana\_pech\_alberich@brown.edu}
\end{tabular}
\end{document}